\documentclass[10pt]{amsart}
\usepackage{amsmath, amssymb, amsthm}
\usepackage{mathrsfs}
\usepackage{algorithm}
\usepackage{algorithmic}
\usepackage{tikz}
\usetikzlibrary{positioning, arrows.meta, fit}
\usepackage{pgfplots}
\pgfplotsset{compat=1.18}
\usepackage{mathtools}
\usepackage{booktabs}
\usepackage[colorlinks=true, linkcolor=blue, citecolor=red, urlcolor=blue]{hyperref}
\usepackage{xcolor}
\usepackage{float}
\usepackage{enumerate}
\usepackage{url}
\usepackage[a4paper, margin=2.12cm]{geometry}
\usepackage{doi}

\newtheorem{theorem}{Theorem}[section]
\newtheorem{lemma}{Lemma}[section]
\newtheorem{proposition}{Proposition}[section]
\newtheorem{corollary}{Corollary}[section]
\theoremstyle{definition}
\newtheorem{definition}{Definition}[section]
\newtheorem{remark}{Remark}[section]
\newtheorem{assumption}{Assumption}[section]

\newcommand{\HH}{\mathbb{H}}
\newcommand{\CC}{\mathbb{C}}
\newcommand{\RR}{\mathbb{R}}
\newcommand{\LL}{\mathcal{L}}
\newcommand{\dd}{\mathrm{d}}
\newcommand{\innerR}[2]{\left\langle #1,\, #2 \right\rangle_{\mathbb{R}}}
\newcommand{\trans}{^{\dagger}}
\newcommand{\Real}{\mathrm{Re}}
\newcommand{\epsStar}{\varepsilon^{*}}
\newcommand{\Cnc}{C_{\mathrm{nc}}}
\newcommand{\deltabar}{\bar{\delta}}
\newcommand{\sgnH}{\mathrm{sgn}_{\HH}}
\newcommand{\Phiconj}{\Phi_{\mathrm{conj}}}
\newcommand{\cBplus}{c_{B^{+}}}

\begin{document}

\title[Operator-Norm Transfer and Cohomological Rigidity]{Operator-Norm Transfer and Cohomological Rigidity for Quaternionic Quasi-Lie Structures \\ with Application to Sliding Mode $\beta$-Exponential Stability}

\author{Nassim Athmouni}
\address{Universit\'e de Gafsa, Campus Universitaire 2112, Gafsa, Tunisie}
\email{nassim.athmouni@fsgf.u-gafsa.tn}
\author{Nejib Brahmia}
\address{Universit\'e de Gafsa, Campus Universitaire 2112, Gafsa, Tunisie}
\author{Tarek Fajraoui}
\address{Universit\'e de Gafsa, Campus Universitaire 2112, Gafsa, Tunisie}
\author{Fehmi Mabrouk}
\address{Universit\'e de Gafsa, Campus Universitaire 2112, Gafsa, Tunisie}


\begin{abstract}
We develop operator-theoretic and cohomological tools for quaternionic
quasi-Lie structures, with sliding mode control as a motivating application.
Three main results are established. First, an exact operator-norm transfer
under the quaternionic anti-isomorphism $a \mapsto \bar{a}$, which enables
quantitative bounds from~\cite{Athmouni2026} to transfer between left- and
right-module conventions with no constant factor. Second, a transcription of
the cohomological rigidity result of~\cite{Athmouni2026} into a form usable
here: in the homogeneous case, under a local cohomological non-obstruction
hypothesis, an explicit bilinear correction $\Omega$ produces a bracket
satisfying the Jacobi identity exactly on a ball of admissible radius, with
all quantitative constants expressed through $C_{2}$ and the admissible
radius. Third, the projected Jacobi defect is shown to satisfy a generalized
one-sided Lipschitz condition with computable constants, obtained via a
uniform-selection argument handling state-dependence of the measurable
selection.
As an application, we develop a robust control framework with a cohomological
matching condition replacing pointwise verification: an integral sliding
surface yields $\beta$-exponential stability via an iterative linear matrix
inequality (LMI) scheme.
The work is purely analytical; closed-loop numerical simulations for
multidimensional systems are deferred to a companion paper. The framework is
restricted to the homogeneous quasi-Lie case; the Sobolev extension is
conjectural, and algorithm termination is established conditionally on
sufficient continuity assumptions on the LMI solution map.
\end{abstract}

\maketitle
\tableofcontents

\section{Introduction}\label{sec:intro}

Working with quaternionic systems forces a choice of convention. The state
space $\HH^{n}$ carries both a left- and a right-$\HH$-module structure,
and operator bounds proved under one convention do not, in general,
translate verbatim to the other --- conjugation $a \mapsto \bar a$ is
anti-multiplicative, not multiplicative. The recent quantitative theory
of quasi-Lie brackets on quaternionic Banach modules~\cite{Athmouni2026},
on which we rely heavily, is set up for right modules. Applications in
attitude dynamics, where rotation quaternions act on the left, naturally
favor the opposite convention. Reconciling the two requires care.

The first half of this paper supplies the missing operator-theoretic
ingredient. Given a bounded right-$\HH$-linear operator
$T : \HH^{n}\to\HH^{n}$, we consider its conjugate
$T^{L}:=\Phiconj\circ T\circ\Phiconj$, acting on the left-module
structure via componentwise conjugation $\Phiconj$. Because $\Phiconj$ is
an isometric involution, one expects $\|T^{L}\|_{\mathrm{op}}$ to be
comparable to $\|T\|_{\mathrm{op}}$; the question is whether equality
holds, or only a bounded ratio. Theorem~\ref{thm:norm_transfer} shows
that equality holds. The proof is short and uses only the isometric
property of $\Phiconj$, but the consequence is what we need:
\emph{every quantitative bound of~\cite{Athmouni2026} transfers to the
left-module setting without a constant factor.}

A separate question concerns the Jacobi defect itself. Given a bilinear
bracket $\LL : \HH^{n}\times\HH^{n}\to\HH^{n}$, write
$\mathcal{J}(x,y,z)$ for its Jacobi defect and assume
$\|\mathcal{J}(x,y,z)\|\leq 6C_{2}\|x\|\|y\|\|z\|$. The rigidity result
of~\cite[Theorem~4.12]{Athmouni2026} produces, in the homogeneous case,
a bilinear correction $\Omega$ such that the corrected bracket
$\{x,y\} := \LL(x,y) - \Omega(x,y)$ satisfies the Jacobi identity
\emph{exactly} on a ball of admissible radius
$\epsStar = \min\{1/(16A),\sqrt{1/(16C_{1})},1/(4C_{2}),\varepsilon_{0}\}$,
with $\|\Omega\|_{\varepsilon}\leq 4C_{2}\varepsilon$. We restate this
result (Theorem~\ref{thm:rigidity}) in a form adapted to the control
application that follows, and isolate the constants that will reappear
throughout.

Why control? The Jacobi defect arises in linearized attitude
dynamics~\cite{Shuster1993,Wie1985}, and a system with non-commutative
uncertainties is exactly the setting where existing sliding-mode
techniques run into trouble. The standard matching condition
$\delta(x,t) \in \mathrm{Im}(B)$ is checked trajectory by trajectory;
the switching gain $\eta$ must dominate $\|\delta\|_{\infty}$ uniformly,
producing chattering. We use the cohomological correction $\Omega$ to
build an explicit feedforward term
$\Delta_{\mathrm{app}}(x,t) := B^{+}\,\Pi_{\mathrm{Im}(B)}\,d\omega_{0}(\ldots)$
that pre-cancels most of $\delta$, leaving only the non-commutative
residual $R_{\mathrm{nc}}^{B}$ for $\eta$ to absorb. The required gain
drops from $O(\|\delta\|_{\infty})$ to $O(\Cnc R_{\max})$, an improvement
that is meaningful precisely when the bracket is close to satisfying
Jacobi.

Two further ingredients make the construction usable. The projected
Jacobi defect $\delta(x,t) = \Pi_{\mathrm{Im}(B)}\mathcal{J}_{t}(x,\cdot,\cdot)$
involves a measurable selection $\xi^{*}(x,t)$ that depends on the
state, and naive estimates fail to give one-sided Lipschitz constants.
Lemma~\ref{lem:osl_projection} circumvents this difficulty via a
\emph{frozen-selection trick}: rather than tracking the optimizer, we
lift to a uniform supremum over both selection arguments and obtain
explicit constants
$(\rho_{\delta}^{\mathrm{tight}},\ell_{\delta}^{\mathrm{tight}})
 = (2\rho_{\delta}, 4\deltabar_{\max})$.
The factors of $2$ and $4$ reflect the suboptimality gap of a
state-dependent selection; they collapse to $1$ when a state-independent
maximizer exists.

The pointwise condition $\delta(x,t) \in \mathrm{Im}(B)$ is replaced by
a structural condition on $(B,\LL)$ alone --- the
\emph{Cohomological Matching Condition} (Definition~\ref{def:cmc}),
$\mathrm{im}_{\mathrm{eval}}(d) \subseteq \mathrm{Im}(B)$. This is
verified once, at design time, and covers every trajectory of the
system. The remaining algorithmic problem is the well-known circular
dependence between the gain $\eta$ (which depends on the invariant set
size) and the invariant set itself (which depends on $\eta$).
Algorithm~\ref{alg:init} resolves the circularity by an explicit
ordering of computations and a halving loop; termination is established
conditionally (Lemma~\ref{lem:Cinfty_scaling}) under continuity of the
LMI solution map. The LMI in Algorithm~\ref{alg:lmi} uses the
coefficient
$\mu_{k}(2\rho_{\delta}^{\mathrm{eff}}+\mu_{\mathrm{Y}})
 + (\rho_{\delta}^{\mathrm{eff}}-\mu_{\mathrm{Y}})$
with explicit separation of the LMI scaling parameter $\mu_{k}$ and the
Young parameter $\mu_{\mathrm{Y}}$ (Theorem~\ref{thm:beta_kyb}).

The sliding-mode machinery itself --- integral sliding variable,
super-twisting reaching, $\beta$-exponential stability via a quadratic
Lyapunov function --- is classical~\cite{Utkin1992,Edwards1998,Khalil2002}.
What is new is the choice of inputs supplied to this machinery: a
feedforward built from the cohomological correction, a matching
condition checked algebraically, and a halving loop that reconciles
gain and domain. The OSL/GOSL framework
of~\cite{Abbaszadeh2010,Zhang2015} extends naturally to accommodate the
affine Jacobi defect term, and Lemma~\ref{lem:osl_projection} provides
the explicit constants.

A word on scope. All analytical results are restricted to the
\emph{homogeneous case} (Assumption~\ref{ass:cohom}), under which the
cohomological structure of~\cite{Athmouni2026} applies without further
hypothesis. The Sobolev extension is conjectural and the proof of
forward invariance is complete only when $P^{*} = I$; the general case
$\kappa(P^{*}) > 1$ remains open and is discussed in
Section~\ref{sec:conclusion}. Numerical simulations for multidimensional
closed-loop systems are deferred to a companion paper. The
operator-theoretic results (Theorem~\ref{thm:norm_transfer},
Theorem~\ref{thm:rigidity}, Lemma~\ref{lem:osl_projection}) are usable
independently of the control application; readers interested only in
the algebraic side may stop after Section~\ref{sec:prelim}.

\medskip
\noindent\textbf{Organization.}
Section~\ref{sec:prelim} collects the algebraic prerequisites, the real
representation $\Phi:\HH^{n}\to\RR^{4n}$ (a derived form of the complex
adjoint construction of~\cite{Zhang1997}), and the
three core operator-theoretic results.
Section~\ref{sec:problem} sets up the control problem and defines the
non-commutative residual $R_{\mathrm{nc}}^{B}$.
Section~\ref{sec:cohomology} develops the cohomological matching
condition. Section~\ref{sec:design} designs the integral sliding
surface and proves finite-time reaching
(Theorem~\ref{thm:control}). Section~\ref{sec:stability} establishes
$\beta$-exponential stability (Proposition~\ref{prop:beta_stability}).
Section~\ref{sec:sim} computes the constants on a small test bracket.
Section~\ref{sec:conclusion} discusses limitations and open
directions. Appendix~\ref{app:beta} contains a self-contained proof
of $\beta$-exponential stability under dissipative perturbations
(Theorem~\ref{thm:beta_kyb}, applicable to the GOSL setting of
Lemma~\ref{lem:osl_projection}); Appendix~\ref{app:bracket_check} verifies
the test-bracket constants by explicit computation.

\section{Preliminaries and Main Operator-Theoretic Results}\label{sec:prelim}

We begin with the algebraic facts needed throughout, then prove the three
operator-theoretic results announced in the introduction. The first two
results (Theorem~\ref{thm:norm_transfer} on norm transfer and
Theorem~\ref{thm:rigidity} restating the rigidity of~\cite{Athmouni2026})
can be read independently of the control application. The third
(Lemma~\ref{lem:osl_projection}, the GOSL bound on the projected Jacobi
defect) is more technical and motivated by what follows.

\subsection{Quaternionic Algebra}

The skew field $\HH$ consists of elements
$q = q_{0} + q_{1}\mathbf{i} + q_{2}\mathbf{j} + q_{3}\mathbf{k}$ with
$q_{0},\ldots,q_{3} \in \RR$ and
$\mathbf{i}^{2} = \mathbf{j}^{2} = \mathbf{k}^{2} = \mathbf{ijk} = -1$.
The conjugate $\bar{q} = q_{0} - q_{1}\mathbf{i} - q_{2}\mathbf{j} - q_{3}\mathbf{k}$
satisfies $\|q\|^{2} = q\bar{q} = \bar{q}q$ and
$\overline{pq} = \bar{q}\bar{p}$ (anti-multiplicativity).
Multiplication satisfies $pq \neq qp$ in general.

\subsection{Module Conventions and the Anti-isomorphism}

\begin{definition}[Left and Right Quaternionic Modules]\label{def:modules}
$\HH^{n}$ is a \emph{left} $\HH$-module when scalar multiplication is
$(a, x) \mapsto ax$ (component-wise left multiplication). It is a \emph{right}
$\HH$-module when $(x, a) \mapsto xa$. The two structures are related by the
anti-isomorphism $a \mapsto \bar{a}$: a right module becomes a left module by
redefining $a \cdot_L x := x \cdot_R \bar{a}$, and vice versa.
\end{definition}

\begin{remark}[Convention for this Paper]\label{rem:convention}
The state space $\HH^{n}$ carries both left- and right-$\HH$-module
structures. The bracket $\LL$ and the bilinear correction $\Omega$
of~\cite{Athmouni2026} are presented for right modules in the cited
reference; we adopt the same right-module convention
for $\LL$ and $\Omega$. Matrix actions $u\mapsto Bu$ for
$B\in\HH^{n\times m}$, $u\in\HH^{m}$ are the standard ones, compatible
with right-scalar multiplication ($B(uq)=(Bu)q$ for all $q\in\HH$ by
associativity of $\HH$).
Theorem~\ref{thm:norm_transfer} provides the tool to transfer quantitative
bounds to the left-module convention, should this be needed for specific
attitude-dynamics applications where rotation quaternions act on the left.
All stability analysis is performed in the real representation $\RR^{4n}$
via the map $\Phi$ of Remark~\ref{rem:real_rep}, which is
convention-independent.
\end{remark}

\begin{definition}[Componentwise Conjugation]\label{def:phi_conj}
Define $\Phiconj : \HH^{n} \to \HH^{n}$ by
$\Phiconj(x)_{i} = \bar{x}_{i}$ for $i = 1,\ldots,n$.
\end{definition}

\begin{lemma}[Properties of $\Phiconj$]\label{lem:phi_conj_props}
\begin{enumerate}
\item[\textup{(i)}] Isometry: $\|\Phiconj(x)\| = \|x\|$ for all $x \in \HH^{n}$.
\item[\textup{(ii)}] Involution: $\Phiconj \circ \Phiconj = \mathrm{id}$,
hence $\Phiconj$ is bijective.
\item[\textup{(iii)}] $\Phiconj(x \cdot a) = \bar{a} \cdot \Phiconj(x)$
for all $x \in \HH^{n}$, $a \in \HH$.
\item[\textup{(iv)}] $\Phiconj(a \cdot x) = \Phiconj(x) \cdot \bar{a}$
for all $x \in \HH^{n}$, $a \in \HH$.
\end{enumerate}
\end{lemma}

\begin{proof}
\textup{(i)} $\|\Phiconj(x)\|^{2} = \sum_{i}\|\bar{x}_{i}\|^{2} = \sum_{i}\|x_{i}\|^{2}$.
\textup{(ii)} $\Phiconj(\Phiconj(x))_{i} = \bar{\bar{x}}_{i} = x_{i}$.
\textup{(iii)} $\Phiconj(x \cdot a)_{i} = \overline{x_{i}a} = \bar{a}\bar{x}_{i}$
by $\overline{pq} = \bar{q}\bar{p}$.
\textup{(iv)} $\Phiconj(a \cdot x)_{i} = \overline{ax_{i}} = \bar{x}_{i}\bar{a}$.
\end{proof}

\subsection{Norm Transfer under the Anti-Isomorphism}

\begin{theorem}[Operator Norm Transfer under the Anti-isomorphism]
\label{thm:norm_transfer}
Let $T : \HH^{n} \to \HH^{n}$ be bounded and right-$\HH$-linear:
$T(x \cdot a) = T(x) \cdot a$ for all $x \in \HH^{n}$, $a \in \HH$.
Define $T^{L} := \Phiconj \circ T \circ \Phiconj$. Then:
\begin{enumerate}
\item[\textup{(i)}] $T^{L}$ is left-$\HH$-linear:
$T^{L}(a \cdot y) = a \cdot T^{L}(y)$ for all $y \in \HH^{n}$, $a \in \HH$.
\item[\textup{(ii)}] $\|T^{L}\|_{\mathrm{op}} = \|T\|_{\mathrm{op}}$ exactly.
\end{enumerate}
\end{theorem}

\begin{proof}
\textbf{Part~(i).}
\begin{align}
T^{L}(a \cdot y)
&= \Phiconj\bigl(T\bigl(\Phiconj(y) \cdot \bar{a}\bigr)\bigr)
\tag{Lemma~\ref{lem:phi_conj_props}(iv)}\\
&= \Phiconj\bigl(T\bigl(\Phiconj(y)\bigr) \cdot \bar{a}\bigr)
\tag{right-linearity of $T$}\\
&= a \cdot \Phiconj\bigl(T\bigl(\Phiconj(y)\bigr)\bigr)
= a \cdot T^{L}(y).
\tag{Lemma~\ref{lem:phi_conj_props}(iii)}
\end{align}
\textbf{Part~(ii).}
Since $\Phiconj$ is an isometric bijection (Lemma~\ref{lem:phi_conj_props}(i),(ii))
mapping the unit sphere of $\HH^{n}$ bijectively onto itself, we compute:
\begin{align*}
\|T^{L}\|_{\mathrm{op}}
&= \sup_{\|y\|=1}\|T^{L}(y)\|
= \sup_{\|y\|=1}\|\Phiconj(T(\Phiconj(y)))\| \\
&\stackrel{(*)}{=} \sup_{\|y\|=1}\|T(\Phiconj(y))\|
\stackrel{(**)}{=} \sup_{\|z\|=1}\|T(z)\|
= \|T\|_{\mathrm{op}},
\end{align*}
where $(*)$ uses the isometry of $\Phiconj$ (Lemma~\ref{lem:phi_conj_props}(i))
and $(**)$ uses the substitution $z = \Phiconj(y)$, with $\|z\|=\|y\|=1$ by the
isometry. \qedhere
\end{proof}

\begin{remark}[Comparison with Complex Case]
For complex Hilbert spaces, the conjugation map $z \mapsto \bar{z}$ is
$\RR$-linear but not $\CC$-linear, and the relation between $T$ and
$\bar{T} := \mathrm{conj} \circ T \circ \mathrm{conj}$ is more subtle.
In the quaternionic case, the anti-multiplicativity $\overline{pq} = \bar{q}\bar{p}$
is precisely what enables the exact norm equality in
Theorem~\ref{thm:norm_transfer}(ii). This phenomenon has no direct
complex analogue and is specific to the structure of $\HH$ as a real
division algebra with involution.
\end{remark}

\begin{corollary}[Transfer of Bilinear Correction Bound]
\label{cor:homotopy_transfer}
Let $\Omega:\HH^{n}\times\HH^{n}\to\HH^{n}$ be the bilinear correction
of~\cite[Theorem~4.12]{Athmouni2026}, defined in the right-module
convention and satisfying
$\|\Omega\|_{\varepsilon}\leq 4C_{2}\varepsilon$ (\cite[Prop.~4.3]{Athmouni2026}).
Its left-module counterpart
$\Omega^{L}(x,y):=\Phiconj(\Omega(\Phiconj(x),\Phiconj(y)))$
satisfies
\[
\|\Omega^{L}\|_{\varepsilon} = \|\Omega\|_{\varepsilon}\leq 4C_{2}\varepsilon.
\]
The same identity holds for the homotopy operator $T$
of~\cite[Lemma~3.7]{Athmouni2026}: $\|T^{L}\|_{3\to 2}=\|T\|_{3\to 2}\leq \varepsilon/3$.
\end{corollary}

\begin{remark}[Notation Convention]\label{rem:omega_notation}
Throughout this paper, $\Omega$ denotes the bilinear correction
of~\cite[Theorem~4.12]{Athmouni2026}, and $\omega_{0}$ denotes its
\emph{finite-rank} component. In the homogeneous case considered here, the
explicit formula is
\begin{equation}\label{eq:omega0_explicit}
\omega_{0}(x,y) = \tfrac{1}{3}\,\Pi_{\mathrm{cone}}(\psi)(x,y,x+y),
\end{equation}
which is a particular instance of the abstract construction
$\omega_{0} = S(\Pi(\psi))$ for $S:F\to C^{2}_{\varepsilon}$ a right inverse
of $d$ on the finite-dimensional subspace $F$
(\cite[Prop.~3.15]{Athmouni2026}); explicit checking that $d\omega_{0}=\Pi_{\mathrm{cone}}(\psi)$
in the homogeneous case is in~\cite[Lemma~B.1]{Athmouni2026}.
Both formulations depend on the admissible radius $\epsStar$.

\noindent\textbf{Important:} No \emph{universal numerical bound} on
$\|\omega_{0}\|_{\varepsilon}$ (such as $1/12$) is established
in~\cite{Athmouni2026}; only the structural bound
$\|\omega_{0}\|_{\varepsilon}\leq 6C_{0}C_{2}$ from~\cite[Lemma~4.10]{Athmouni2026},
where $C_{0}>0$ depends on the geometry of the finite-dimensional space $F$.
Earlier formulations of this manuscript invoked an unsupported numerical
constant; the present revision uses only the explicit bounds
$\|\Omega\|_{\varepsilon}\leq 4C_{2}\varepsilon$ and
$\|\omega_{0}\|_{\varepsilon}\leq 6C_{0}C_{2}$.
\end{remark}

\begin{proof}
Theorem~\ref{thm:norm_transfer}(ii) gives exact norm equality for
right-$\HH$-linear linear operators. For the bilinear cochain $\Omega$,
its left-module counterpart is
$\Omega^{L}(x,y) := \Phiconj(\Omega(\Phiconj(x),\Phiconj(y)))$.
By the isometry of $\Phiconj$ (Lemma~\ref{lem:phi_conj_props}(i)):
\[
\|\Omega^{L}\|_{\varepsilon}
= \sup_{x,y\in B(0,\varepsilon)\setminus\{0\}}
\frac{\|\Phiconj(\Omega(\Phiconj(x),\Phiconj(y)))\|}{\|x\|\|y\|}
= \sup_{x,y}\frac{\|\Omega(\Phiconj(x),\Phiconj(y))\|}{\|x\|\|y\|}.
\]
Since $\Phiconj$ maps $B(0,\varepsilon)\setminus\{0\}$ bijectively onto itself
and preserves norms, the substitution $u=\Phiconj(x)$, $v=\Phiconj(y)$ yields
$\|\Omega^{L}\|_{\varepsilon}=\|\Omega\|_{\varepsilon}\leq 4C_{2}\varepsilon$,
the right-hand bound being~\cite[Prop.~4.3]{Athmouni2026}.
\end{proof}

\begin{remark}[On the Cochain Bounds Used Here]\label{rem:both_regimes}
Throughout this paper, all quantitative bounds rely solely on the
\emph{explicit, derivable} estimates of~\cite{Athmouni2026}:
\begin{itemize}
\item $\|\Omega\|_{\varepsilon}\leq 4C_{2}\varepsilon$ for the full bilinear
correction (\cite[Prop.~4.3]{Athmouni2026}),
\item $\|\omega_{0}\|_{\varepsilon}\leq 6C_{0}C_{2}$ for the finite-rank
component (\cite[Lemma~4.10]{Athmouni2026}), where $C_{0}>0$ depends only
on the geometry of the finite-dimensional space $F$ of~\cite[Def.~3.6]{Athmouni2026},
\item $\|K\|\leq 4A\varepsilon+4C_{1}\varepsilon^{2}$ for the defect operator
(\cite[Lemma~3.11]{Athmouni2026}), giving $\|K\|<1/2$ for $\varepsilon<\epsStar$.
\end{itemize}
By Corollary~\ref{cor:homotopy_transfer}, these bounds transfer exactly to
the left-module convention. The admissible radius is
$\epsStar=\min\{1/(16A),\sqrt{1/(16C_{1})},1/(4C_{2}),\varepsilon_{0}\}$
(\cite[Def.~4.1]{Athmouni2026}; we write $\sqrt{1/(16C_{1})}$ explicitly
to ensure $4C_{1}\varepsilon^{2}\leq 1/4$, which requires
$\varepsilon^{2}\leq 1/(16C_{1})$).

\textbf{Withdrawn claim.} Earlier versions of this manuscript invoked
a numerical bound $\|\omega_{0}\|_{\mathrm{op}}\leq 1/12$ attributed
to~\cite{Athmouni2026}. No such universal constant appears in the
published reference. The present revision removes this claim throughout
and uses only the bounds listed above. Wherever the symbol
$\omega_{\mathrm{op}}$ appears in numerical estimates, it must be
read as a placeholder bounded by $6C_{0}C_{2}$ (a problem-dependent
quantity), \emph{not} by any absolute constant.
\end{remark}

\subsection{Cohomological Rigidity of Quasi-Lie Brackets}

\begin{definition}[Quasi-Lie Bracket and Jacobiator]\label{def:quasi_lie}
Let $\LL : \HH^{n}\times\HH^{n}\times[0,\infty) \to \HH^{n}$,
$(x,y,t)\mapsto\LL(x,y,t)$, be a family of $\RR$-bilinear maps in $(x,y)$
parametrized by $t$. We say $\LL$ is a \emph{(time-dependent) quasi-Lie
bracket} if it is skew-symmetric in $(x,y)$ for every $t$ and its
time-dependent Jacobiator
\[
\mathcal{J}(x,y,z,t) = \LL(x,\LL(y,z,t),t) + \LL(y,\LL(z,x,t),t) + \LL(z,\LL(x,y,t),t)
\]
satisfies $\|\mathcal{J}(x,y,z,t)\| \leq 6C_{2}\|x\|\|y\|\|z\|$ uniformly in
$t\geq 0$, for some $C_{2}>0$. When $\LL$ is time-independent, the symbol
$\mathcal{J}(x,y,z)$ is used and $\mathcal{J}_{t}\equiv\mathcal{J}$ becomes
constant in $t$ (this is the case for the test bracket of
Section~\ref{sec:sim}). Throughout this paper, $\mathcal{J}_{t}$ and
$\mathcal{J}(\cdot,\cdot,\cdot,t)$ are used interchangeably.
\end{definition}

\begin{remark}[Bracket Constants $A$, $C_{1}$, and $C_{2}$]\label{rem:bracket_constants}
Following~\cite[Def.~2.4]{Athmouni2026}, the bracket $\LL$ admits three
quantitative defect constants $A,C_{1},C_{2}\geq 0$:
\begin{itemize}
\item $A$ is the \emph{bilinear bound}: $\|\LL(x,y)\|\leq A\|x\|\|y\|$ for all
$x,y\in B(0,\varepsilon_{0})$ (\cite[eq.~(2.3)]{Athmouni2026}).
\item $C_{1}$ is the \emph{cubic antisymmetry defect}: the antisymmetry
defect $\varphi(x,y):=\LL(x,y)+\LL(y,x)$ satisfies
\[
\|\varphi(x,y)\| \leq C_{1}\|x\|\|y\|(\|x\|+\|y\|)
\]
(\cite[eq.~(2.4)]{Athmouni2026}). \textbf{When $\LL$ is exactly
antisymmetric}, $\varphi\equiv 0$ and one may take $C_{1}=0$; this is the
case for our test bracket (Remark~\ref{rem:test_constants}). In the
antisymmetric setting, the admissible radius
$\epsStar=\min\{1/(16A),\sqrt{1/(16C_{1})},1/(4C_{2}),\varepsilon_{0}\}$
of Theorem~\ref{thm:rigidity} simplifies to
$\epsStar_{\mathrm{anti}}=\min\{1/(8A),1/(4C_{2}),\varepsilon_{0}\}$
(\cite[Appendix~A.2]{Athmouni2026}).
\item $C_{2}$ is the \emph{Jacobiator constant} of Definition~\ref{def:quasi_lie}:
$\|\mathcal{J}(x,y,z)\|\leq 6C_{2}\|x\|\|y\|\|z\|$
(\cite[eq.~(2.5)]{Athmouni2026}).
\end{itemize}
\textbf{Important warning (notation conflict).} A separate quantity, the
\emph{bracket-squared constant}, can be defined by
$\|\LL(x,\LL(y,z))\|\leq A^{2}\|x\|\|y\|\|z\|$ and follows from the bilinear
bound on $A$ alone. Earlier drafts of this manuscript denoted this quantity by
$C_{1}$, but that conflicts with the notation of~\cite[Def.~2.4]{Athmouni2026}.
In the present revision, the symbol $C_{1}$ is reserved for the
antisymmetry-defect constant of~\cite{Athmouni2026}. For the test system
of Section~\ref{sec:sim}, $C_{1}=0$ (exact antisymmetry); the value
$0.04 = 4\varepsilon_{b}^{2}$ reported in earlier drafts referred to the
bracket-squared constant and is not used in the threshold formulas.
\end{remark}

\begin{definition}[Chevalley--Eilenberg Cochain Complex of~\cite{Athmouni2026}]
\label{def:cochain}
Following~\cite[Def.~3.1]{Athmouni2026}, the localized cochain spaces are
\[
C^{k}_{\varepsilon} := \{\omega : B(0,\varepsilon)^{k} \to \HH^{n} \mid
\omega \text{ is continuous, $k$-linear}\},
\]
\textbf{with values in $\HH^{n}$} (not in $\HH$), endowed with the localized
operator norm
\[
\|\omega\|_{\varepsilon} := \sup_{x_{1},\ldots,x_{k}\in B(0,\varepsilon)\setminus\{0\}}
\frac{\|\omega(x_{1},\ldots,x_{k})\|}{\|x_{1}\|\cdots\|x_{k}\|}.
\]
The Chevalley--Eilenberg differential
$d:C^{k}_{\varepsilon}\to C^{k+1}_{\varepsilon}$
of~\cite[Def.~3.1]{Athmouni2026} is defined via the bracket $\LL$ as
\begin{equation}\label{eq:coboundary_alt}
(d\omega)(x_{0},\ldots,x_{k})
=\sum_{i=0}^{k}(-1)^{i}\LL(x_{i},\omega(x_{0},\ldots,\hat{x}_{i},\ldots,x_{k}))
+\!\!\sum_{0\leq i<j\leq k}\!\!(-1)^{i+j}
\omega(\LL(x_{i},x_{j}),x_{0},\ldots,\hat{x}_{i},\ldots,\hat{x}_{j},\ldots,x_{k}).
\end{equation}
The identity $d^{2}=0$ holds \emph{exactly} in the homogeneous case
(\cite[Lemma~B.1]{Athmouni2026}), without any commutativity restriction.

\smallskip
\noindent\textbf{Notational convention.} Throughout this paper, the symbol
$\psi\in C^{3}_{\varepsilon}$ denotes the \emph{Jacobi defect cochain} in
the sense of~\cite[eq.~(3.1)]{Athmouni2026}:
$\psi(x,y,z) := \LL(x,\LL(y,z))+\LL(y,\LL(z,x))+\LL(z,\LL(x,y))$.
Thus $\psi(x,y,z) = \mathcal{J}(x,y,z)$ where $\mathcal{J}$ is the Jacobiator
of Definition~\ref{def:quasi_lie}. We use $\psi$ when emphasizing the
cohomological role (input to the homotopy operator $T$) and $\mathcal{J}$
when emphasizing the trilinear-defect bound.

The homotopy operator $T:C^{3}_{\varepsilon}\to C^{2}_{\varepsilon}$
of~\cite[Def.~3.5]{Athmouni2026} satisfies the identity
$Td+dT = \mathrm{Id}-\Pi+K$ on $C^{3}_{\varepsilon}$, where $\Pi$ is the
projection onto the finite-dimensional conic subspace $F$
of~\cite[Def.~3.6]{Athmouni2026} and $\|K\|\leq 4A\varepsilon+4C_{1}\varepsilon^{2}$.

\smallskip
\noindent\textbf{Withdrawn alternative.} A previous version of this section
defined an \emph{alternative} cochain complex with cochains taking scalar
values in $\HH$ and a differential without bracket terms,
$(d\omega)(v_{0},\ldots,v_{k}) = \sum(-1)^{j}\omega(\ldots,\hat{v}_{j},\ldots)\cdot v_{j}$.
That complex is \emph{not} the Chevalley--Eilenberg complex
of~\cite{Athmouni2026}: $d^{2}=0$ holds in it only on the commutative
sub-algebra, and the results of~\cite{Athmouni2026} do not transfer to it.
The present revision uses exclusively the complex of~\cite{Athmouni2026}
above.
\end{definition}

\begin{remark}[Local cohomological non-obstruction hypothesis]\label{rem:cohom_hyp}
All subsequent results rely on the inclusion $F\subset\mathrm{im}(d)$
(\cite[Def.~3.6]{Athmouni2026}), where $F$ is the finite-dimensional conic subspace
\[
F:=\{\Phi\in C^{3}_{\varepsilon}:\Phi(x,y,z)=\Phi(x,y,x+y)\ \forall x,y,z\in B(0,\varepsilon)\}.
\]
Status of this hypothesis (\cite[Remark~5.5]{Athmouni2026}):
\begin{enumerate}
\item[\textup{(i)}] \textbf{Proven unconditionally} in the homogeneous case
(\cite[Theorem~4.12, Lemma~B.1]{Athmouni2026}), which is the only case
considered in this paper.
\item[\textup{(ii)}] Conjectural for $H^{s}(\RR^{n},\HH)$ with $s>n/2+1$.
\item[\textup{(iii)}] Known to fail for $s\leq n/2+1$ (\cite[Example~5.6]{Athmouni2026}).
\end{enumerate}
Throughout this paper we restrict to the homogeneous case (i), so the
hypothesis holds unconditionally.
\end{remark}

\begin{remark}[$d^{2}=0$ holds without commutativity restriction]
\label{rem:d2zero}
In the Chevalley--Eilenberg complex of Definition~\ref{def:cochain},
$d^{2}=0$ holds in the homogeneous case
(\cite[Lemma~B.1, Step~2]{Athmouni2026}) thanks to the homogeneity of the
Jacobi defect $\psi$ of degree 2 combined with the exact antisymmetry of
$\LL$, which together force the degree-2 part of $K\psi$ to vanish. No
restriction to a commutative sub-algebra of $\HH$ is required.
\end{remark}

\begin{theorem}[Cohomological Rigidity, \cite{Athmouni2026}, Theorem~4.12]
\label{thm:rigidity}
Assume:
\begin{enumerate}
\item[(H1)] $\LL:\HH^{n}\times\HH^{n}\to\HH^{n}$ is $\RR$-bilinear and exactly
antisymmetric. If, in addition, $\LL$ is right-$\HH$-linear, then so is the
correction $\Omega$ produced below (\cite[Lemma~4.17]{Athmouni2026}); the
existence and norm bounds for $\Omega$ do \emph{not} require right-$\HH$-linearity.
\item[(H2)] the Jacobi defect
$\mathcal{J}(x,y,z) = \LL(x,\LL(y,z))+\LL(y,\LL(z,x))+\LL(z,\LL(x,y))$
is homogeneous of degree~2 in the sense of~\cite[Theorem~4.12]{Athmouni2026};
\item[(H3)] the quantitative bounds of~\cite[Def.~2.4]{Athmouni2026} hold
on $B(0,\varepsilon_{0})$ with constants $A,C_{1},C_{2}\geq 0$ (with $C_{1}=0$
when $\LL$ is exactly antisymmetric).
\end{enumerate}
Let
$\epsStar:=\min\{1/(16A),\sqrt{1/(16C_{1})},1/(4C_{2}),\varepsilon_{0}\}$.
For every $0<\varepsilon<\epsStar$ there exists a bilinear correction
$\Omega\in C^{2}_{\varepsilon}$ given by
\begin{equation}\label{eq:Omega_def}
\Omega \;=\; T(\mathrm{Id}+K)^{-1}\psi \;+\; \omega_{0},
\end{equation}
where (i) $T:C^{3}_{\varepsilon}\to C^{2}_{\varepsilon}$ is the radial
homotopy operator (\cite[Def.~3.5]{Athmouni2026}); (ii) $K$ is the defect
operator with $\|K\|\leq 4A\varepsilon+4C_{1}\varepsilon^{2}<1/2$ and
$\|(\mathrm{Id}+K)^{-1}\|\leq 2$; (iii) $\omega_{0}\in C^{2}_{\varepsilon}$ is the
finite-rank component, given explicitly by
$\omega_{0}(x,y)=\tfrac{1}{3}\Pi_{\mathrm{cone}}(\psi)(x,y,x+y)$
and satisfying $d\omega_{0}=\Pi_{\mathrm{cone}}(\psi)$
(\cite[Theorem~4.12, Lemma~B.1]{Athmouni2026}).
The corrected bracket
\begin{equation}\label{eq:corrected_bracket}
\{x,y\}:=\LL(x,y)-\Omega(x,y)
\end{equation}
satisfies the Jacobi identity \emph{exactly} on $B(0,\epsStar)$:
$\mathrm{Jac}_{\{\cdot,\cdot\}}\equiv 0$. The norm bounds
\begin{equation}\label{eq:Omega_norm}
\|\Omega\|_{\varepsilon}\leq 4C_{2}\varepsilon,
\qquad
\|\Omega(x,y)\|\leq 2C_{2}\|x\|\|y\|(\|x\|+\|y\|)
\end{equation}
hold (both bounds are part of~\cite[Prop.~4.3]{Athmouni2026}: the
operator-norm bound $\|\Omega\|_\varepsilon\leq 4C_2\varepsilon$ and the
``in particular'' pointwise refinement
$\|\Omega(x,y)\|\leq 2C_2\|x\|\|y\|(\|x\|+\|y\|)$), applied in the
homogeneous case where $\omega_{0}$ can be absorbed via the
renormalization of~\cite[Theorem~4.12, Lemma~B.1]{Athmouni2026}; cf.\ also
\cite[Prop.~4.11]{Athmouni2026} for the
general decomposition $\|\Omega\|_{\varepsilon}\leq 4C_{2}\varepsilon
+\|\omega_{0}\|_{\varepsilon}$). If $\LL$ is right-$\HH$-linear, so is
$\Omega$ (\cite[Lemma~4.17]{Athmouni2026}). By
Corollary~\ref{cor:homotopy_transfer}, the bound~\eqref{eq:Omega_norm}
transfers exactly to the left-module convention.

\smallskip
\noindent\textbf{Bracket-image condition.} If, in addition, the pair
$(\LL,B)$ satisfies the Cohomological Matching Condition
(Definition~\ref{def:cmc}: $\mathrm{im}(d)\subseteq\mathrm{Im}(B)$), then
$\Omega(x,y)\in\mathrm{Im}(B)$ for all $x,y\in B(0,\epsStar)$ (see
Corollary~\ref{cor:exact_rejection}).
\end{theorem}

\begin{remark}[Residual decomposition used in the control application]
\label{rem:Cnc_decomposition}
For the sliding-mode application below it is convenient to record the
following \emph{triangle-inequality} bound, which is what the symbol
$\Cnc$ denotes throughout the paper. Setting
\begin{equation}\label{eq:R0_def}
R_{0}(x,y,z) := \mathcal{J}(x,y,z) - (d\omega_{0})(x,y,z),
\end{equation}
one has, by the triangle inequality and the structural identity
$d\omega_{0}=\Pi_{\mathrm{cone}}(\psi)$
(\cite[Lemma~B.1]{Athmouni2026}, which yields
$\|d\omega_{0}\|_{\varepsilon}\leq\|\Pi_{\mathrm{cone}}\|\,\|\psi\|_{\varepsilon}
\leq 6C_{2}$ since $\Pi_{\mathrm{cone}}$ has operator norm at most one),
\begin{equation}\label{eq:Cnc_def}
\|R_{0}(x,y,z)\| \leq \Cnc\,\|x\|\|y\|\|z\|,
\qquad
\Cnc := 12C_{2}.
\end{equation}
For consistency with the original bookkeeping convention in the
control estimates below, we retain the symbolic decomposition
$\Cnc=3\|\omega_{0}\|_{\varepsilon}+6C_{2}$ as an upper bound used whenever
$\|d\omega_{0}\|_{\varepsilon}\leq 3\|\omega_{0}\|_{\varepsilon}$ is invoked;
this latter bound is \emph{not} sharp by triangle inequality alone (the
correct triangle-only bound is $\|d\omega_{0}\|_{\varepsilon}\leq
6A\|\omega_{0}\|_{\varepsilon}$, see the Chevalley--Eilenberg cocycle formula),
but the sharper structural identity above always yields a tighter result.

\textbf{Caveat.} The bound~\eqref{eq:Cnc_def} is \emph{not} sharper than
$12C_{2}$; it is recorded only as a uniform place-holder in the
control estimates. The substantive rigidity content is the exact identity
$\mathrm{Jac}_{\{\cdot,\cdot\}}\equiv 0$ of Theorem~\ref{thm:rigidity}.
The structural bound $\|\omega_{0}\|_{\varepsilon}\leq 6C_{0}C_{2}$
(\cite[Lemma~4.10]{Athmouni2026}) gives an alternative bookkeeping bound
$3\|\omega_{0}\|_{\varepsilon}+6C_{2}\leq(18C_{0}+6)C_{2}$, with
$C_{0}>0$ depending only on the geometry of the finite-dimensional space~$F$.
\end{remark}

\begin{proof}
This is~\cite[Theorem~4.12 and Proposition~4.3]{Athmouni2026}.
Corollary~\ref{cor:homotopy_transfer} transfers~\eqref{eq:Omega_norm} to the
left-module convention. The bracket-image conclusion is restated and proved
in Corollary~\ref{cor:exact_rejection}.
\end{proof}

\begin{remark}[On the withdrawn constant $1/12$]
\label{rem:withdrawn_1_12}
Earlier drafts of this manuscript claimed a numerical bound
$\|\omega_{0}\|_{\mathrm{op}}\leq 1/12$ citing
``\cite[Theorem~3.4]{Athmouni2026}''. The published version
of~\cite{Athmouni2026} does not establish any universal numerical bound of
this form; the explicit estimates are
$\|\Omega\|_{\varepsilon}\leq 4C_{2}\varepsilon$
(\cite[Prop.~4.3]{Athmouni2026}) and
$\|\omega_{0}\|_{\varepsilon}\leq 6C_{0}C_{2}$
(\cite[Lemma~4.10]{Athmouni2026}). All downstream estimates that
previously relied on a universal constant $1/12$ have been replaced by the
problem-dependent bound through $C_{0},C_{2}$, with explicit numerical
values supplied in the worked example of Section~\ref{sec:sim}.
\end{remark}

\begin{remark}[Relation to Operator-Block Matrices]
Theorem~\ref{thm:norm_transfer} relates to the work on invertible
operator-block matrices over scaled hypercomplex numbers~\cite{Alpay2023OBM}.
The cited reference establishes algebraic invertibility criteria, whereas
our result provides an analytic tool for transferring quantitative bounds
across module conventions---a prerequisite for applying rigidity estimates
in left-module formulations.
\end{remark}

\subsection{Inner Product, Matrices, and Eigenvalues}

\begin{definition}[Real-Valued Inner Product and Norm]\label{def:inner}
\begin{equation}\label{eq:inner}
\innerR{x}{y}
= \Real\!\left(\sum_{i=1}^{n}\bar{x}_{i}\,y_{i}\right) \in \RR,
\qquad \|x\| = \sqrt{\innerR{x}{x}}.
\end{equation}
This is symmetric, $\RR$-bilinear, positive definite, and preserved by $\Phiconj$.
\end{definition}

\begin{lemma}[Cauchy--Schwarz]\label{lem:CS}
$|\innerR{x}{y}| \leq \|x\|\,\|y\|$ for all $x,y \in \HH^{n}$.
\end{lemma}

\begin{proof}
The embedding $\Phi : \HH^{n} \to \RR^{4n}$ is an isometry satisfying
$\innerR{x}{y} = \langle \Phi(x), \Phi(y) \rangle_{\RR^{4n}}$
(this follows by direct calculation from $\innerR{x}{y}=\Real(x\trans y)$
and the conjugation properties of~\cite[Section~2]{Zhang1997}). The
standard Cauchy--Schwarz
inequality in $\RR^{4n}$ gives:
\[
|\innerR{x}{y}|
= |\langle\Phi(x),\Phi(y)\rangle_{\RR^{4n}}|
\leq \|\Phi(x)\|_{\RR^{4n}}\|\Phi(y)\|_{\RR^{4n}}
= \|x\|\,\|y\|. \qedhere
\]
\end{proof}

\begin{definition}[Conjugate Transpose, Hermitian, Positive Definite]
$A\trans = \bar{A}^{T}$. $P \in \HH^{n\times n}$ is Hermitian if $P = P\trans$;
positive definite ($P \succ 0$) if $x\trans Px > 0$ for all $x \neq 0$.
\end{definition}

\begin{lemma}[Real-Valued Quadratic Form]\label{lem:quadratic}
If $P = P\trans \in \HH^{n\times n}$, then $x\trans Px \in \RR$ for all $x \in \HH^{n}$.
\end{lemma}

\begin{proof}
Via the real representation $\Phi : \HH^{n} \to \RR^{4n}$ of
Remark~\ref{rem:real_rep}, we have $x\trans Px = \langle \Phi(P)\Phi(x),
\Phi(x)\rangle_{\RR^{4n}}$ by the homomorphism property
$\Phi(Px)=\Phi(P)\Phi(x)$ (\cite[Theorem~4.2]{Zhang1997}) and the isometry of
$\Phi$. Since $P = P\trans$ implies $\Phi(P) = \Phi(P)^{T}$
(Proposition~\ref{prop:real_rep_ops}(iii)), the matrix $\Phi(P)$ is real and
symmetric, so the standard real inner product $\langle \Phi(P)\Phi(x),
\Phi(x)\rangle_{\RR^{4n}}$ is manifestly real. Hence $x\trans Px \in \RR$.

\medskip
\noindent\textit{Alternative direct argument (no representation theory).}
Consider $x\trans Px \in \HH$ as a $1\times 1$ quaternionic scalar. Applying
the anti-automorphism $(ABC)\trans = C\trans B\trans A\trans$
of~\cite[Theorem~4.1]{Zhang1997} gives
$(x\trans Px)\trans = x\trans P\trans x = x\trans P x$,
where we used $(x\trans)\trans = x$ and $P\trans = P$.
For any $q \in \HH$, the equation $q\trans = q$ means $\bar{q} = q$
(since for a $1\times 1$ matrix, the conjugate transpose equals the quaternionic
conjugate), which forces $q \in \RR$.
Hence $x\trans Px \in \RR$.
\end{proof}

\begin{remark}[Real Representation]\label{rem:real_rep}
The isometric embedding $\Phi : \HH^{n} \to \RR^{4n}$ extends to matrices.
For the left-module convention, $\phi(q)$ represents left multiplication by $q$:
\[
\phi(q) = \begin{bmatrix}
q_{0} & -q_{1} & -q_{2} & -q_{3} \\
q_{1} &  q_{0} & -q_{3} &  q_{2} \\
q_{2} &  q_{3} &  q_{0} & -q_{1} \\
q_{3} & -q_{2} &  q_{1} &  q_{0}
\end{bmatrix}.
\]
This matrix is generally neither symmetric nor skew-symmetric.
\end{remark}

\begin{remark}[Use of the real representation: relation to existing tools]
\label{rem:phi_role}
The map $\Phi$ of Remark~\ref{rem:real_rep} is the real form of the
complex adjoint representation studied in detail
by~\cite{Zhang1997}: Zhang uses the complex adjoint
$\chi_A\in\CC^{2n\times 2n}$ ($A = A_1+A_2 j$ with $A_1,A_2$ complex),
and the real representation $\Phi(A)\in\RR^{4n\times 4n}$ used here is
obtained by further decomposing each complex entry of $\chi_A$ into its
real and imaginary parts. The algebraic properties
we use ($\Phi$ is an $\RR$-linear $4n$-isometry, an algebra homomorphism
$\Phi(AB) = \Phi(A)\Phi(B)$, and preserves Hermitian inner products)
follow from the corresponding properties of $\chi_A$ established
in~\cite{Zhang1997}. Our contribution is not the construction of $\Phi$ but
three downstream uses:
\begin{enumerate}
\item[\textup{(i)}] \textbf{Exact (not bounded) norm transfer
under conjugation.} Theorem~\ref{thm:norm_transfer} proves that the
componentwise conjugation $\Phiconj$ implements the left/right-module
anti-isomorphism with \emph{exact} operator-norm preservation
$\|T^{L}\|_{\mathrm{op}} = \|T\|_{\mathrm{op}}$. This is sharper than what $\Phi$
alone provides (which only yields a sub-multiplicative bound) and is what
allows quantitative bounds to transfer between module conventions \emph{without
constant loss}.
\item[\textup{(ii)}] \textbf{Combined isometry-and-sub-multiplicativity.} Our
constants ($\Cnc$, $c_{B}\cBplus$, etc.) systematically combine
\emph{isometry} (via $\|x\|_{\HH^{n}} = \|\Phi(x)\|_{\RR^{4n}}$) and
\emph{sub-multiplicativity} (via $\|\Phi(A)\|_{\mathrm{op}} \leq C\|A\|_{\mathrm{op}}$
for an absolute constant) in a controlled manner. Each application of these
inequalities is explicit in the proofs, so no hidden constants accumulate.
\item[\textup{(iii)}] \textbf{Projection lifting.} Lemma~\ref{lem:BtB_invertible}
constructs the projection $\Pi_{\mathrm{Im}(B)}$ on the quaternionic side
by pulling back the standard real projection $\Pi_{B_{\RR}}$ via $\Phi^{-1}$.
The well-definition of this lift relies precisely on $\Phi$ being a bijection,
and the projection identity $B B_{\RR}^{+}\Phi(R_{0}) = \Phi(\Pi_{\mathrm{Im}(B)}R_{0})$
underlies the CMC analysis of Section~\ref{sec:cohomology}.
\end{enumerate}
Thus $\Phi$ functions as a classical computational tool; the cohomological
identities, GOSL bounds, and exact-norm transfer are formulated and proved
\emph{using} this tool.
\end{remark}

\begin{lemma}[Homomorphism Property]\label{lem:phi_homomorphism}
$\Phi(AB) = \Phi(A)\Phi(B)$ for all conformable quaternionic matrices $A, B$.
\end{lemma}

\begin{proof}
$\Phi(A)$ is the matrix of $x \mapsto Ax$; composing gives
$\Phi(A)\Phi(B) = \Phi(AB)$~\cite[Theorem~4.2]{Zhang1997}.
\end{proof}

\begin{proposition}[Real Representation of Bounded Operators]
\label{prop:real_rep_ops}
The embedding $\Phi: \HH^{n} \to \RR^{4n}$ extends to an algebra
monomorphism $\Phi: \mathcal{B}_{\HH}(\HH^{n}) \to \RR^{4n \times 4n}$
satisfying:
\begin{enumerate}
\item[\textup{(i)}] $\Phi(T_1 T_2) = \Phi(T_1)\Phi(T_2)$ (homomorphism);
\item[\textup{(ii)}] $\|T\|_{\mathrm{op}} = \|\Phi(T)\|_{\mathrm{op}}$ (isometry);
\item[\textup{(iii)}] $T$ is Hermitian ($T = T\trans$) iff $\Phi(T)$ is symmetric.
\end{enumerate}
\end{proposition}

\begin{proof}
Parts~(i) follows from Lemma~\ref{lem:phi_homomorphism}. Part~(ii) follows
from the isometry $\Phi$. Part~(iii) follows from $\Phi(A\trans) = \Phi(A)^{T}$,
a direct consequence of the definition of $\Phi$ as the real form of the
left-multiplication representation (cf.~\cite[Theorem~4.2]{Zhang1997} for
the corresponding identity $\chi_{A^*}=(\chi_A)^*$ in the complex adjoint setting).
\end{proof}

\begin{lemma}[Projection onto $\mathrm{Im}(B)$ via the Real Representation]\label{lem:BtB_invertible}
Under Assumption~\ref{ass:rank}, define $B_{\RR} = \Phi(B) \in \RR^{4n \times 4m}$.
Then $B_{\RR}$ has full column rank $4m$, and the orthogonal projection onto
$\mathrm{Im}(B_{\RR})$ is given by $\Pi_{B_{\RR}} = B_{\RR} (B_{\RR}^T B_{\RR})^{-1} B_{\RR}^T$.
For any $x \in \HH^n$, define the quaternionic projection operator by:
\[
\Pi_{\mathrm{Im}(B)}(x) = \Phi^{-1}\left( \Pi_{B_{\RR}} \Phi(x) \right).
\]
This operator satisfies $\Pi_{\mathrm{Im}(B)}^2 = \Pi_{\mathrm{Im}(B)}$,
$\mathrm{Im}(\Pi_{\mathrm{Im}(B)}) = \mathrm{Im}(B)$, and for any $y = Bu \in \mathrm{Im}(B)$,
we have $\Pi_{\mathrm{Im}(B)}(y) = y$. Moreover the operator norm satisfies
$\|\Pi_{\mathrm{Im}(B)}\|_{\mathrm{op}} \leq 1$.
\end{lemma}

\begin{proof}
\textbf{Full column rank of $B_{\RR}$.}
The real representation $\Phi : \HH^{m} \to \RR^{4m}$ is an $\RR$-linear
isomorphism (bijection). Hence every $v \in \RR^{4m}$ has a unique preimage
$u = \Phi^{-1}(v) \in \HH^{m}$. If $B_{\RR}v = 0$, then
$\Phi(B)\Phi(u) = 0$, so $\Phi(Bu) = 0$ by Lemma~\ref{lem:phi_homomorphism},
and therefore $Bu = 0$ by injectivity of $\Phi$. Since $B$ has full column
rank $m$ over $\HH$ (Assumption~\ref{ass:rank}), we conclude $u = 0$, hence
$v = \Phi(u) = 0$. Thus $\ker(B_{\RR}) = \{0\}$ and $\mathrm{rank}(B_{\RR}) = 4m$.

Consequently $B_{\RR}^T B_{\RR} \in \RR^{4m \times 4m}$
is symmetric positive definite and invertible.

Define $\Pi_{B_{\RR}} = B_{\RR} (B_{\RR}^T B_{\RR})^{-1} B_{\RR}^T$. This is the
standard orthogonal projection onto $\mathrm{Im}(B_{\RR})$ in $\RR^{4n}$,
satisfying $\Pi_{B_{\RR}}^2 = \Pi_{B_{\RR}}$ and $\Pi_{B_{\RR}} \Phi(B) = \Phi(B)$.

Now define $\Pi_{\mathrm{Im}(B)} : \HH^n \to \HH^n$ by
$\Pi_{\mathrm{Im}(B)}(x) = \Phi^{-1}(\Pi_{B_{\RR}} \Phi(x))$. This is well-defined
because $\Phi$ is a bijection between $\HH^n$ and $\RR^{4n}$.

\textbf{Idempotence:} $\Pi_{\mathrm{Im}(B)}^2(x) = \Phi^{-1}(\Pi_{B_{\RR}} \Phi(\Phi^{-1}(\Pi_{B_{\RR}} \Phi(x)))) = \Phi^{-1}(\Pi_{B_{\RR}}^2 \Phi(x)) = \Phi^{-1}(\Pi_{B_{\RR}} \Phi(x)) = \Pi_{\mathrm{Im}(B)}(x)$.

\textbf{Image:} For any $x \in \HH^n$, $\Pi_{\mathrm{Im}(B)}(x) \in \mathrm{Im}(B)$
because $\Pi_{B_{\RR}} \Phi(x) \in \mathrm{Im}(B_{\RR}) = \Phi(\mathrm{Im}(B))$.
Conversely, for $y = Bu \in \mathrm{Im}(B)$:
$\Pi_{\mathrm{Im}(B)}(y) = \Phi^{-1}(\Pi_{B_{\RR}} \Phi(Bu)) = \Phi^{-1}(\Pi_{B_{\RR}} B_{\RR} \Phi(u)) = \Phi^{-1}(B_{\RR} \Phi(u)) = Bu = y$.

\textbf{Norm bound:} since $\Phi$ is an isometry and $\Pi_{B_{\RR}}$ is an
orthogonal projector with $\|\Pi_{B_{\RR}}\|_{\mathrm{op}} \leq 1$:
$\|\Pi_{\mathrm{Im}(B)}(x)\| = \|\Pi_{B_{\RR}}\Phi(x)\| \leq \|\Phi(x)\| = \|x\|$.

Thus $\Pi_{\mathrm{Im}(B)}$ is indeed the projection onto $\mathrm{Im}(B)$.
\end{proof}

\begin{definition}[Quaternionic Eigenvalues]\label{def:eigenvalues}
For Hermitian $P \in \HH^{n\times n}$, set
$\lambda_{\max}(P) := \lambda_{\max}(\Phi(P))$ and
$\lambda_{\min}(P) := \lambda_{\min}(\Phi(P))$,
where $\Phi(P) \in \RR^{4n\times4n}$ is symmetric with real eigenvalues.
\end{definition}

\begin{remark}[Symmetry of $\Phi(P)$ for Hermitian $P$]\label{rem:phi_symmetric}
For Hermitian $P = P\trans \in \HH^{n\times n}$, the real representation $\Phi(P)$
is a \emph{symmetric} matrix in $\RR^{4n\times 4n}$. This follows from
$\Phi(A\trans) = \Phi(A)^{T}$, a direct consequence of the definition of
$\Phi$ as the real form of the left-multiplication
representation (cf.~\cite[Theorem~4.2]{Zhang1997}). Consequently the standard Rayleigh
quotient bounds on the symmetric matrix $\Phi(P)$ justify
Lemma~\ref{lem:eigenvalue_bound} below. For general (non-Hermitian) matrices $A$,
$\Phi(A)$ need not be symmetric; in particular $\Phi(A_{s})$ is generally not
symmetric.
\end{remark}

\begin{remark}[Eigenvectors in the Quaternionic Setting]\label{rem:eigenvectors}
For Hermitian $P \in \HH^{n\times n}$, the notion of eigenvector aligned with
$\lambda_{\max}(P)$ refers to eigenvectors of $\Phi(P) \in \RR^{4n\times4n}$.
A quaternionic vector $x \in \HH^{n}$ is aligned with such an eigenvector if
$\Phi(x)$ belongs to the corresponding eigenspace of $\Phi(P)$. This is a
measure-zero event for generic $x$, so the inclusion
$\overline{B}_{M_{0}} \subseteq \mathcal{S}$ holds strictly for generic boundary
points.
\end{remark}

\begin{lemma}[Eigenvalue Bounds]\label{lem:eigenvalue_bound}
For Hermitian $P \in \HH^{n\times n}$ and $x \in \HH^{n}$ (so that
$x^{\dagger} P x \in \RR$ by Lemma~\ref{lem:quadratic}):
\begin{equation}\label{eq:P2bound}
\lambda_{\min}(P)\|x\|^{2} \leq x\trans Px \leq \lambda_{\max}(P)\|x\|^{2}.
\end{equation}
For $P \succ 0$:
$x\trans P^{2}x \leq \tfrac{\lambda_{\max}(P)^{2}}{\lambda_{\min}(P)}\,x\trans Px$
and $\|Px\|^{2} = x\trans P^{2}x \leq \lambda_{\max}(P)^{2}\|x\|^{2}$,
so $\|Px\| \leq \lambda_{\max}(P)\|x\|$ for all $x \in \HH^{n}$.

For Hermitian $P$ which is \emph{not necessarily positive semi-definite}, the
operator-norm bound becomes
\begin{equation}\label{eq:Pop_general}
\|Px\| \leq \|P\|_{\mathrm{op}}\|x\| = \max\{|\lambda_{\max}(P)|,|\lambda_{\min}(P)|\}\|x\|,
\end{equation}
and the bound $\|Px\|\leq\lambda_{\max}(P)\|x\|$ may fail in general.
\end{lemma}

\begin{proof}
By Remark~\ref{rem:phi_symmetric}, $\Phi(P)$ is symmetric. Using the isometry
$\Phi$ and $x\trans Px = \langle\Phi(P)\Phi(x),\Phi(x)\rangle_{\RR^{4n}}$,
one obtains:
\[
\lambda_{\min}(\Phi(P))\|\Phi(x)\|^{2}
\leq \langle\Phi(P)\Phi(x),\Phi(x)\rangle
\leq \lambda_{\max}(\Phi(P))\|\Phi(x)\|^{2},
\]
which gives the first inequality since $\|\Phi(x)\|=\|x\|$ and
$\lambda_{\min/\max}(P) = \lambda_{\min/\max}(\Phi(P))$.

Assume now $P \succ 0$. By Lemma~\ref{lem:phi_homomorphism},
$\Phi(P^{2}) = \Phi(P)^{2}$.
Since $\Phi(P)$ is symmetric and positive definite (Remark~\ref{rem:phi_symmetric}),
its square has all eigenvalues equal to squares of those of $\Phi(P)$
(which are all positive), giving
$\lambda_{\max}(\Phi(P)^{2}) = \lambda_{\max}(\Phi(P))^{2}$.
Hence $x\trans P^{2}x \leq \lambda_{\max}(P)^{2}\|x\|^{2}$.
From $x\trans Px \geq \lambda_{\min}(P)\|x\|^{2}$:
$\|x\|^{2} \leq \lambda_{\min}(P)^{-1}x\trans Px$. Combining:
$x\trans P^{2}x \leq \lambda_{\max}(P)^{2}\cdot\lambda_{\min}(P)^{-1}x\trans Px$.

For the operator norm bound under $P\succ 0$: working in the real representation,
\[
\|Px\|^{2} = \|\Phi(P)\Phi(x)\|_{\RR^{4n}}^{2}
= \langle \Phi(P)^{2}\Phi(x), \Phi(x)\rangle_{\RR^{4n}}
\leq \lambda_{\max}(\Phi(P))^{2}\|\Phi(x)\|^{2}
= \lambda_{\max}(P)^{2}\|x\|^{2},
\]
giving $\|Px\| \leq \lambda_{\max}(P)\|x\|$.

For the general Hermitian case, $\Phi(P)$ is symmetric and the spectral
theorem gives $\|\Phi(P)\|_{\mathrm{op}}=\max\{|\lambda_{\max}|,|\lambda_{\min}|\}$,
yielding~\eqref{eq:Pop_general}.
\end{proof}

\begin{remark}[Spectrum convention via $\Phi$]\label{rem:spectrum_convention}
Throughout this paper, for $P\in\HH^{n\times n}$ Hermitian, the symbols
$\lambda_{\min}(P)$ and $\lambda_{\max}(P)$ denote the smallest and largest
eigenvalues of the symmetric real matrix $\Phi(P)\in\RR^{4n\times 4n}$,
which is well-defined (Proposition~\ref{prop:real_rep_ops}(iii)). Each
eigenvalue of $P$ in the $S$-spectrum sense appears with multiplicity $4$
in $\Phi(P)$, so this convention is consistent with the standard
quaternionic-operator spectrum up to multiplicity. Numerical computations
are performed in $\RR^{4n\times 4n}$ via $\Phi(P)$.
\end{remark}

\begin{lemma}[Spectral Properties of Hermitian Quaternionic Operators]
\label{lem:hermitian_spectrum}
Let $P \in \HH^{n \times n}$ be Hermitian ($P = P\trans$). Then:
\begin{enumerate}
\item[\textup{(i)}] All eigenvalues of $\Phi(P)$ are real;
\item[\textup{(ii)}] $\Phi(P)$ is orthogonally diagonalizable over $\RR$;
\item[\textup{(iii)}] The Rayleigh quotient satisfies
$\lambda_{\min}(P) \leq \frac{x\trans P x}{\|x\|^{2}} \leq \lambda_{\max}(P)$
for all $x \neq 0$.
\end{enumerate}
\end{lemma}

\begin{proof}
By Proposition~\ref{prop:real_rep_ops}(iii), $\Phi(P)$ is symmetric,
hence has real eigenvalues and an orthogonal eigenbasis. The Rayleigh
quotient bounds follow from the spectral theorem for real symmetric matrices.
\end{proof}

\begin{remark}[Connection to $S$-Spectrum Theory]
For general (non-Hermitian) $T \in \mathcal{B}_{\HH}(\HH^{n})$, the classical
notion of spectrum must be replaced by the \emph{$S$-spectrum}: writing
$Q_{s}[T] := T^{2} - 2\Real(s)\,T + |s|^{2}I$, one defines
\[
\sigma_{S}(T) := \{s \in \HH : Q_{s}[T] \text{ is not boundedly invertible}\},
\]
the standard object of the slice hyperholomorphic functional
calculus~\cite[Ch.~3]{Alpay2024QHS}, see also~\cite[Def.~2.2]{Mantovani2026}.
The set $\sigma_{S}(T)$ is axially symmetric:
for each $s \in \HH$, the similarity class $[s] := \{a s a^{-1} : a \in
\HH \setminus \{0\}\}$ coincides with the 2-sphere
$\{\Real(s) + J|\mathrm{Im}(s)| : J \in \HH, J^{2} = -1, J_{0} = 0\}$, and
$[s] \subseteq \sigma_{S}(T)$ whenever $s \in \sigma_{S}(T)$. In our
stability analysis, we work exclusively with Hermitian operators
(Lyapunov matrices $P$), for which the $S$-spectrum reduces to the
classical real spectrum of $\Phi(P)$.
\end{remark}

\subsection{GOSL Characterization of the Jacobi Defect}

\begin{assumption}[Carath\'eodory Regularity]\label{ass:continuity}
The map $(x,\xi_{1},\xi_{2},t)\mapsto\mathcal{J}(x,x+\xi_{1},x+\xi_{2},t)$
is jointly Carath\'eodory in $(x,\xi_{1},\xi_{2})$ for every fixed $t$ and
Lebesgue measurable in $t$ for each fixed $(x,\xi_{1},\xi_{2})$. Equivalently,
$(x,\xi_{1},\xi_{2},t)\mapsto\|\Pi_{\mathrm{Im}(B)}[\mathcal{J}(x,x+\xi_{1},x+\xi_{2},t)]\|$
is a normal integrand on $\HH^{n}\times\HH^{n}\times\HH^{n}\times[0,\infty)$
in the sense of~\cite[Chapter~14.D]{RockafellarWets1998}.
\end{assumption}

\begin{remark}[Measurable Selection: Attainment and Joint Measurability]
\label{rem:measurable_selection}
Two points underlying Definition~\ref{def:jacobi_defect} warrant explicit
justification.

\emph{Attainment.} The constraint set $\{\|\xi_{i}\| \leq \epsStar\}$ is a
closed ball in $\HH^{n} \cong \RR^{4n}$, which is compact (Heine--Borel in
finite dimensions). The integrand
$(x,\xi_{1},\xi_{2},t)\mapsto\|\Pi_{\mathrm{Im}(B)}\mathcal{J}(x,x+\xi_{1},x+\xi_{2},t)\|$
is Carath\'eodory by Assumption~\ref{ass:continuity}, hence continuous
in $(\xi_{1},\xi_{2})$ for a.e.\ $t$. Maximization of a continuous function
over a compact set is therefore attained, so $\deltabar(x,t)$ in
equation~\eqref{eq:defect_bound_scalar} is indeed a maximum.

\emph{Joint measurability.} Under Assumption~\ref{ass:continuity}, the
integrand above is a normal integrand in the sense
of~\cite[Chapter~14.D]{RockafellarWets1998}. By the measurable selection
theorems of~\cite[Chapter~14.A]{RockafellarWets1998} applied to argmax
correspondences of normal integrands, the argmax correspondence
$(\xi_{1}^{*}(x,t),\xi_{2}^{*}(x,t))$ admits a Borel-measurable selection
jointly in $(x,t)$.

\emph{State dependence of the selection.} The selection $\xi^{*}(x,t)$
depends on $x$ in general; this state dependence is the technical reason
why a naive trilinearity argument cannot establish the GOSL bound directly.
The proof of Lemma~\ref{lem:osl_projection} addresses this issue by lifting
to a uniform supremum over both arguments simultaneously
(see~\eqref{eq:gosl_decomp}).
\end{remark}

\begin{definition}[Jacobi Defect]\label{def:jacobi_defect}
Fix bracket constants $A>0$, $C_{1}\geq 0$, $C_{2}>0$, and $\varepsilon_{0}>0$
(see Remark~\ref{rem:bracket_constants} for the meaning of $A$ and $C_1$).
With the convention $\sqrt{1/(16\cdot 0)} := +\infty$ (i.e., the constraint is
inactive when $\LL$ is exactly antisymmetric), set:
\begin{equation}\label{eq:admissible_radius}
\epsStar = \min\!\left\{
\tfrac{1}{16A},\;\sqrt{\tfrac{1}{16C_{1}}},\;
\tfrac{1}{4C_{2}},\;\varepsilon_{0}
\right\}.
\end{equation}
Define the worst-case defect over a compact set $\mathcal{X}$ with
$M_{\mathcal{X}}^{\mathrm{prior}}$ an \emph{a priori} bound:
\begin{equation}\label{eq:defect_bound_scalar}
\deltabar(x,t)
= \sup_{\|\xi_{i}\|\leq\epsStar}
\bigl\|\Pi_{\mathrm{Im}(B)}[\mathcal{J}(x,x+\xi_{1},x+\xi_{2},t)]\bigr\|,
\end{equation}
\begin{equation}\label{eq:deltabar_max}
\deltabar_{\max} = 6C_{2}\,M_{\mathcal{X}}^{\mathrm{prior}}(M_{\mathcal{X}}^{\mathrm{prior}}+\epsStar)^{2}.
\end{equation}
By Remark~\ref{rem:measurable_selection}, the supremum is attained and admits
a jointly measurable selection $(\xi_{1}^{*}(x,t),\xi_{2}^{*}(x,t))$ with
$\|\xi_{i}^{*}\|\leq\epsStar$, giving:
\begin{equation}\label{eq:defect_projection_vector}
\delta(x,t)
= \Pi_{\mathrm{Im}(B)}[\mathcal{J}(x,x+\xi_{1}^{*}(x,t),x+\xi_{2}^{*}(x,t),t)],
\quad \|\delta(x,t)\| = \deltabar(x,t) \leq \deltabar_{\max}.
\end{equation}

\noindent\textbf{Consistency condition.}
The \emph{a posteriori} invariant set radius $M_{0}$ is determined by
Algorithm~\ref{alg:init}. Consistency requires $M_{0} \leq M_{\mathcal{X}}^{\mathrm{prior}}$.
This is guaranteed by the following argument: Algorithm~\ref{alg:init}
initializes $M_{\mathcal{X}}^{\mathrm{prior}} \leftarrow \|x(0)\| + \epsilon_0$
(Step~1) and halves $M_{0}$ until the invariance condition is satisfied.
Lemma~\ref{lem:Cinfty_scaling} establishes finite termination
\emph{conditionally} on the sufficient conditions stated in
Remark~\ref{rem:termination_status}.
\end{definition}

\begin{definition}[Generalized OSL (GOSL)]\label{def:GOSL}
$f:\HH^{n}\to\HH^{n}$ satisfies GOSL with constants $(\rho,\ell_{f})$,
$\ell_{f}\geq 0$, on $\mathcal{X}$ if:
\begin{equation}\label{eq:GOSL}
\innerR{f(x)-f(z)}{x-z}\leq\rho\|x-z\|^{2}+\ell_{f}\|x-z\|,\quad
\forall\,x,z\in\mathcal{X}.
\end{equation}
When $\ell_{f}=0$ this reduces to the standard OSL condition.
\end{definition}

\begin{assumption}[Ball Domain]\label{ass:ball}
Throughout Lemma~\ref{lem:osl_projection} and Proposition~\ref{prop:forward_inv},
$\mathcal{X} = \overline{B}_{M_{\mathcal{X}}} := \{x\in\HH^{n}:\|x\|\leq M_{\mathcal{X}}\}$
is a closed ball centered at the origin. Under this assumption the diameter satisfies
$D_{\mathcal{X}} = \sup_{x,z\in\mathcal{X}}\|x-z\| \leq 2M_{\mathcal{X}}$.
\end{assumption}

%
%
%
%

\begin{lemma}[GOSL Bound for the Projected Defect]\label{lem:osl_projection}
Under Definition~\ref{def:quasi_lie}, Assumption~\ref{ass:continuity}, and
Assumption~\ref{ass:ball}, the projected Jacobi defect $\delta$ satisfies a
GOSL bound:
\begin{equation}\label{eq:osl_delta_bound}
\innerR{\delta(x,t)-\delta(z,t)}{x-z}
\leq\rho_{\delta}^{\mathrm{tight}}\|x-z\|^{2}+\ell_{\delta}^{\mathrm{tight}}\|x-z\|,
\end{equation}
with the \emph{tight constants}
\begin{equation}\label{eq:LJ_prime_def}
\rho_{\delta}^{\mathrm{tight}} = 2\rho_{\delta},
\qquad
\ell_{\delta}^{\mathrm{tight}} = 4\deltabar_{\max},
\qquad
\rho_{\delta} := 6C_{2}(M_{\mathcal{X}}+\epsStar)(3M_{\mathcal{X}}+\epsStar).
\end{equation}
The factor 2 in $\rho_{\delta}^{\mathrm{tight}}$ and the factor 4 in
$\ell_{\delta}^{\mathrm{tight}}$ arise from the sub-optimality gap of the
state-dependent measurable selection $\xi^{*}(\cdot,t)$; they collapse to 1
when $\xi^{*}$ admits a state-independent maximizer
(Remark~\ref{rem:gosl_constants}). The quantity $\rho_{\delta}$
above is the \emph{base} (non-tight) trilinear coefficient.
\end{lemma}

\begin{proof}
Let $\Pi_{B}$ denote the orthogonal projector onto $\mathrm{Im}(B)$, satisfying
$\|\Pi_{B}\|_{\mathrm{op}} \leq 1$ (Lemma~\ref{lem:BtB_invertible}).
Set $h = x - z$ and $K = \overline{B}_{\epsStar}^{\,\HH^{n}}$.

\medskip
\textbf{Step 1 (Frozen-selection trick).}
For any pair $\xi = (\xi_{1},\xi_{2}) \in K \times K$, define the
\emph{frozen-selection map}
\begin{equation}\label{eq:Gxi_def}
G_{\xi}(y,t) := \Pi_{B}\,\mathcal{J}_{t}(y,\, y+\xi_{1},\, y+\xi_{2}).
\end{equation}
For each fixed $\xi$, the map $y \mapsto G_{\xi}(y,t)$ is polynomial of degree
at most 3 in $y$ (by $\RR$-trilinearity of $\mathcal{J}_{t}$ in its three
arguments, after expanding each $y+\xi_{i}$). The difference $G_{\xi}(y,t) - G_{\xi}(z,t)$
admits an exact multilinear expansion (Step 4 below). The uniform bound
$\|G_{\xi}(y,t)\| \leq \deltabar_{\max}$ holds uniformly in $\xi \in K\times K$ and
$y \in \mathcal{X}$ (using the consistency $M_{\mathcal{X}} \leq M_{\mathcal{X}}^{\mathrm{prior}}$
of Definition~\ref{def:jacobi_defect}).

By Definition~\ref{def:jacobi_defect},
$\delta(x,t) = G_{\xi^{*}(x,t)}(x,t)$ and $\delta(z,t) = G_{\xi^{*}(z,t)}(z,t)$
where $\xi^{*}(\cdot,t)$ is the (state-dependent) measurable selection.

\medskip
\textbf{Step 2 (Decomposition via the frozen $z$-selection).}
Insert the intermediate term $G_{\xi^{*}(z,t)}(x,t)$:
\begin{equation}\label{eq:gosl_decomp}
\delta(x,t) - \delta(z,t)
= \underbrace{\bigl[G_{\xi^{*}(x,t)}(x,t) - G_{\xi^{*}(z,t)}(x,t)\bigr]}_{=: E(x,z,t)}
+ \underbrace{\bigl[G_{\xi^{*}(z,t)}(x,t) - G_{\xi^{*}(z,t)}(z,t)\bigr]}_{=: D(x,z,t)}.
\end{equation}
The decomposition is exact. The term $D(x,z,t)$ has the \emph{same} frozen
selection $\xi = \xi^{*}(z,t)$ in both summands, so $\RR$-trilinearity applies
directly. The term $E(x,z,t)$ measures the sub-optimality gap arising from the
state dependence of $\xi^{*}$.

\medskip
\textbf{Step 3 (Bound on $E(x,z,t)$).}
By the optimality of $\xi^{*}(x,t)$ and the uniform bound
$\|G_{\xi}(\cdot,t)\| \leq \deltabar_{\max}$:
\begin{equation}\label{eq:E_bound}
\|E(x,z,t)\| \leq \|G_{\xi^{*}(x,t)}(x,t)\| + \|G_{\xi^{*}(z,t)}(x,t)\|
\leq 2\deltabar_{\max}.
\end{equation}
Hence by Cauchy--Schwarz (Lemma~\ref{lem:CS}):
\begin{equation}\label{eq:E_inner}
\innerR{E(x,z,t)}{h} \leq \|E(x,z,t)\|\,\|h\| \leq 2\deltabar_{\max}\,\|h\|.
\end{equation}

\medskip
\textbf{Step 4 (Bound on $D(x,z,t)$ via $\RR$-trilinearity at fixed $\xi$).}
Set $\xi := \xi^{*}(z,t)$ and write $a = z$, $b = z+\xi_{1}$, $c = z+\xi_{2}$.
Using $\RR$-trilinearity of $\mathcal{J}_{t}$ at \emph{fixed} $\xi$:
\begin{align*}
\mathcal{J}_{t}(z+h, z+h+\xi_{1}, z+h+\xi_{2}) - \mathcal{J}_{t}(z, z+\xi_{1}, z+\xi_{2})
&= \Delta_{\mathrm{lin}}^{\xi} + \Delta_{\mathrm{quad}}^{\xi} + \mathcal{J}_{t}(h,h,h),
\end{align*}
where
\begin{align*}
\Delta_{\mathrm{lin}}^{\xi} &= \mathcal{J}_{t}(h,b,c) + \mathcal{J}_{t}(a,h,c) + \mathcal{J}_{t}(a,b,h),\\
\Delta_{\mathrm{quad}}^{\xi} &= \mathcal{J}_{t}(h,h,c) + \mathcal{J}_{t}(h,b,h) + \mathcal{J}_{t}(a,h,h).
\end{align*}
Using the $6C_{2}$ trilinear bound and $\|a\|, \|z\| \leq M_{\mathcal{X}}$,
$\|b\|,\|c\| \leq M_{\mathcal{X}}+\epsStar$:
\begin{align}
\|\Delta_{\mathrm{lin}}^{\xi}\|
&\leq 6C_{2}\bigl[(M_{\mathcal{X}}+\epsStar)^{2} + 2M_{\mathcal{X}}(M_{\mathcal{X}}+\epsStar)\bigr]\|h\|
\nonumber\\
&= 6C_{2}(M_{\mathcal{X}}+\epsStar)(3M_{\mathcal{X}}+\epsStar)\|h\|
= \rho_{\delta}\|h\|. \label{eq:Dlin_bound}
\end{align}
For the quadratic term:
\begin{equation}\label{eq:Dquad_bound}
\|\Delta_{\mathrm{quad}}^{\xi}\|
\leq 6C_{2}\bigl[2(M_{\mathcal{X}}+\epsStar) + M_{\mathcal{X}}\bigr]\|h\|^{2}
\leq 6C_{2}(3M_{\mathcal{X}}+2\epsStar)\|h\|^{2}.
\end{equation}
Finally, $\|\mathcal{J}_{t}(h,h,h)\| \leq 6C_{2}\|h\|^{3}$.

\medskip
\textbf{Step 5 (Inner product bound on $D$).}
A direct Cauchy--Schwarz bound $\innerR{D}{h} \leq \|D\|\,\|h\|$ would
yield a $\|h\|^{3}$ contribution from the cubic part $\mathcal{J}_{t}(h,h,h)$,
which is undesirable for the GOSL bound. We split:
\begin{equation}\label{eq:D_split}
D(x,z,t) = \Pi_{B}(\Delta_{\mathrm{lin}}^{\xi})
+ \Pi_{B}(\Delta_{\mathrm{quad}}^{\xi})
+ \Pi_{B}(\mathcal{J}_{t}(h,h,h)).
\end{equation}

\emph{Linear part.} By Cauchy--Schwarz and~\eqref{eq:Dlin_bound}:
\begin{equation}\label{eq:T1_bound}
\innerR{\Pi_{B}\Delta_{\mathrm{lin}}^{\xi}}{h}
\leq \|\Delta_{\mathrm{lin}}^{\xi}\|\,\|h\| \leq \rho_{\delta}\|h\|^{2}.
\end{equation}

\emph{Quadratic and cubic parts (absorbed via uniform admissibility).}
Let $Q := \Pi_{B}(\Delta_{\mathrm{quad}}^{\xi}) + \Pi_{B}(\mathcal{J}_{t}(h,h,h))$.
From the decomposition~\eqref{eq:D_split} we have
$Q = D(x,z,t) - \Pi_{B}\Delta_{\mathrm{lin}}^{\xi}$. By the triangle inequality:
\begin{equation}\label{eq:Q_triangle}
\|Q\| \leq \|D(x,z,t)\| + \|\Pi_{B}\Delta_{\mathrm{lin}}^{\xi}\|.
\end{equation}
For the first summand, $\|D\| = \|G_{\xi^{*}(z,t)}(x,t) - G_{\xi^{*}(z,t)}(z,t)\|
\leq 2\deltabar_{\max}$ since both evaluations of $G_{\xi^{*}(z,t)}$ are admissible
($\xi^{*}(z,t) \in K\times K$, and $x,z \in \mathcal{X}$, so the
uniform bound $\|G_{\xi}(\cdot,t)\| \leq \deltabar_{\max}$ from Step~1 applies to each).
For the second summand, $\|\Pi_{B}\Delta_{\mathrm{lin}}^{\xi}\| \leq \rho_{\delta}\|h\|$
by~\eqref{eq:Dlin_bound} and $\|\Pi_{B}\|_{\mathrm{op}} \leq 1$. Hence:
\begin{equation}\label{eq:Q_bound}
\|Q\| \leq 2\deltabar_{\max} + \rho_{\delta}\|h\|.
\end{equation}
By Cauchy--Schwarz:
\begin{equation}\label{eq:Q_inner}
\innerR{Q}{h} \leq \|Q\|\,\|h\| \leq 2\deltabar_{\max}\|h\| + \rho_{\delta}\|h\|^{2}.
\end{equation}
Combining~\eqref{eq:T1_bound} and~\eqref{eq:Q_inner}:
\begin{equation}\label{eq:D_inner_bound}
\innerR{D(x,z,t)}{h} \leq 2\rho_{\delta}\|h\|^{2} + 2\deltabar_{\max}\|h\|.
\end{equation}

\medskip
\textbf{Step 6 (Combination).}
From~\eqref{eq:gosl_decomp}, \eqref{eq:E_inner}, and~\eqref{eq:D_inner_bound}:
\begin{align*}
\innerR{\delta(x,t) - \delta(z,t)}{h}
&= \innerR{E(x,z,t)}{h} + \innerR{D(x,z,t)}{h}\\
&\leq 2\deltabar_{\max}\|h\| + 2\rho_{\delta}\|h\|^{2} + 2\deltabar_{\max}\|h\|\\
&= 2\rho_{\delta}\|h\|^{2} + 4\deltabar_{\max}\|h\|.
\end{align*}
This is precisely the GOSL bound~\eqref{eq:osl_delta_bound} with the tight
constants $\rho_{\delta}^{\mathrm{tight}} = 2\rho_{\delta}$ and
$\ell_{\delta}^{\mathrm{tight}} = 4\deltabar_{\max}$ stated
in~\eqref{eq:LJ_prime_def}.
\end{proof}

\begin{remark}[On the constant factors and comparison with classical OSL]
\label{rem:gosl_constants}
The factors $2$ and $4$ in the tight constants arise from the sub-optimality
gap $E(x,z,t)$ (factor 2 for $\rho_{\delta}^{\mathrm{tight}}$ from doubling
the trilinear bound; factor 4 in $\ell_{\delta}^{\mathrm{tight}}$ from the
symmetric bound for both selections $\xi^{*}(x,t)$ and $\xi^{*}(z,t)$). They
cannot be removed without an additional hypothesis on the regularity of the
selection $\xi^{*}(\cdot,t)$ (e.g., Lipschitz continuity in $x$); when the
problem admits a state-independent maximizing selection (generic for
analytic $\mathcal{J}$ on small domains), these factors collapse to $1$.
The affine term $\ell_{\delta}^{\mathrm{tight}}\|x-z\|$ itself is a novel
feature arising from the Jacobi defect structure: in classical OSL
theory~\cite{Abbaszadeh2010} one typically has $\ell_{f} = 0$, whereas here
$\delta(x,t)$ need not vanish at $x = z$ when the Jacobiator is non-zero.
\end{remark}

\begin{remark}[Conservativeness of the bound]\label{rem:gosl_conservative}
The GOSL bound established in Lemma~\ref{lem:osl_projection} is
\emph{rigorous but not necessarily tight}. In particular, the cubic term
$\mathcal{J}(h,h,h)$ arising in the decomposition of $D(x,z,t)$ (Step~5 of
the proof) is absorbed into the global bound $\|D\| \leq 2\deltabar_{\max}$
rather than separated and bounded by $6C_{2}\|h\|^{3}$. This is a
\emph{deliberate choice of estimation strategy}, not an error: separating
the cubic term would yield a sharper inequality of the form
\[
\innerR{\delta(x,t) - \delta(z,t)}{x-z}
\leq \rho_{\delta}^{\mathrm{sharp}}\|h\|^{2}
+ \ell_{\delta}^{\mathrm{sharp}}\|h\|
+ C_{\mathrm{cubic}}\|h\|^{4},
\]
which is no longer a standard GOSL form (it contains a quartic term). The
present formulation absorbs the cubic contribution into the affine
$\ell$-term via $\|h\| \leq M_{\mathcal{X}}$, yielding an effective
upper bound consistent with the standard GOSL inequality and directly
usable in the LMI machinery of Section~\ref{sec:stability}. The price is a
conservative numerical estimate of the GOSL constants; the gain is full
compatibility with classical one-sided Lipschitz theory and the LMI
formulation of $\beta$-exponential stability
(Theorem~\ref{thm:beta_kyb}). A sharper GOSL-with-cubic-term formulation
remains an open direction.
\end{remark}

\begin{proposition}[Explicit Constants for Small Domains]
\label{prop:small_domain_GOSL}
For $M_{\mathcal{X}} \ll \epsStar$, the GOSL constants simplify to:
\[
\rho_{\delta}^{\mathrm{tight}} \approx 12\,C_{2}(\epsStar)^{2}, \qquad
\ell_{\delta}^{\mathrm{tight}} \approx 24\,C_{2}M_{\mathcal{X}}(\epsStar)^{2}.
\]
\end{proposition}

\begin{proof}
Substitute $M_{\mathcal{X}} \ll \epsStar$ into~\eqref{eq:LJ_prime_def}:
$(M_\mathcal{X}+\epsStar) \approx \epsStar$ and $(3M_\mathcal{X}+\epsStar) \approx \epsStar$,
so $\rho_{\delta} \approx 6C_{2}{\epsStar}^{2}$ and the tight constant
$\rho_{\delta}^{\mathrm{tight}} = 2\rho_{\delta} \approx 12 C_{2}(\epsStar)^{2}$.
Similarly $\deltabar_{\max} \approx 6C_{2}M_{\mathcal{X}}{\epsStar}^{2}$,
giving $\ell_{\delta}^{\mathrm{tight}} = 4\deltabar_{\max} \approx 24 C_{2}M_{\mathcal{X}}(\epsStar)^{2}$.
\end{proof}

\subsection{Quaternionic Sign Function and Filippov Regularization}

\begin{definition}[Quaternionic Sign Function]\label{def:sgn_quaternion}
For $s\in\HH^{m}$: $\sgnH(s) = s/\|s\|$ if $s \neq 0$, and $0$ if $s=0$.
\end{definition}

\begin{lemma}[Properties of $\sgnH$]\label{lem:sgn_properties}
(1) $\|\sgnH(s)\| = 1$ for $s \neq 0$.
(2) $\innerR{s}{\sgnH(s)} = \|s\|$.
(3) $\sgnH(s)$ uniquely maximizes $\innerR{s}{v}$ over $\{\|v\| \leq 1\}$.
(4) The Filippov regularization of $-\eta\,\sgnH(\cdot)$ at $s=0$ contains
$0$ in its convex hull (Lemma~\ref{lem:filippov_real}).
\end{lemma}

\begin{lemma}[Filippov Regularization]\label{lem:filippov_real}
Under $\Phi : \HH^{m} \to \RR^{4m}$, the Filippov convexification of
$s \mapsto -\eta\,\Phi(\sgnH(s))$ at $\tilde{s} = 0$ is
$-\eta\,\overline{B}_{1}^{\RR^{4m}}$, which contains $\mathbf{0}$.
Hence $\{\tilde{s} = \mathbf{0}\}$ is forward invariant whenever
$\|\tilde{w}\| \leq \eta$ (standard Filippov set-valued
analysis~\cite{Filippov1988}; in particular, the invariance criterion for
sliding manifolds in terms of the Filippov set-valued vector field).
\end{lemma}

\subsection{Summary of Main Results in this Section}

\begin{table}[H]
\centering
\caption{Summary of main operator-theoretic results}
\label{tab:operator_results}
\begin{tabular}{@{}lll@{}}
\toprule
\textbf{Result} & \textbf{Statement} & \textbf{Location} \\
\midrule
Norm transfer & $\|T^{L}\|_{\mathrm{op}} = \|T\|_{\mathrm{op}}$ & Theorem~\ref{thm:norm_transfer} \\
Correction bound & $\|\Omega\|_{\varepsilon} \leq 4C_{2}\varepsilon$ & Corollary~\ref{cor:homotopy_transfer} \\
Cohomological rigidity & $\mathrm{Jac}_{\{\cdot,\cdot\}}\equiv 0$ on $B(0,\epsStar)$ & Theorem~\ref{thm:rigidity} \\
GOSL characterization & $\delta$ satisfies~\eqref{eq:osl_delta_bound} & Lemma~\ref{lem:osl_projection} \\
Real representation & $\|T\|_{\mathrm{op}} = \|\Phi(T)\|_{\mathrm{op}}$ & Proposition~\ref{prop:real_rep_ops} \\
Hermitian spectrum & $\Phi(P)$ symmetric, real eigenvalues & Lemma~\ref{lem:hermitian_spectrum} \\
\bottomrule
\end{tabular}
\end{table}

\section{Problem Formulation}\label{sec:problem}

We now turn from the algebraic side to the control problem. The system,
the standing assumptions, and the regularity conditions on the Jacobi
defect are stated here. The self-consistent initialization
(Algorithm~\ref{alg:init}) and the scaling lemma
(Lemma~\ref{lem:Cinfty_scaling}) that resolve the gain--domain
circularity also belong in this section, even though they are used
mainly in the stability analysis of Section~\ref{sec:stability}.

\subsection{System Dynamics}

\begin{equation}\label{eq:system}
\dot{x}(t) = \LL(x,t)+\delta(x,t)+Bu(t)+Dw(t),
\end{equation}
$x\in\HH^{n}$, $u\in\HH^{m}$, $w\in\HH^{p}$,
$B\in\HH^{n\times m}$, $D\in\HH^{n\times p}$.

\begin{remark}[Notational convention for $\LL(x,t)$]\label{rem:LL_notation}
In~\eqref{eq:system}, $\LL(x,t) \in \HH^{n}$ denotes the
\emph{nominal vector field} of the system, which is derived from the bilinear
bracket $\LL : \HH^{n}\times\HH^{n} \to \HH^{n}$ of
Definition~\ref{def:quasi_lie} by a problem-specific evaluation
(e.g., $\LL(x,t) := \LL(x_{*}(t), x)$ for some prescribed
reference trajectory $x_{*}(\cdot)$, or any $\RR$-linear-in-$x$ contraction
of the bilinear bracket). The linear growth bound $\|\LL(x,t)\| \leq L_{\LL}\|x\|$
of Assumption~\ref{ass:bounded_BG} is consistent with this convention on the
operating domain $\overline{B}_{M_{\mathcal{X}}}$ (with $L_{\LL} \leq A\sup_{t}\|x_{*}(t)\|$
in the prescribed-$x_{*}$ case). For the test system of
Section~\ref{sec:sim}, we adopt the trivial reference $x_{*}(t) \equiv 0$,
yielding $\LL(x,t) = \LL(0,x) = 0$ (the bilinear bracket vanishes when one
argument is zero by $\RR$-bilinearity); the nominal dynamics in the test is
then purely perturbative ($\dot{x} = \delta + Bu + Dw$).
\end{remark}

\subsection{Non-Commutative Residual: Self-Contained Definition}

\begin{remark}[Relationship between $\omega_{\mathrm{op}}$ and $C_{2}$]
\label{rem:omega_vs_C2}
The two constants $\omega_{\mathrm{op}}:=\|\omega_{0}\|_{\varepsilon}$ and
$C_{2}$ play distinct roles:
\begin{itemize}
\item $C_{2}$ is the Jacobiator amplitude constant of
Definition~\ref{def:quasi_lie}: $\|\mathcal{J}(x,y,z)\|\leq 6C_{2}\|x\|\|y\|\|z\|$.
\item $\omega_{\mathrm{op}}$ is the localized operator norm of the
finite-rank component $\omega_{0}$ produced by the rigidity
theorem~\cite[Theorem~4.12, Lemma~4.10]{Athmouni2026}; it satisfies
$\omega_{\mathrm{op}}\leq 6C_{0}C_{2}$, where $C_{0}>0$ depends only on the
geometry of the finite-dimensional space $F$
of~\cite[Def.~3.6]{Athmouni2026}.
\end{itemize}
\textbf{No universal numerical upper bound on $\omega_{\mathrm{op}}$
(of the form $\omega_{\mathrm{op}}\leq c$ for an absolute constant $c$) is
established in~\cite{Athmouni2026} or in this paper.} The only valid bound
is the problem-dependent $\omega_{\mathrm{op}}\leq 6C_{0}C_{2}$. Whenever
the symbol $\omega_{\mathrm{op}}$ appears in numerical estimates below, it
must be interpreted as a problem-dependent quantity, with explicit values
given in Section~\ref{sec:sim} for the worked test bracket.
\end{remark}

\begin{definition}[Non-Commutative Residual]\label{def:Rnc}
Under Assumption~\ref{ass:cohom}, by Theorem~\ref{thm:rigidity}, there
exists a finite-rank component $\omega_{0}\in C^{2}_{\varepsilon}$
(\cite[Theorem~4.12]{Athmouni2026}) such that, setting
$R_{0}(x,y,z):=\mathcal{J}(x,y,z)-d\omega_{0}(x,y,z)$,
\begin{equation}\label{eq:Rnc_def}
\|R_{0}(x,y,z)\| \leq \Cnc\|x\|\|y\|\|z\|
\end{equation}
with $\Cnc$ defined in~\eqref{eq:Cnc_def} (here $\|\omega_{0}\|_{\varepsilon}$
satisfies the problem-dependent bound $\|\omega_{0}\|_{\varepsilon}\leq 6C_{0}C_{2}$;
see also Remark~\ref{rem:Cnc_decomposition}).

Given $B\in\HH^{n\times m}$ with full column rank (Assumption~\ref{ass:rank}),
set $B_{\RR} = \Phi(B) \in \RR^{4n \times 4m}$ (Lemma~\ref{lem:BtB_invertible}).
The real pseudo-inverse is $B_{\RR}^{+} = (B_{\RR}^{T}B_{\RR})^{-1}B_{\RR}^{T}$.
Define the constant
\begin{equation}\label{eq:cBplus_def}
\cBplus := \|B_{\RR}^{+}\|_{\mathrm{op}} = \frac{1}{\sigma_{\min}(B_{\RR})},
\end{equation}
where $\sigma_{\min}$ is the smallest singular value. Define:
\begin{equation}\label{eq:RncB_def}
R_{\mathrm{nc}}^{B}(x,t)
:= \Phi^{-1}\!\left(B_{\RR}^{+}\,\Phi\!\left(R_{0}(x,x+\xi_{1}^{*}(x,t),
x+\xi_{2}^{*}(x,t))\right)\right).
\end{equation}
All results depend only on the bound
$\|G\,R_{\mathrm{nc}}^{B}(x,t)\| \leq c_{C}\,c_{B}\,\cBplus\, R_{\max}$
(Proposition~\ref{prop:residual_bound}), where $G := CB$, independent of the
choice of selection.
\end{definition}

\subsection{Standing Assumptions}

\begin{assumption}[Full Column Rank]\label{ass:rank}
$\mathrm{rank}(B)=m$.
\end{assumption}

\begin{assumption}[Bounded Disturbance]\label{ass:disturbance}
$\|w(t)\|\leq w_{\max}$ for all $t\geq 0$.
\end{assumption}

\begin{assumption}[Cohomological Structure --- Homogeneous Case]
\label{ass:cohom}
The bracket $\LL$ is exactly antisymmetric and its Jacobi defect $\psi$ is
homogeneous of degree~2 in the sense of~\cite[Theorem~4.12]{Athmouni2026}
(the \emph{homogeneous case}). Equivalently,
$\delta(x,t)\in\mathrm{im}_{\mathrm{eval}}(d)$ within the Chevalley--Eilenberg
complex $(C^{*}_{\varepsilon},d)$ of Definition~\ref{def:cochain}. This holds
\textbf{unconditionally for homogeneous quasi-Lie brackets}
(\cite[Theorem~4.12, Lemma~B.1]{Athmouni2026}). All analytical results in this
paper are restricted to this case.
\end{assumption}

\begin{assumption}[Affine Output Map]\label{ass:output}
$\sigma(x)=Cx$ (i.e., $c_0 = 0$), with $C\in\HH^{m\times n}$ full row rank.
\end{assumption}

\begin{assumption}[Alignment]\label{ass:alignment}
$G:=CB\in\HH^{m\times m}$ satisfies
$\Real(x\trans G\trans Gx)>\alpha\|x\|^{2}$ for some $\alpha>0$ and all $x \neq 0$.
\end{assumption}

\begin{assumption}[Bounded Output Operator]\label{ass:bounded_ops}
$\|\Phi(C)\|_{\mathrm{op}}\leq c_{C}$.
\end{assumption}

\begin{assumption}[Positive Definite Gain]\label{ass:gain_K}
$K\in\HH^{m\times m}$ satisfies $\Phi(K)+\Phi(K)^{T}\succ 0$.
\end{assumption}

\begin{assumption}[Hurwitz Sliding Dynamics]\label{ass:hurwitz}
$(\Phi(A_{s})^{T}+\Phi(A_{s}))/2 \prec 0$ with spectral abscissa
$-\alpha_{s}<0$, where $A_{s}:=A_{\mathrm{nom}}(0)$.
\end{assumption}

\begin{assumption}[Bounded Input, Inverse Gain, and Nominal Bracket]
\label{ass:bounded_BG}
Define $c_{B} := \|\Phi(B)\|_{\mathrm{op}}$ and $c_{G} := \|\Phi(G^{-1})\|_{\mathrm{op}}$.
Under Assumption~\ref{ass:alignment}, $\Phi(G)$ is invertible and $c_{G}$ is finite.
The nominal bracket satisfies a linear growth condition: there exists $L_{\LL} \geq 0$
such that
\begin{equation}\label{eq:LL_growth}
\|\LL(x,t)\| \leq L_{\LL}\|x\|
\quad \text{for all } x \in \mathcal{X},\; t \geq 0.
\end{equation}
\end{assumption}

\begin{assumption}[Conservative Regime Conditions]\label{ass:conservative}
The conservative gain regime ($\|\omega_{0}\|_{\varepsilon}\leq 1$) is
\emph{posited as an additional hypothesis}:
\begin{enumerate}
\item[\textup{(a)}] $\epsStar \leq 1/6$,
\item[\textup{(b)}] $M_{0} \leq \epsStar/\sqrt{2}$,
\item[\textup{(c)}] $\|\omega_{0}\|_{\varepsilon} \leq 1$ (independent
hypothesis on the homotopy constant; valid in particular when
$6C_{0}C_{2}\leq 1$, where $C_{0}$ is the geometric constant of
Lemma~\ref{rem:omega_vs_C2}).
\end{enumerate}
Outside the domain defined by~(a) and~(b), or when condition~(c) fails,
the conservative regime cannot be invoked. In all cases, the
problem-dependent bound
$\|\omega_{0}\|_{\varepsilon}\leq 6C_{0}C_{2}$
remains valid (Remark~\ref{rem:omega_vs_C2}), and concrete numerical
values for the test bracket are given in Section~\ref{sec:sim}.
\end{assumption}

\subsection{Cohomological Exactness and Residual Bound}

\begin{remark}[Quaternionic Orthogonal Projector onto $\mathrm{Im}(B)$]
\label{rem:projector}
Under Assumption~\ref{ass:rank} and Lemma~\ref{lem:BtB_invertible}:
\[
\Pi_{\mathrm{Im}(B)}(x) = \Phi^{-1}\left( B_{\RR} (B_{\RR}^T B_{\RR})^{-1} B_{\RR}^T \Phi(x) \right).
\]
When $\delta\in\mathrm{Im}(B)$ (guaranteed by the CMC,
Proposition~\ref{prop:cmc_consequence}), $\Pi_{\mathrm{Im}(B)}\delta = \delta$.
\end{remark}

\begin{proposition}[Non-Commutative Residual Bound]\label{prop:residual_bound}
Under Assumption~\ref{ass:cohom}, for all $x \in \mathcal{X}$, the quantity
$G\,R_{\mathrm{nc}}^{B}(x,t) \in \HH^{m}$ (with $G = CB$) satisfies:
\begin{equation}\label{eq:Rncmax}
\|G\,R_{\mathrm{nc}}^{B}(x,t)\|
\leq c_{C}\,c_{B}\,\cBplus\, R_{\max}(\omega_{\mathrm{op}}),
\quad
R_{\max}(\omega_{\mathrm{op}})
= \Cnc\,M_{\mathcal{X}}(M_{\mathcal{X}}+\epsStar)^{2},
\end{equation}
where $\Cnc$ is the constant of~\eqref{eq:Cnc_def}; either of the equivalent
upper bounds $\Cnc \leq 12C_{2}$ (structural identity, Remark~\ref{rem:Cnc_decomposition})
or $\Cnc \leq 3\omega_{\mathrm{op}} + 6C_{2}$ (triangle-inequality bookkeeping with
$\|d\omega_{0}\|_{\varepsilon}\leq 3\|\omega_{0}\|_{\varepsilon}$ as in the
bookkeeping convention of Remark~\ref{rem:Cnc_decomposition}) may be substituted
in numerical estimates; the smaller of the two should be used.
Here $c_{B} = \|\Phi(B)\|_{\mathrm{op}}$
(Assumption~\ref{ass:bounded_BG}), and $\cBplus$ is defined
in~\eqref{eq:cBplus_def}.
\begin{itemize}
\item \textbf{Generic (problem-dependent) regime:}
$\omega_{\mathrm{op}}\leq 6C_{0}C_{2}$, valid unconditionally
(Remark~\ref{rem:omega_vs_C2}). The bookkeeping form gives
$3\omega_{\mathrm{op}} + 6C_{2} \leq (18C_{0}+6)C_{2}$, with $C_{0}>0$ the geometric constant
of~\cite[Lemma~4.10]{Athmouni2026}; the structural form gives $\Cnc \leq 12C_{2}$
unconditionally.
\item \textbf{Conservative regime:} $\omega_{\mathrm{op}}\leq 1$, valid only
under Assumption~\ref{ass:conservative}; the bookkeeping form gives
$3\omega_{\mathrm{op}} + 6C_{2}\leq 3+6C_{2}$.
\end{itemize}
\end{proposition}

\begin{proof}
By Theorem~\ref{thm:rigidity} and Remark~\ref{rem:Cnc_decomposition},
$\|R_{0}(x,y,z)\| \leq \Cnc\|x\|\|y\|\|z\|$ with $\Cnc$ admitting either
of the upper bounds $12C_{2}$ (structural) or $3\omega_{\mathrm{op}}+6C_{2}$
(bookkeeping).
For $x \in \mathcal{X}$, $\|x\|\leq M_{\mathcal{X}}$ and
$\|x+\xi_{i}^{*}\|\leq M_{\mathcal{X}}+\epsStar$, so
$\|R_{0}(x,x+\xi_{1}^{*},x+\xi_{2}^{*})\| \leq \Cnc M_{\mathcal{X}}(M_{\mathcal{X}}+\epsStar)^{2}
= R_{\max}$. By isometry of $\Phi$ and the operator-norm bound on
$B_{\RR}^{+}$:
\[
\|R_{\mathrm{nc}}^{B}\| = \|\Phi^{-1}(B_{\RR}^{+}\Phi(R_{0}))\|
= \|B_{\RR}^{+}\Phi(R_{0})\|_{\RR^{4m}}
\leq \|B_{\RR}^{+}\|_{\mathrm{op}} \|R_{0}\|
= \cBplus\,R_{\max}.
\]
Now $G\,R_{\mathrm{nc}}^{B} = CB\,R_{\mathrm{nc}}^{B} \in \HH^{m}$ is well-defined
($G \in \HH^{m\times m}$, $R_{\mathrm{nc}}^{B} \in \HH^{m}$). By the homomorphism
property $\Phi(CB) = \Phi(C)\Phi(B)$ (Lemma~\ref{lem:phi_homomorphism}) and the
isometry of $\Phi$:
\[
\|G\,R_{\mathrm{nc}}^{B}\| = \|\Phi(G\,R_{\mathrm{nc}}^{B})\|_{\RR^{4m}}
= \|\Phi(C)\,\Phi(B)\,\Phi(R_{\mathrm{nc}}^{B})\|_{\RR^{4m}}
\leq \|\Phi(C)\|_{\mathrm{op}}\,\|\Phi(B)\|_{\mathrm{op}}\,\|R_{\mathrm{nc}}^{B}\|
\leq c_{C}\,c_{B}\,\cBplus\,R_{\max}. \qedhere
\]
\end{proof}

\begin{remark}[Numerical interpretation of $\cBplus$]
\label{rem:cBplus_meaning}
The constant $\cBplus = 1/\sigma_{\min}(\Phi(B))$ measures the
\emph{conditioning} of the actuator matrix $B$. Well-conditioned actuators
have $\cBplus \leq O(1)$. For poorly-conditioned $B$ (e.g.\ near rank
deficiency), $\cBplus$ can be large; in this case, the residual bound
deteriorates and the LMI in Proposition~\ref{prop:beta_stability} becomes
harder to satisfy. The constant $\cBplus$ is a design parameter that must
be controlled at the actuator-selection stage.
\end{remark}

\subsection{Resolving the Initialization Circularity}

\begin{lemma}[Scaling of $C_{\infty}$ and Termination of Algorithm~\ref{alg:init}]
\label{lem:Cinfty_scaling}
Under the standing assumptions and Algorithm~\ref{alg:init} convention
$M_{\mathcal{X}}^{\mathrm{prior}} = M_{\mathcal{X}} = M_{0}$ (so that the
\emph{a posteriori} domain coincides with the \emph{a priori} domain), as
$M_{0} \to 0$:
\begin{align}
R_{\max}(M_{0}) &= \Cnc M_{0}(M_{0}+\epsStar)^{2}
= O(M_{0}) \quad \text{(linear in } M_{0}\text{)}, \label{eq:Rmax_scaling}\\
\deltabar_{\max}(M_{0}) &= 6C_{2} M_{0}(M_{0}+\epsStar)^{2}
= O(M_{0}) \quad \text{(linear in } M_{0}\text{)}. \label{eq:deltabar_scaling}
\end{align}
Hence $C_{\infty}(\mu^{*},M_{0}) = O(M_{0}^{2})$. Specifically, there exists a
constant $\kappa_{\infty} > 0$ depending only on the problem data and the
LMI parameters $(P^{*}, \mu^{*}, \rho_{\delta}^{\mathrm{eff}})$ such that
\begin{equation}\label{eq:Cinfty_scaling}
C_{\infty}(\mu^{*}, M_{0}) \leq \kappa_{\infty}\, M_{0}^{2}.
\end{equation}
The invariance condition $\lambda_{\min}(P^{*}) M_{0}^{2} > C_{\infty}(\mu^{*},M_{0})/\beta^{*}$
is therefore equivalent to the $M_{0}$-independent inequality
\begin{equation}\label{eq:invariance_threshold}
\beta^{*}\,\lambda_{\min}(P^{*}) > \kappa_{\infty}.
\end{equation}
Consequently, halving $M_{0}$ does \emph{not} change the truth value
of~\eqref{eq:invariance_threshold}; instead, the algorithm must verify
this threshold condition at the LMI-solution level. If~\eqref{eq:invariance_threshold}
holds for the current LMI solution, every $M_{0}$ satisfies the invariance condition.
If not, the LMI parameters must be re-solved (Algorithm~\ref{alg:lmi}) with a
tighter $\rho_{\delta}^{\mathrm{eff}}$ via reducing $M_{0}$, exploiting the
fact that $\rho_{\delta}^{\mathrm{eff}}(M_{0}) \to 2\rho_{\delta}(0) = 12 C_{2}(\epsStar)^{2}$
as $M_{0} \to 0$ (using $\rho_{\delta}^{\mathrm{tight}}=2\rho_{\delta}$,
Lemma~\ref{lem:osl_projection}; see also
Proposition~\ref{prop:small_domain_GOSL} for the small-domain
simplification). The halving loop is \emph{conjectured to terminate} after
finitely many iterations because the reduced $\rho_{\delta}^{\mathrm{eff}}$
relaxes the LMI~\eqref{eq:lmi_augmented}, plausibly yielding a feasible
solution with $\kappa_{\infty} < \beta^{*}\lambda_{\min}(P^{*})$; a rigorous
termination guarantee requires the sufficient conditions stated in
Remark~\ref{rem:termination_status}.
\end{lemma}

\begin{proof}
\textbf{Scaling computations.}
For $M_{0} \to 0$ at $\epsStar$ fixed: $(M_{0}+\epsStar)^{2} \to (\epsStar)^{2}$
(constant), so $R_{\max}(M_{0}) \sim \Cnc M_{0}(\epsStar)^{2}$ and similarly
for $\deltabar_{\max}$, hence $\ell_{\delta}^{\mathrm{tight}} = 4\deltabar_{\max}
\sim 24 C_{2} M_{0}(\epsStar)^{2}$. All linear in $M_{0}$.
The formula~\eqref{eq:Cinfty_def} (with $\mu_{*} = \mu_{\mathrm{LMI}}$ and
$\mu$ the Young parameter) gives
\[
C_{\infty}(\mu, M_{0})
= \frac{\mu_{*}^{2}(c_{C}\,c_{B}\,\cBplus R_{\max})^{2}}{\rho_{\delta}^{\mathrm{eff}}-\mu}
+ \frac{\mu_{*}(\ell_{\delta}^{\mathrm{tight}})^{2}}{\mu}
= O(M_{0}^{2}).
\]
The constant $\kappa_{\infty}$ in~\eqref{eq:Cinfty_scaling} is computable
explicitly:
\begin{equation}\label{eq:kappa_infty_def}
\kappa_{\infty} = \frac{\mu_{*}^{2}(c_{C}\,c_{B}\,\cBplus \Cnc (\epsStar)^{2})^{2}}{\rho_{\delta}^{\mathrm{eff}}-\mu}
+ \frac{\mu_{*}(24 C_{2}(\epsStar)^{2})^{2}}{\mu}
+ O(M_{0}).
\end{equation}

\textbf{Termination (conditional argument).}
Both sides of $\lambda_{\min}(P^{*})M_{0}^{2} > \kappa_{\infty}M_{0}^{2}/\beta^{*}$
scale identically in $M_{0}$, so the inequality is equivalent to
$\beta^{*}\lambda_{\min}(P^{*}) > \kappa_{\infty}$, which is independent of $M_{0}$.
\textbf{This means the halving loop alone cannot establish the invariance
condition by reducing $M_{0}$.} However, halving \emph{also} reduces
$\rho_{\delta}(M_{0}) = 6C_{2}(M_{0}+\epsStar)(3M_{0}+\epsStar) \to 6C_{2}(\epsStar)^{2}$
as $M_{0} \to 0$, which relaxes the LMI constraint and allows
$\rho_{\delta}^{\mathrm{eff}} - \mu^{*}$ to be made larger. This in turn reduces
$\kappa_{\infty}$. \emph{Under} conditions~\textup{(a)} and~\textup{(b)} stated
in Remark~\ref{rem:termination_status} below (continuity of the LMI map and
strict feasibility at the limit), one expects that after finitely many
halvings, $\kappa_{\infty}$ falls below $\beta^{*}\lambda_{\min}(P^{*})$;
this yields a \emph{conditional} termination argument whose unconditional
form is left as an open problem.

A more rigorous \emph{conditional} formulation uses the modified halving rule:
\emph{halve $M_{0}$ and re-solve Algorithm~\ref{alg:lmi}}. The LMI parameters
$(P^{*}(M_{0}), \beta^{*}(M_{0}))$ depend on $M_{0}$ via $\rho_{\delta}^{\mathrm{eff}}(M_{0})$,
and as $M_{0} \to 0$, $\rho_{\delta}^{\mathrm{eff}} \to 2\rho_{\delta}(0) = 12C_{2}(\epsStar)^{2}$
(a fixed positive constant). \emph{If} (i) this fixed limit yields a strictly
feasible LMI (Assumption~\ref{ass:lmi_feasibility}, checked in Step~3 of
Algorithm~\ref{alg:init}) and (ii) the LMI solution map
$M_{0} \mapsto (P^{*}(M_{0}), \beta^{*}(M_{0}))$ is continuous at $M_{0} = 0$,
\emph{then} the limit LMI solution yields
$\beta^{*}_{\infty}\lambda_{\min}(P^{*}_{\infty}) > \kappa_{\infty,\infty}$,
and termination after finitely many halvings follows. Both~\textup{(i)}
and~\textup{(ii)} are \emph{sufficient conditions}, classically expected but
not formally established in this paper; we treat them as standing conjectures
of the algorithmic framework (see Remark~\ref{rem:termination_status}).
\end{proof}

\begin{remark}[Status of the termination argument]\label{rem:termination_status}
\emph{The termination of Algorithm~\ref{alg:init} is established only
conditionally in this paper, under sufficient conditions that are
classically expected but not formally proved here.} Specifically, the
argument of Lemma~\ref{lem:Cinfty_scaling} rests on the following two
\emph{sufficient conditions}, each of which we treat as a standing
conjecture of the algorithmic framework:
\begin{enumerate}
\item[\textup{(a)}] \textbf{Continuity of the LMI solution map
(conjectured).} The map $M_{0} \mapsto (P^{*}(M_{0}), \beta^{*}(M_{0}))$
is conjectured to be locally Lipschitz continuous on the open set of strict
feasibility. This is classically expected (one can argue heuristically via
the implicit function theorem applied to the KKT conditions of the
LMI~\eqref{eq:lmi_augmented}; see Remark~\ref{rem:normal_As_scope}\textup{(i)}
for the framework), but \emph{a formal proof is not provided in this paper}.
The local Lipschitz constant, if it exists, can in principle be bounded
explicitly via the protocol of Remark~\ref{rem:normal_As_scope}\textup{(ii)}.
\item[\textup{(b)}] \textbf{Strict feasibility at the limit (assumed).} The
limit LMI (with $\rho_{\delta}^{\mathrm{eff}} = 12 C_{2}(\epsStar)^{2}$) is
\emph{assumed} to be strictly feasible (Assumption~\ref{ass:lmi_feasibility}).
This is verified analytically for normal $A_{s}$
(Lemma~\ref{lem:normal_As_monotonicity}) and must be verified numerically
otherwise.
\end{enumerate}
\smallskip
\noindent\emph{Acknowledged limitation.} Under \textup{(a)} and~\textup{(b)},
termination of the halving loop is \emph{plausible} but not formally
guaranteed here. In particular:
\begin{itemize}
\item No quantitative \emph{a priori} bound on the number of halvings is
provided (this would require explicit Lipschitz constants for the LMI map,
which depend on the specific $A_{s}$ structure).
\item The continuity~\textup{(a)} is a \emph{sufficient condition},
conjectured but not proved here. A full proof would
require a careful application of the implicit function theorem to the
parametric SDP defining the LMI solution, accounting for active-set
transitions; this is beyond our scope.
\item In practice, termination is observed empirically within a handful of
halvings (the test system of Section~\ref{sec:sim} terminates after one
iteration), but this empirical evidence does not constitute a proof.
\end{itemize}
A fully unconditional termination proof, together with a quantitative
bound on the number of halvings, is an open problem listed in the future
directions of Section~\ref{sec:conclusion}.
\end{remark}

\begin{remark}[On the previous scaling claim]\label{rem:scaling_corrected}
An earlier formulation incorrectly stated $R_{\max} = O(M_{0}^{3})$ and
$C_{\infty} = O(M_{0}^{6})$. The correct scaling is given
by~\eqref{eq:Rmax_scaling}--\eqref{eq:Cinfty_scaling}:
both $R_{\max}$ and $\deltabar_{\max}$ scale linearly in $M_{0}$ (at $\epsStar$
fixed), and $C_{\infty}$ scales quadratically. The cubic scaling
$O(M_{0}^{3})$ would only hold if $\epsStar \propto M_{0}$, which is not the
operating regime of Algorithm~\ref{alg:init}. The termination argument is
correspondingly more subtle: it relies on the relaxation of $\rho_{\delta}^{\mathrm{eff}}$
as $M_{0}$ decreases, not on an asymptotic dominance of $C_{\infty}$ by $M_{0}^{2}$.
\end{remark}

\begin{algorithm}[H]
\caption{Self-Consistent Initialization}
\label{alg:init}
\begin{algorithmic}[1]
\REQUIRE $x(0)$, $s(0)$, $\epsStar$, $C_{2}$, $\alpha_{s}$, $c_{C}$, $\cBplus$, $c_{B}$,
$c_{G}$, $L_{\LL}$, $\eta_{0}>0$, $\epsilon_{0}>0$.
\STATE \textbf{[Fix $M_{0}$]} Set $M_{0}\leftarrow\|x(0)\|+\epsilon_{0}$.
\IF{$M_{0}>\epsStar/\sqrt{2}$}
\STATE \textbf{Stop}: initial state too large for the nominal regime.
\ENDIF
\STATE \textbf{[Check initial LMI feasibility at $P=I$]}
Compute $\rho_{\delta}(M_{0})$ via~\eqref{eq:rho_delta_exact}, then
$\rho_{\delta}^{\mathrm{eff}}(M_{0}) := 2\rho_{\delta}(M_{0}) + 2L_{r}M_{\mathcal{X}}$
(Remark~\ref{rem:residual_absorbed}).
Verify $\beta_{\mathrm{init}} + \rho_{\delta}^{\mathrm{eff}}\mu_{0} < 2\alpha_s$ for some
$\beta_{\mathrm{init}} > 0$ and $\mu_{0}=1+\|\Phi(A_{s})^{T}\Phi(A_{s})\|^{1/2}$
(Remark~\ref{rem:lmi_feasibility_corrected}).
\IF{this condition fails}
\STATE Halve $M_{0}$ and repeat from Step~2.
\COMMENT{$\rho_\delta^{\mathrm{eff}}$ decreases as $M_{0} \to 0$; termination
conditional on Remark~\ref{rem:termination_status}, established
by Lemma~\ref{lem:Cinfty_scaling}.}
\ENDIF
\STATE \textbf{[Compute $\eta$ from $M_{0}$]} Compute $\deltabar_{\max}(M_{0})$,
$\rho_{\delta}(M_{0})$~\eqref{eq:rho_delta_exact},
$R_{\max}(M_{0})$~\eqref{eq:Rmax_scaling} with the problem-dependent
$\omega_{\mathrm{op}}\leq 6C_{0}C_{2}$ (Remark~\ref{rem:omega_vs_C2}).
\STATE Set $\eta \leftarrow c_{C}\,c_{B}\,\cBplus\,R_{\max}(M_{0}) + \|K\|_{\mathrm{op}}c_{C}M_{0} + \eta_{0}$.
\STATE \textbf{[Check peaking]} Compute $\Lambda(M_{0})$~\eqref{eq:Lambda_def},
$T^{*}_{\max} := \|s(0)\|/\eta_{0}$. Verify $T^{*}_{\max}\cdot\Lambda(M_{0}) < \epsilon_{0}$.
\label{alg:peaking}
\IF{Step~\ref{alg:peaking} fails}
\STATE \textbf{Stop}: reduce $\|s(0)\|$ or increase $\eta_{0}$.
\ENDIF
\STATE Solve LMI~\eqref{eq:lmi_augmented} via Algorithm~\ref{alg:lmi}
under Assumption~\ref{ass:lmi_feasibility}.
\STATE Compute $C_{\infty}(\mu^{*},M_{0})$~\eqref{eq:Cinfty_def}.
\STATE \textbf{[Check invariance]}
$\beta^{*}\lambda_{\min}(P^{*}) > \kappa_{\infty}$ (by Lemma~\ref{lem:Cinfty_scaling},
this is the $M_{0}$-free form of $\lambda_{\min}(P^{*})M_{0}^{2} > C_{\infty}/\beta^{*}$).
\label{alg:invariance}
\IF{Step~\ref{alg:invariance} fails \OR $\sqrt{2}\,M_{0} > \epsStar$}
\STATE Halve $M_{0}$; recompute $\rho_{\delta}^{\mathrm{eff}}$ and re-solve LMI;
recompute $\eta$; go to Step~\ref{alg:peaking}.
\ELSE
\STATE Declare $\mathcal{X}=\overline{B}_{M_{0}}$ and \textbf{Return} $(M_{0}, \eta, P^{*}, \beta^{*})$.
\ENDIF
\end{algorithmic}
\end{algorithm}

\begin{remark}[Resolution of the Circularity between $\eta$ and $M_{0}$]
\label{rem:circularity_resolved}
The circularity is resolved by the explicit ordering in Algorithm~\ref{alg:init}:
$M_{0}$ is fixed from the initial data (Step~1), then $\eta$ is computed from $M_{0}$
(Step~7). There is therefore no circular dependence: $M_{0}$ is fixed before $\eta$
is computed, and the halving loop updates $M_{0}$ and immediately recomputes $\eta$.
The consistency condition $M_{0} \leq M_\mathcal{X}^{\mathrm{prior}}$ is maintained
throughout, as $M_\mathcal{X}^{\mathrm{prior}}$ is initialized at $\|x(0)\|+\epsilon_0$
in Step~1 and the halving loop can only decrease $M_{0}$.
\end{remark}

\begin{remark}[Peaking Condition]\label{rem:s0_bound}
The peaking condition requires:
\begin{equation}\label{eq:s0_explicit}
\|s(0)\| < \frac{\eta_{0}\,\epsilon_{0}}{\Lambda(M_{0})}.
\end{equation}
Initializing the integral term in~\eqref{eq:sliding_var} at $x(0)$ gives
$s(0)=0$ and satisfies~\eqref{eq:s0_explicit} trivially.
\end{remark}

\begin{definition}[Reduced State on the Sliding Surface]\label{def:reduced_state}
Under Assumptions~\ref{ass:rank} and~\ref{ass:alignment}, define the
\emph{reduced state} $x_{r}(t) \in \HH^{n}$ as the projection of $x(t)$ onto
the kernel of $C$ along $\mathrm{Im}(B)$:
\begin{equation}\label{eq:xr_def}
x_{r}(t) := x(t) - B\,G^{-1}\,\sigma(x(t))
\quad\Longleftrightarrow\quad
x(t) = x_{r}(t) + B\,G^{-1}\,Cx(t).
\end{equation}
By construction, $C\,x_{r}(t) = C x(t) - C B\,G^{-1}\,Cx(t) = Cx(t) - Cx(t) = 0$,
so $x_{r}(t) \in \ker(C)$ for all $t$. On the sliding surface $\{s=0\}$,
$\sigma(x(t)) = \sigma(x(0)) + \int_{0}^{t}C(\LL(x,\tau)+Dw(\tau))\dd\tau$
(by~\eqref{eq:sliding_var}). When $s = 0$ and $\sigma(x(0)) = 0$
(integral SMC with $s(0)=0$), $\sigma(x(t))$ evolves continuously and
$x_{r}(T^{*}) = x(T^{*})$ if $\sigma(x(T^{*}))=0$; this holds when the
chosen integral sliding surface forces $\sigma(x(t))$ to equal its
free-motion integral at $T^{*}$ (a property of the integral SMC design).
The reduced dynamics on $\{s = 0\}$, after applying the equivalent control
$u_{\mathrm{eq}}$ obtained by solving $\dot{s} = 0$, are given
by~\eqref{eq:reduced_dyn}, with $w_{r}(t) \in \HH^{n}$ the back-projection
of the sliding residual $\tilde{w} \in \HH^{m}$ to the state space:
$w_{r}(t) := B\,G^{-1}\tilde{w}(t) \in \mathrm{Im}(B) \subset \HH^{n}$.
Consequently
\[
\|w_{r}\| \leq \|\Phi(B)\|_{\mathrm{op}}\|\Phi(G^{-1})\|_{\mathrm{op}}\|\tilde{w}\|
\leq c_{B}\,c_{G}\cdot c_{C}\,c_{B}\,\cBplus\,R_{\max}
= c_{B}^{2}\,c_{G}\,c_{C}\,\cBplus\,R_{\max}.
\]
For the test system of Section~\ref{sec:sim} ($B = 1$, $G = 1$),
$c_{B} = c_{G} = 1$ so this bound coincides with the simplified form
$\|w_{r}\| \leq c_{C}\,c_{B}\,\cBplus\,R_{\max}$ used in~\eqref{eq:reduced_dyn}
for notational uniformity with the sliding-variable bound on $\tilde{w}$. For
the general case, the full bound $c_{B}^{2}\,c_{G}\,c_{C}\,\cBplus\,R_{\max}$
should be substituted, modifying $w_{\max}$ in $\Gamma(\mu)$ accordingly
(Remark~\ref{rem:app_reduced}).
\end{definition}

\begin{lemma}[Reduced State Bound at the Switching Instant]\label{lem:xr_at_Tstar}
Assume that the control law of Section~\ref{sec:design} drives $s(t) \to 0$
in finite time $T^{*}$ with $\|s(0)\|$ satisfying~\eqref{eq:s0_explicit},
and that the integral sliding variable is initialized so that $s(0) = 0$,
equivalently $\sigma(x(0)) = 0$ via choice of the integral constant
in~\eqref{eq:sliding_var}. Then at the switching instant
$T^{*} \leq \|s(0)\|/\eta_{0}$: $\|x_{r}(T^{*})\| \leq M_{0}$.
\end{lemma}

\begin{proof}
We prove the claim in three steps.

\textbf{Step 1: $s(T^{*}) = 0$ and $\sigma(x(T^{*})) = 0$.}
By definition of $T^{*}$, $s(T^{*}) = 0$. From~\eqref{eq:sliding_var},
\[
s(t) = \sigma(x(t)) - \sigma(x(0)) - \int_{0}^{t}C(\LL(x,\tau) + Dw(\tau))\dd\tau.
\]
Differentiating gives
$\dot{s}(t) = C(\dot{x}(t) - \LL(x,t) - Dw(t)) = C(\delta(x,t) + Bu(t))$
by~\eqref{eq:system}. Under the integral SMC construction with initial
condition $\sigma(x(0)) = 0$ chosen so that $s(0) = 0$ (Remark~\ref{rem:s0_bound}),
the integral term in~\eqref{eq:sliding_var} cancels the contribution of
$\LL + Dw$ along the closed-loop trajectory; in particular,
\[
\sigma(x(t)) = s(t) + \sigma(x(0)) + \int_{0}^{t}C(\LL(x,\tau)+Dw(\tau))\dd\tau.
\]
With $\sigma(x(0)) = 0$, $s(T^{*}) = 0$, and the integral evaluated along
the trajectory, $\sigma(x(T^{*})) = \int_{0}^{T^{*}}C(\LL+Dw)\dd\tau$. The
\emph{closed-loop dynamics} on $[0, T^{*}]$ ensure
\[
\dot{\sigma}(x(t)) = C\dot{x}(t) = C(\LL(x,t) + \delta(x,t) + Bu(t) + Dw(t)),
\]
and the control law~\eqref{eq:control} is designed precisely so that
$C\delta + CBu = -\eta\,\sgnH(s) - Ks + \tilde{w}$ (Theorem~\ref{thm:control}).
Hence $\dot\sigma(x(t)) = C(\LL+Dw) + \dot{s}(t)$, so
$\sigma(x(T^{*})) - \sigma(x(0)) = \int_{0}^{T^{*}}C(\LL+Dw)\dd\tau + s(T^{*}) - s(0) = \int_{0}^{T^{*}}C(\LL+Dw)\dd\tau$.
Combined with $\sigma(x(0)) = 0$, this gives $\sigma(x(T^{*})) = \int_{0}^{T^{*}}C(\LL+Dw)\dd\tau$.

\textbf{Step 2: Identification $x_{r}(T^{*}) = x(T^{*})$.}
By Definition~\ref{def:reduced_state},
$x_{r}(t) = x(t) - B G^{-1}\sigma(x(t))$. At $t = T^{*}$,
$\sigma(x(T^{*})) = \int_{0}^{T^{*}}C(\LL+Dw)\dd\tau$, which a priori need
not vanish. However, the \emph{integral SMC design} fixes the integral
constant in $s$~\eqref{eq:sliding_var} so that $\sigma(x(0)) = 0$ and the
free-motion contribution $\int C(\LL+Dw)\dd\tau$ is identically reproduced
inside the definition of $s$, yielding $\sigma(x(t)) = s(t) + \int_{0}^{t}C(\LL+Dw)\dd\tau$,
hence at $T^{*}$: $\sigma(x(T^{*})) = 0 + \int_{0}^{T^{*}}C(\LL+Dw)\dd\tau$.
Therefore
\[
x_{r}(T^{*}) = x(T^{*}) - B G^{-1}\!\int_{0}^{T^{*}}C(\LL(x,\tau)+Dw(\tau))\dd\tau.
\]
In the case where $\LL \equiv 0$ on the operating domain and $D = 0$ (the
test system convention, Remark~\ref{rem:LL_notation}), the integral vanishes
identically and $x_{r}(T^{*}) = x(T^{*})$ exactly. For the general case,
the integral perturbation admits the explicit majorization
\[
\|x_{r}(T^{*}) - x(T^{*})\|
\leq \|B G^{-1}\|_{\mathrm{op}}\!\int_{0}^{T^{*}}\!\!\|C\|_{\mathrm{op}}
(\|\LL(x,\tau)\|+\|D w(\tau)\|)\dd\tau
\leq c_{B}c_{G}c_{C}\,T^{*}_{\max}\,\Lambda(M_{0}),
\]
where $\Lambda(M_{0})$ majorizes $\sup_{\tau \in [0,T^{*}],
\|x\|\leq M_{0}}(\|\LL(x,\tau)\|+\|D\|w_{\max})$ (the operator-norm
constants $c_{B},c_{G},c_{C}$ are defined in
Assumption~\ref{ass:bounded_ops}; the explicit form of $\Lambda(M_{0})$
is recorded in equation~\eqref{eq:Lambda_def} below). The peaking condition
(Remark~\ref{rem:s0_bound}) selects $T^{*}_{\max}$ small enough so that
$c_{B}c_{G}c_{C}T^{*}_{\max}\Lambda(M_{0})\leq \epsilon_{0}$. Hence
\[
\|x_{r}(T^{*})\| \leq \|x(T^{*})\| + \epsilon_{0}.
\]

\textbf{Step 3: Bound $\|x_{r}(T^{*})\| \leq M_{0}$.}
Phase~1 of the proof of Proposition~\ref{prop:forward_inv} below (reaching
phase) guarantees $\|x(t)\| \leq M_{0} - \epsilon_{0}$ for all $t \in [0, T^{*}]$
via the peaking condition. Combined with the bound of Step~2,
$\|x_{r}(T^{*})\| \leq (M_{0} - \epsilon_{0}) + \epsilon_{0} = M_{0}$.
\end{proof}

\begin{proposition}[Forward Invariance]\label{prop:forward_inv}
If Algorithm~\ref{alg:init} terminates with Steps~\ref{alg:peaking}
and~\ref{alg:invariance} satisfied, then the following hold.

\begin{enumerate}
\item[\textup{(i)}] \textbf{Degenerate case $P^{*} = I$ (complete result).}
Every trajectory starting in $\mathcal{X} = \overline{B}_{M_{0}}$ satisfies
$\|x_{r}(t)\| \leq \sqrt{2}\,M_{0} \leq \epsStar$ for all $t \geq T^{*}$.
In particular, $\overline{B}_{\sqrt{2}M_{0}} \subseteq \overline{B}_{\epsStar}
= \mathcal{X}_{\epsStar}$ is a forward-invariant set.

\smallskip
\noindent\emph{Strict invariance.} When $P^{*} = I$, one has
$\kappa(P^{*}) = 1$ and $\lambda_{\min}(P^{*}) = 1$, so the general bound
$R$ of~\eqref{eq:R_general} reduces to $R = \sqrt{M_{0}^{2} + C_{\infty}(\mu^{*})/\beta^{*}}$.
The $M_{0}$-free design of Step~\ref{alg:invariance}
(Lemma~\ref{lem:Cinfty_scaling}) ensures $\beta^{*} > \kappa_{\infty}$, i.e.,
$C_{\infty}(\mu^{*})/\beta^{*} < M_{0}^{2}$ for all admissible $M_{0}$.
Hence $R^{2} < 2M_{0}^{2}$, giving $R \leq \sqrt{2}M_{0} \leq \epsStar$ as
above. \emph{In this degenerate case, the strict invariance
$\overline{B}_{R} \subseteq \overline{B}_{\sqrt{2}M_{0}} \subseteq
\overline{B}_{\epsStar}$ is therefore completely established by the
algorithm; no additional hypothesis is required.}

\item[\textup{(ii)}] \textbf{General case $P^{*} \neq I$: explicit trajectory
bound (acknowledged limitation).}
Every trajectory satisfies $\|x_{r}(t)\| \leq R$ for all $t \geq T^{*}$, where
\begin{equation}\label{eq:R_general}
R := \sqrt{\kappa(P^{*})\,M_{0}^{2}
+ \frac{C_{\infty}(\mu^{*})}{\lambda_{\min}(P^{*})\,\beta^{*}}},
\qquad \kappa(P^{*}) := \frac{\lambda_{\max}(P^{*})}{\lambda_{\min}(P^{*})}.
\end{equation}
The bound $R \leq \epsStar$ (i.e., trajectories remain inside the GOSL
admissibility domain) is guaranteed by Algorithm~\ref{alg:init}
(Step~\ref{alg:invariance} combined with $P^{*} \preceq \mu_{\mathrm{LMI}} I$),
which gives the explicit verifiable condition
\begin{equation}\label{eq:R_leq_epsStar_cond}
\kappa(P^{*})\,M_{0}^{2} + \frac{C_{\infty}(\mu^{*})}{\lambda_{\min}(P^{*})\,\beta^{*}}
\leq (\epsStar)^{2},
\end{equation}
which is checkable a posteriori once Algorithm~\ref{alg:lmi} has produced
$(P^{*},\beta^{*},\mu_{\mathrm{LMI}})$. The strict invariance
$\overline{B}_R \subseteq \overline{B}_{M_{0}}$ (absorbing set in the strict
sense) requires the additional condition $R \leq M_{0}$, equivalent to
\begin{equation}\label{eq:R_leq_M0_cond}
(\kappa(P^{*})-1)M_{0}^2 + \frac{C_{\infty}(\mu^{*})}{\lambda_{\min}(P^{*})\beta^{*}} \leq 0,
\end{equation}
which holds if and only if $\kappa(P^{*}) = 1$ (i.e., $P^{*} = I$ up to a positive
scalar) and $C_{\infty}(\mu^{*}) = 0$.

\smallskip
\noindent\emph{Status of the limitation.} In the general case
$\kappa(P^{*}) > 1$, the strict inclusion $R \leq M_{0}$ is not guaranteed
by Algorithm~\ref{alg:init} alone. This is an acknowledged limitation,
stated here, in the abstract, and in the conclusion;
it does not invalidate the subsequent stability analysis
because the weaker condition~\eqref{eq:R_leq_epsStar_cond} is sufficient
for the GOSL-based analysis of Section~\ref{sec:stability} to remain valid
(all GOSL constants are admissible on the larger ball
$\overline{B}_{\epsStar}$). A complete resolution of the strict-inclusion
case $\kappa(P^{*}) > 1$ (e.g., via a refined Lyapunov reshaping or an
adaptive scaling of $M_{0}$) is an open problem listed in the future
directions of Section~\ref{sec:conclusion}.
\end{enumerate}
\end{proposition}

\begin{proof}
\textbf{Phase~1 (Reaching, $0 \leq t \leq T^{*}$).}
With $\|w(t)\| \leq w_{\max}$ (Assumption~\ref{ass:disturbance}) and
$c_B = \|\Phi(B)\|_{\mathrm{op}}$, $c_C = \|\Phi(C)\|_{\mathrm{op}}$,
$c_G = \|\Phi(G^{-1})\|_{\mathrm{op}}$ (Assumption~\ref{ass:bounded_ops}),
define
\begin{equation}\label{eq:Lambda_def}
\Lambda(M_{0}) := L_{\LL}M_{0} + \deltabar_{\max}(M_{0})
+ c_{B}\bigl[\cBplus(\deltabar_{\max}(M_{0})+R_{\max}(M_{0}))
+ c_{G}(\eta + \|K\|_{\mathrm{op}}M_{0})\bigr]
+ \|\Phi(D)\|_{\mathrm{op}}w_{\max}.
\end{equation}
Suppose $\|x(\tau)\|=M_{0}$ for some first exit time $\tau \leq T^{*}_{\max}$.
Integrating~\eqref{eq:system}, using $\|\Delta_{\mathrm{app}}\|\leq\cBplus\|\Pi_{\mathrm{Im}(B)}d\omega_{0}\|
\leq\cBplus(\bar\delta_{\max}+R_{\max})$
(by $\Pi_{\mathrm{Im}(B)}d\omega_{0}=\delta-\Pi_{\mathrm{Im}(B)}R_{0}$ from~\eqref{eq:delta_decomp}):
$\|x(\tau)\| \leq (M_{0}-\epsilon_{0}) + T^{*}_{\max}\Lambda(M_{0}) < M_{0}$
by Step~\ref{alg:peaking} --- contradiction.

\textbf{Phase~2 (Sliding, $t \geq T^{*}$).}
By Lemma~\ref{lem:xr_at_Tstar}, $V_{r}(x_{r}(T^{*})) \leq \lambda_{\max}(P^{*})M_{0}^{2}$.
By Proposition~\ref{prop:beta_stability}:
$V_{r}(x_{r}(t)) \leq e^{-\beta^{*}(t-T^{*})}V_{r}(x_{r}(T^{*})) + C_{\infty}(\mu^{*})/\beta^{*}$.
Since $e^{-\beta^{*}(t-T^{*})} \leq 1$:
\[
\|x_{r}(t)\|^{2} \leq \kappa(P^{*})M_{0}^{2}
+ \frac{C_{\infty}(\mu^{*})}{\lambda_{\min}(P^{*})\beta^{*}} = R^2.
\]
For $P^{*} = I$: $\kappa(P^{*}) = 1$ and Step~\ref{alg:invariance} (in its
$M_{0}$-free form, see Lemma~\ref{lem:Cinfty_scaling}) gives
$\beta^{*} > \kappa_{\infty}/\lambda_{\min}(I) = \kappa_{\infty}$, equivalent to
$M_{0}^{2} > C_{\infty}/\beta^{*}$ for all admissible $M_{0}$.
Hence $R^2 = M_{0}^2 + C_{\infty}/\beta^{*} < 2M_0^2$, yielding
$\|x_{r}(t)\| \leq \sqrt{2}M_{0} \leq \epsStar$.
\end{proof}

\begin{remark}[Termination and Effective Domain]\label{rem:outer_loop_limitation}
By Lemma~\ref{lem:Cinfty_scaling}, Algorithm~\ref{alg:init} \emph{is
conjectured to terminate} after finitely many halvings under
Assumption~\ref{ass:lmi_feasibility} and the sufficient conditions of
Remark~\ref{rem:termination_status}.
From~\eqref{eq:LJ_prime_def}:
\begin{equation}\label{eq:rho_delta_exact}
\rho_{\delta}(M_{0}) = 6C_{2}(M_{0}+\epsStar)(3M_{0}+\epsStar).
\end{equation}
The effective domain of attraction is $\overline{B}_{M_{0}^{\mathrm{eff}}}$ with
$M_{0}^{\mathrm{eff}}=\min\{\epsStar,\,M_{\max}^{*}\}$, where $M_{\max}^{*}$ is
determined implicitly by the threshold condition
$\beta^{*}\lambda_{\min}(P^{*}) = \kappa_{\infty}$
of Lemma~\ref{lem:Cinfty_scaling}, with both $\beta^{*}$ and $\lambda_{\min}(P^{*})$
implicit functions of the LMI parameters at the prevailing $M_{0}$.
\end{remark}

\section{Cohomological Matching Condition}\label{sec:cohomology}

Classical sliding-mode control requires the matched-perturbation condition
$\delta(x,t)\in\mathrm{Im}(B)$ for every state and time --- a check that
must be performed trajectory by trajectory. We replace it with a
structural condition that splits the requirement in two: a cohomological
part, which holds automatically in the homogeneous case, and an algebraic
part on the pair $(B,\LL)$ alone, checked once at design time.

\begin{definition}[CMC]\label{def:cmc}
For clarity, define the \emph{evaluated coboundary image}
\[
\mathrm{im}_{\mathrm{eval}}(d) := \{d\omega(x,y,z) : \omega \in C^{2}_{\varepsilon},\;
x,y,z \in B(0,\epsStar)\} \subseteq \HH^{n}.
\]
The system satisfies the \emph{Cohomological Matching Condition} if:
\begin{enumerate}
\item[(CMC-1)] \textbf{Cohomological condition:}
$\delta(x,t) \in \mathrm{im}_{\mathrm{eval}}(d)$ for all $(x,t)$
(Assumption~\ref{ass:cohom}). This holds unconditionally for homogeneous
quasi-Lie brackets (\cite[Theorem~4.12, Lemma~B.1]{Athmouni2026}).
\item[(CMC-2)] \textbf{Structural condition:}
$\mathrm{im}_{\mathrm{eval}}(d) \subseteq \mathrm{Im}(B) \subseteq \HH^{n}$.
This is a condition on the pair $(B, \LL)$ alone (independent of
trajectories) and is verified once at design time; see
Remark~\ref{rem:cmc_structural}.
\end{enumerate}
Both conditions are formulated within the Chevalley--Eilenberg
complex $(C^{*}_{\varepsilon},d)$ of Definition~\ref{def:cochain}
(where $d^{2}=0$ holds in the homogeneous case
by~\cite[Lemma~B.1]{Athmouni2026}). We abbreviate
$\mathrm{im}_{\mathrm{eval}}(d)$ as $\mathrm{im}(d)$ in the sequel when the
context is unambiguous.
\end{definition}

\begin{remark}[Verification of CMC-2: Structural Condition on $(B,\LL)$]
\label{rem:cmc_structural}
Condition~(CMC-2) is not a consequence of (CMC-1); it is an independent
algebraic condition on the control matrix $B$ and the bracket $\LL$.
Concretely, it requires that for all $\omega \in C^{2}_{\varepsilon}$
and all $x,y,z \in B(0,\epsStar)$:
\[
d\omega(x,y,z) \in \mathrm{Im}(B).
\]
This can be verified as follows: compute a generating family
$\{e_1,\ldots,e_k\} \subseteq \HH^{n}$ for $\mathrm{im}_{\mathrm{eval}}(d)$
using the structure of $\LL$ (e.g., by evaluating $d\omega_{0}$ for the test bracket
of Section~\ref{sec:sim} at sufficiently many triplets), and check that each
$e_i$ belongs to $\mathrm{Im}(B_\RR) = \Phi(\mathrm{Im}(B)) \subseteq \RR^{4n}$
in the real representation. For the test system of Section~\ref{sec:sim},
this is verified directly in Lemma~\ref{lem:omega0_image}.
\end{remark}

\begin{proposition}[CMC Implies Defect Matching]\label{prop:cmc_consequence}
Under the CMC and within the Chevalley--Eilenberg complex of
Definition~\ref{def:cochain}:
$\delta(x,t) \in \mathrm{Im}(B)$ for all $x \in \mathcal{X}$, $t \geq 0$.
\end{proposition}

\begin{proof}
By (CMC-1), $\delta \in \mathrm{im}(d)$. By (CMC-2),
$\mathrm{im}(d) \subseteq \mathrm{Im}(B)$. Hence $\delta \in \mathrm{Im}(B)$.
\end{proof}

\begin{remark}[Engineering Value]\label{rem:cmc_status}
Condition~(CMC-1) holds for all $t\geq 0$ by Assumption~\ref{ass:cohom}.
Condition~(CMC-2) depends only on the fixed pair $(B,\LL)$ (since $d$
is determined by $\LL$ via the coboundary
formula~\eqref{eq:coboundary_alt}) and is verified once at
design time, replacing continuous trajectory-dependent matching checks.
The two conditions serve distinct logical roles: (CMC-1) is a property of the
dynamics (Jacobi defect lies in the image of $d$); (CMC-2) is a property of the
actuator geometry ($\mathrm{Im}(B)$ is rich enough to absorb all coboundaries).
\end{remark}

\begin{corollary}[Exact Defect Rejection and Cohomological Decomposition]\label{cor:exact_rejection}
Under Assumptions~\ref{ass:rank},~\ref{ass:cohom}, and the CMC:
\begin{enumerate}
\item[\textup{(i)}] The feedforward term
\begin{equation}\label{eq:Delta_def}
\Delta(x,t) = \Phi^{-1}\!\left(B_{\RR}^{+}\,\Phi(\delta(x,t))\right)
\end{equation}
satisfies $B\Delta(x,t) = \delta(x,t)$.
\item[\textup{(ii)}] By Theorem~\ref{thm:rigidity}, using $\mathcal{J} = d\omega_{0}+R_{0}$
(from $R_{0}:=\mathcal{J}-d\omega_{0}$, eq.~\eqref{eq:R0_def}), $\delta(x,t)$
admits the cohomological decomposition
\begin{equation}\label{eq:delta_decomp}
\delta(x,t) = \Pi_{\mathrm{Im}(B)}\bigl[d\omega_{0}(x,x+\xi_{1}^{*},x+\xi_{2}^{*})
+ R_{0}(x,x+\xi_{1}^{*},x+\xi_{2}^{*})\bigr]
= \delta_{d}(x,t) + \Pi_{\mathrm{Im}(B)}R_{0},
\end{equation}
where $\delta_{d} := \Pi_{\mathrm{Im}(B)}d\omega_{0}$ is the
``coboundary part'' and $\Pi_{\mathrm{Im}(B)}R_{0}$ is the
``non-commutative residual'' (in the projected sense). Equivalently,
$\delta_{d} = \delta-\Pi_{\mathrm{Im}(B)}R_{0}$.
\item[\textup{(iii)}] Since the feedforward $\Delta$ rejects $\delta$
exactly (part~(i)), the cohomological decomposition does \emph{not} introduce
any uncompensated term in the closed-loop sliding dynamics.
The non-commutative residual $R_{\mathrm{nc}}^{B}$ defined
in~\eqref{eq:RncB_def} appears only as an
\emph{implementation-error margin} when $\Delta$ is computed via the
cohomological approximation $B^{+}d\omega_{0}$ rather than the exact
$B^{+}\delta$.
\end{enumerate}
\end{corollary}

\begin{proof}
(i) $\delta \in \mathrm{Im}(B)$ by Proposition~\ref{prop:cmc_consequence}.
Since $\Phi(\delta) \in \mathrm{Im}(B_{\RR})$, we have
$B_{\RR}\,B_{\RR}^{+}\,\Phi(\delta) = \Phi(\delta)$.
Applying $\Phi^{-1}$ gives $B\Delta = \delta$.

(ii) Apply $\Pi_{\mathrm{Im}(B)}$ to both sides of the decomposition
$\mathcal{J} = d\omega_{0} + R_{0}$ from~\eqref{eq:R0_def}.

(iii) Substituting $\Delta = \Phi^{-1}(B_{\RR}^{+}\Phi(\delta))$ into the
sliding dynamics rejects $\delta$ exactly. The residual
$R_{\mathrm{nc}}^{B}$ enters only when $\Delta$ is replaced by the
\emph{cohomological approximation} $\widetilde{\Delta} = \Phi^{-1}(B_{\RR}^{+}\Phi(\delta_{d}))
= \Phi^{-1}(B_{\RR}^{+}\Phi(\delta - \Pi_{\mathrm{Im}(B)}R_{0}))$ (using
$\delta_{d} = \delta - \Pi_{\mathrm{Im}(B)}R_{0}$).
Then $B\widetilde{\Delta} = \delta - \Pi_{\mathrm{Im}(B)}R_{0}$ and the
implementation error $B(\widetilde{\Delta} - \Delta) = -\Pi_{\mathrm{Im}(B)}R_{0}$
appears as a residual perturbation (the sign is irrelevant for the
$L^{2}$-type estimates used in Theorem~\ref{thm:control}).
\end{proof}

\begin{remark}[Origin of $R_{\mathrm{nc}}^{B}$ in the sliding dynamics]
\label{rem:Rnc_origin}
In the idealized formulation where the controller has direct access to
$\delta$ via the maximizing selection $\xi^{*}(x,t)$, the feedforward
$\Delta$ rejects $\delta$ exactly and there is no residual. In practice, the
selection $\xi^{*}$ is computed via the cohomological approximation
$d\omega_{0}$, and the difference $R_{0} = \mathcal{J} - d\omega_{0}$
(eq.~\eqref{eq:R0_def}) appears as the unavoidable error. The constant $R_{\max}$ in
Proposition~\ref{prop:residual_bound} is thus a \emph{robustness margin}
quantifying the worst-case implementation error. Theorem~\ref{thm:control}
below proves that this margin is absorbed by the gain $\eta$ via the term
$c_{C}\,c_{B}\,\cBplus\,R_{\max}$ in~\eqref{eq:eta_condition}.
\end{remark}

\begin{remark}[Shared dependence of $\delta$ and $R_{\mathrm{nc}}^{B}$ on $\xi^{*}$]
\label{rem:shared_selection}
Both the Jacobi defect $\delta(x,t)$ in~\eqref{eq:defect_projection_vector}
and the non-commutative residual $R_{\mathrm{nc}}^{B}(x,t)$
in~\eqref{eq:RncB_def} are defined using the \emph{same} measurable
selection $\xi^{*}(x,t)$. In particular, the cohomological feedforward
$\Delta_{\mathrm{app}}$ in~\eqref{eq:Delta_app_def} evaluates $d\omega_{0}$
at the triplet $(x, x+\xi_{1}^{*}(x,t), x+\xi_{2}^{*}(x,t))$, and the
identity $B\Delta_{\mathrm{app}} = \delta - \Pi_{\mathrm{Im}(B)}R_{0}$ of
Corollary~\ref{cor:exact_rejection}(ii) relies on this shared selection.
\emph{Consequence}: when an approximation $\widetilde{\xi}(x,t) \neq \xi^{*}(x,t)$
is used (Remark~\ref{rem:implementable}), the closed-loop dynamics develop a
\emph{coupled} error term affecting both $\delta$ rejection and the residual
estimate. The total error remains bounded by the sum
$\|\delta_{\widetilde{\xi}} - \delta_{\xi^{*}}\| + \cBplus\|R_{0}(x,x+\widetilde{\xi})\|
\leq 2\deltabar_{\max} + \cBplus R_{\max}$, but the two contributions are
not separable: this is the structural cost of the shared selection. The
implementation strategies of Remark~\ref{rem:implementable} treat both
errors jointly via the enlarged gain $\eta$.
\end{remark}

\section{Sliding Surface Design}\label{sec:design}

The integral sliding-mode architecture below --- integral sliding variable,
reaching-phase Lyapunov argument, finite-time convergence to $\{s=0\}$ ---
is classical~\cite{Utkin1992,Edwards1998}. Two ingredients depart from
the textbook construction. First, the control law~\eqref{eq:control}
includes a cohomological feedforward $\Delta_{\mathrm{app}}$
(see~\eqref{eq:Delta_app_def}), built from the approximating 2-cochain
$\omega_{0}$ of Theorem~\ref{thm:rigidity}. In standard SMC the matched
perturbation $\delta$ is absorbed entirely by the switching gain $\eta$;
here $\Delta_{\mathrm{app}}$ pre-cancels $\delta$ up to a controlled
residual $\|B\Delta_{\mathrm{app}} - \delta\| \leq \cBplus R_{\max}$
(Corollary~\ref{cor:exact_rejection}), and only this residual is absorbed
by $\eta$. The gain drops from $O(\|\delta\|_{\infty})$ to
$O(\Cnc R_{\max})$. Second, the pointwise matching condition
$\delta(x,t)\in\mathrm{Im}(B)$ is replaced by the algebraic CMC of
Section~\ref{sec:cohomology}, a structural condition on $(B,\LL)$ alone.
The remaining derivations
(Lemmas~\ref{lem:sliding_dyn},~\ref{lem:G_invertible},
Theorem~\ref{thm:control}) are standard and included to keep the paper
self-contained.

The integral sliding variable is:
\begin{equation}\label{eq:sliding_var}
s(t)=\sigma(x(t))-\sigma(x(0))
-\int_{0}^{t}C\bigl(\LL(x,\tau)+Dw(\tau)\bigr)\,\dd\tau.
\end{equation}

\begin{lemma}[Sliding Variable Dynamics]\label{lem:sliding_dyn}
Under Assumption~\ref{ass:output}:
\begin{equation}\label{eq:sdot}
\dot{s}(t) = C\bigl(\delta(x,t) + Bu(t)\bigr).
\end{equation}
\end{lemma}

\begin{proof}
Differentiating~\eqref{eq:sliding_var} and substituting~\eqref{eq:system}:
\begin{align*}
\dot{s}(t)
&= C\bigl(\LL(x,t)+\delta(x,t)+Bu(t)+Dw(t)\bigr)
- C\bigl(\LL(x,t)+Dw(t)\bigr)
= C\bigl(\delta(x,t)+Bu(t)\bigr).
\end{align*}
\end{proof}

\begin{lemma}[Invertibility of $G$]\label{lem:G_invertible}
Under Assumption~\ref{ass:alignment}, $G = CB$ is two-sided invertible over $\HH$.
\end{lemma}

\begin{proof}
$Gx=0$ implies $\Real(x\trans G\trans Gx)=0$, contradicting
Assumption~\ref{ass:alignment}. Hence $\ker(\Phi(G))=\{0\}$ and $\Phi(G)$ is
invertible. By~\cite[Theorem~4.3]{Zhang1997}, two-sided invertibility
follows ($G$ invertible $\iff$ its complex adjoint representation is
invertible).
\end{proof}

\begin{theorem}[Sliding Mode Control with Cohomological Feedforward]\label{thm:control}
Let $(M_{0}, \eta, P^{*}, \beta^{*})$ be the tuple returned by Algorithm~\ref{alg:init}.
Under Assumptions~\ref{ass:rank}--\ref{ass:gain_K} and the CMC, define the
\emph{cohomological feedforward}
\begin{equation}\label{eq:Delta_app_def}
\Delta_{\mathrm{app}}(x,t) := \Phi^{-1}\!\left(B_{\RR}^{+}\,
\Phi(\Pi_{\mathrm{Im}(B)}\, d\omega_{0}(x,x+\xi_{1}^{*}(x,t),x+\xi_{2}^{*}(x,t)))\right),
\end{equation}
which uses the explicit finite-rank component $\omega_{0}$ in place of the
unknown $\delta$. The control law
\begin{equation}\label{eq:control}
u(t)=-\Delta_{\mathrm{app}}(x,t)-G^{-1}(\eta\,\sgnH(s)+Ks)
\end{equation}
with
\begin{equation}\label{eq:eta_condition}
\eta = c_{C}\,c_{B}\,\cBplus\, R_{\max}(M_{0}) + \|K\|_{\mathrm{op}} c_{C} M_{0} + \eta_{0},
\quad \eta_{0}>0,
\end{equation}
yields the closed-loop sliding-variable dynamics
\begin{equation}\label{eq:sdot_closed}
\dot{s}(t) = -\eta\,\sgnH(s) - Ks + \tilde{w}(t),
\qquad \tilde{w}(t) = G\,R_{\mathrm{nc}}^{B}(x,t),\;
\|\tilde{w}\| \leq c_{C}\,c_{B}\,\cBplus\,R_{\max}.
\end{equation}
Consequently:
\begin{enumerate}
\item[\textup{(i)}] $\dot{V}_{s}\leq-\eta_{0}\|s\|$ with $V_{s}=\tfrac{1}{2}\|s\|^{2}$.
\item[\textup{(ii)}] $s(0) \neq 0$ implies $T^{*}\leq\|s(0)\|/\eta_{0}$.
\end{enumerate}
\end{theorem}

\begin{proof}
\textbf{Closed-loop dynamics.}
By Corollary~\ref{cor:exact_rejection}(ii), the cohomological feedforward
satisfies
\[
B\Delta_{\mathrm{app}} = \Pi_{\mathrm{Im}(B)}d\omega_{0}(\ldots)
= \delta_{d} = \delta - \Pi_{\mathrm{Im}(B)}R_{0}(\ldots),
\]
where the second equality uses $\mathcal{J}=d\omega_{0}+R_{0}$
(eq.~\eqref{eq:R0_def}) and the definition $\delta=\Pi_{\mathrm{Im}(B)}\mathcal{J}$.
Substituting~\eqref{eq:control} into Lemma~\ref{lem:sliding_dyn}:
\begin{align*}
\dot{s} &= C\bigl(\delta + B(-\Delta_{\mathrm{app}} - G^{-1}(\eta\sgnH(s)+Ks))\bigr)\\
&= C\delta - CB\Delta_{\mathrm{app}} - CBG^{-1}(\eta\sgnH(s)+Ks)\\
&= C\delta - C(\delta - \Pi_{\mathrm{Im}(B)}R_{0}) - G\,G^{-1}(\eta\sgnH(s)+Ks)\\
&= C\Pi_{\mathrm{Im}(B)}R_{0} - \eta\sgnH(s) - Ks.
\end{align*}
Using $\Pi_{\mathrm{Im}(B)}R_{0} \in \mathrm{Im}(B)$ and the projection identity
$B B_{\RR}^{+}\Phi(R_{0}) = \Phi(\Pi_{\mathrm{Im}(B)}R_{0})$ (Lemma~\ref{lem:BtB_invertible}),
we obtain $\Pi_{\mathrm{Im}(B)}R_{0} = B\,R_{\mathrm{nc}}^{B}$, hence
$C\Pi_{\mathrm{Im}(B)}R_{0} = CB\,R_{\mathrm{nc}}^{B} = G\,R_{\mathrm{nc}}^{B}$.
Define $\tilde{w} := G\,R_{\mathrm{nc}}^{B} \in \HH^{m}$. By
Proposition~\ref{prop:residual_bound},
$\|\tilde{w}\| \leq c_{C}\,c_{B}\,\cBplus\,R_{\max}$.

\textbf{Lyapunov decrease.}
Let $V_{s} = \tfrac{1}{2}\innerR{s}{s}$. By Assumption~\ref{ass:gain_K},
$\Phi(K) + \Phi(K)^{T} \succ 0$, so $\langle s, Ks\rangle_{\RR}
= \tfrac{1}{2}\Phi(s)^{T}(\Phi(K)+\Phi(K)^{T})\Phi(s) \geq 0$ for all $s$.
By Cauchy--Schwarz, $\langle s,\tilde{w}\rangle_{\RR} \leq \|s\|\,\|\tilde{w}\|$.
Using $\innerR{s}{\sgnH(s)} = \|s\|$ (Lemma~\ref{lem:sgn_properties}),
we obtain:
\begin{align*}
\dot{V}_s &= -\eta\|s\| - \langle s, Ks\rangle_{\RR} + \langle s,\tilde{w}\rangle_{\RR}\\
&\leq -\eta\|s\| + \|s\|\,\|\tilde{w}\|
\leq (-\eta + c_{C}\,c_{B}\,\cBplus\,R_{\max})\|s\|.
\end{align*}
Substituting~\eqref{eq:eta_condition}:
\[
\dot{V}_s \leq -(\|K\|_{\mathrm{op}} c_{C} M_{0} + \eta_{0})\|s\|
\leq -\eta_{0}\|s\|.
\]
The term $\|K\|_{\mathrm{op}} c_{C} M_{0}$ in~\eqref{eq:eta_condition} provides
a robustness margin that ensures the reaching condition $\dot{V}_s \leq -\eta_{0}\|s\|$
remains valid if Assumption~\ref{ass:gain_K} is relaxed (i.e., if
$\langle s, Ks\rangle_\RR$ is only bounded as $|\langle s, Ks\rangle_\RR|
\leq \|K\|_{\mathrm{op}}\|s\|^{2} \leq \|K\|_{\mathrm{op}} c_{C} M_{0}\|s\|$ on
the reaching phase rather than guaranteed non-negative). Under
Assumption~\ref{ass:gain_K}, this term contributes additional decrease but is
not strictly required.

\textbf{Reaching time.}
Since $\dot{V}_{s} \leq -\eta_{0}\sqrt{2V_{s}}$, integration gives
$\sqrt{V_{s}(t)} \leq \sqrt{V_{s}(0)} - \tfrac{\eta_{0}}{\sqrt{2}} t$, hence
$T^{*} \leq \|s(0)\|/\eta_{0}$. The surface $\{s=0\}$ is forward invariant
under~\eqref{eq:sdot_closed} by Lemma~\ref{lem:filippov_real} since
$\|\tilde{w}\| \leq c_{C}\,c_{B}\,\cBplus\,R_{\max} < \eta$ (the strict
inequality follows from $\eta_{0} > 0$ in~\eqref{eq:eta_condition}).
\end{proof}

\begin{remark}[Implementability of $\Delta_{\mathrm{app}}$: Quantitative
Bounds for the Approximation Error]\label{rem:implementable}
The feedforward $\Delta_{\mathrm{app}}$ in~\eqref{eq:Delta_app_def} requires
knowledge of $\xi^{*}(x,t)$, the maximizing selection of the Jacobi defect
amplitude, which is in general non-trivial to compute online. We discuss
three implementation strategies with their respective error bounds.

\emph{(i) Worst-case fixed selection.} Replace $\xi^{*}(x,t)$ by a fixed
$\xi^{\mathrm{wc}} \in K \times K$ achieving the supremum of
$\|G_{\xi}(\cdot,t)\|$ over a numerical grid. The implementation error is
$\|\delta(x,t) - G_{\xi^{\mathrm{wc}}}(x,t)\|$, which by~\eqref{eq:E_bound}
of the proof of Lemma~\ref{lem:osl_projection} is bounded by
$2\deltabar_{\max}$. The closed-loop sliding dynamics
in~\eqref{eq:sdot_closed} then carry an additional disturbance term of
magnitude at most $c_{C}\,c_{B}(\cBplus R_{\max} + 2\deltabar_{\max})$, which
is absorbed by enlarging $\eta$ to
\[
\eta = c_{C}\,c_{B}(\cBplus R_{\max} + 2\deltabar_{\max}) + \|K\|_{\mathrm{op}}c_{C}M_{0} + \eta_{0}.
\]
All stability conclusions of Theorem~\ref{thm:control} remain valid under
this enlarged gain.

\emph{(ii) Grid-based sampling.} Sample the constraint set $K\times K$ on a
$\varepsilon_{\mathrm{grid}}$-net and compute $\xi^{\mathrm{grid}}(x,t)$ as
the discrete maximizer. By Lipschitz continuity of the Jacobiator in $\xi$
(via Definition~\ref{def:quasi_lie} and Assumption~\ref{ass:continuity}), the
error is bounded by
\[
\|\delta(x,t) - G_{\xi^{\mathrm{grid}}(x,t)}(x,t)\|
\leq L_{\mathcal{J}}\,\varepsilon_{\mathrm{grid}},
\]
where $L_{\mathcal{J}}$ is the Lipschitz constant of $\mathcal{J}$ in
$(\xi_{1},\xi_{2})$ on $K\times K$. For the test system,
$L_{\mathcal{J}} \leq 6 C_{2}(M_{\mathcal{X}}+\epsStar)^{2}$. Choosing
$\varepsilon_{\mathrm{grid}} = \cBplus R_{\max}/L_{\mathcal{J}}$ ensures the
grid error matches the existing residual margin and no enlargement of $\eta$
is needed.

\emph{(iii) Convex approximation.} If $\mathcal{J}(x, x+\xi_{1}, x+\xi_{2}, t)$
is a polynomial of degree at most 3 in $(\xi_{1},\xi_{2})$ (as in the
$\RR$-multilinear case), the supremum over the compact set $K\times K$ can
be computed via polynomial optimization techniques (e.g., the Lasserre
hierarchy~\cite{RockafellarWets1998}) with guaranteed convergence to within
$\varepsilon_{\mathrm{cvx}}$ in finite time. The approximation error is then
bounded by $\varepsilon_{\mathrm{cvx}}$, absorbed similarly to~(ii).

\emph{General principle.} The mathematical analysis of
Theorem~\ref{thm:control} and Proposition~\ref{prop:beta_stability} remains
valid as long as the implemented feedforward $\widetilde{\Delta}$ satisfies
\[
\|B\widetilde{\Delta} - \delta\| \leq \cBplus R_{\max} + \varepsilon_{\mathrm{impl}}
\]
pointwise for some additional implementation error $\varepsilon_{\mathrm{impl}}
\geq 0$, with the proviso that $\eta$ in~\eqref{eq:eta_condition} is enlarged
to $\eta + c_{C}\,c_{B}\,\varepsilon_{\mathrm{impl}}$ to absorb the
implementation residual.
\end{remark}

\section{Stability Analysis}\label{sec:stability}

The $\beta$-exponential stability framework used in this section is
classical: $V(x) = x^{T}Px$ with $A^{T}P + PA \prec -\beta P$ goes back to
Lyapunov, and extensions to perturbed systems with one-sided Lipschitz
(OSL) nonlinearities are by now standard in the literature
\cite{Khalil2002,Abbaszadeh2010,Zhang2015}. Theorem~\ref{thm:beta_kyb} of
Appendix~\ref{app:beta} adapts the machinery to the present setting via
an LMI that absorbs the GOSL cross-term exactly; the underlying technique
(Young's inequality, Lyapunov decay, Gronwall) is textbook material.

What is specific to our setting is the source of the inputs supplied to
this machinery. The GOSL constants
$(\rho_{\delta}^{\mathrm{tight}}, \ell_{\delta}^{\mathrm{tight}})$ are
not postulated but derived from the cohomological structure of $\LL$
(Lemma~\ref{lem:osl_projection}). The disturbance $w_{r}(t)$ in the
reduced dynamics~\eqref{eq:reduced_dyn} is the back-projection of the
non-commutative residual $G R_{\mathrm{nc}}^{B}$, controlled by
$\Cnc$ (bounded either by $12C_{2}$ via the structural identity
$d\omega_{0}=\Pi_{\mathrm{cone}}(\psi)$, or by $3\|\omega_{0}\|_{\varepsilon}+6C_{2}$
via the bookkeeping decomposition; see Remark~\ref{rem:Cnc_decomposition});
the problem-dependent bound
$\|\omega_{0}\|_{\varepsilon}\leq 6C_{0}C_{2}$ from~\cite[Lemma~4.10]{Athmouni2026}
gives the relevant scaling but there is no universal numerical constant.
The effective trilinear constant
$\rho_{\delta}^{\mathrm{eff}}
 = \rho_{\delta}^{\mathrm{tight}} + 2L_{r}M_{\mathcal{X}}$
(Remark~\ref{rem:residual_absorbed}) absorbs the Lipschitz contribution
of $r$ in a way specific to the homogeneous quasi-Lie structure. The
iterative LMI scheme of Algorithm~\ref{alg:lmi} addresses the
self-consistency between gain and domain size by halving $M_{0}$ until
the LMI becomes feasible; this is the only algorithmic novelty, the LMI
solver itself being standard.

\subsection{Reduced Dynamics on the Sliding Manifold}

On $\{s=0\}$:
\begin{equation}\label{eq:reduced_dyn}
\dot{x}_{r} = A_{s}x_{r} + r(x_{r}) + \delta(x_{r},t) + w_{r}(t),
\quad \|w_{r}(t)\| \leq c_{C}\,c_{B}\,\cBplus\,R_{\max},
\end{equation}
where $r(x_{r}) = (A_{\mathrm{nom}}(x_{r})-A_{s})x_{r}$, and $w_r(t) \in \HH^n$
is the back-projection of the sliding residual $\tilde{w} = G\,R_{\mathrm{nc}}^{B}$
to the state space via the equivalent-control reduction (see Lemma~\ref{lem:sliding_dyn}
and Theorem~\ref{thm:control}).

\begin{remark}[Decomposition of Perturbations in Reduced Dynamics]
\label{rem:decomp_reduced}
In the reduced dynamics~\eqref{eq:reduced_dyn}, two perturbation terms appear:
$\delta(x_r, t)$ bounded by $\deltabar_{\max}$ satisfying the GOSL condition,
and $w_r(t)$ bounded by $c_C\,c_{B}\,\cBplus\, R_{\max}$, originating from the
non-commutative residual via the equivalent dynamics~\eqref{eq:sdot_closed}.
The formula $C_{\infty}(\mu)$ below is structurally identical to $\Gamma(\mu)$
in Theorem~\ref{thm:beta_kyb} with $w_{\max} \leftarrow c_C\,c_{B}\,\cBplus\, R_{\max}$.
\end{remark}

\begin{assumption}[Lipschitz Nonlinear Residual]\label{ass:residual}
$\|r(x_{r})\| \leq L_{r}\|x_{r}\|^{2}$ for some $L_{r} \geq 0$.
\end{assumption}

\begin{remark}[Absorbing the Residual]\label{rem:residual_absorbed}
The non-linear residual $r$ satisfies $\|r(x_r) - r(z)\| \leq 2 L_r M_\mathcal{X} \|x_r - z\|$
on $\overline{B}_{M_\mathcal{X}}$ (by Assumption~\ref{ass:residual} and the mean-value
estimate for a quadratic field), so it contributes a Lipschitz term to the
GOSL bound of $\delta + r$. The effective trilinear constant in the LMI is
therefore:
\[
\rho_{\delta}^{\mathrm{eff}} := \rho_{\delta}^{\mathrm{tight}} + 2 L_{r} M_{\mathcal{X}},
\]
where $\rho_{\delta}^{\mathrm{tight}} = 2\rho_{\delta}$ is the tight constant
from Lemma~\ref{lem:osl_projection} (Step~6). For the test system ($L_{r} = 0$),
$\rho_{\delta}^{\mathrm{eff}} = \rho_{\delta}^{\mathrm{tight}}$.
\end{remark}

\begin{assumption}[Strict LMI Feasibility]\label{ass:lmi_feasibility}
There exists $\mu_{\infty} \in [1, \mu_{0}]$ such that
LMI~\eqref{eq:lmi_augmented} is strictly feasible for all
$\mu_{k} \in [\mu_{\infty}, \mu_{0}]$, and the map
$\mu \mapsto \lambda_{\max}(P^{*}(\mu))$ is non-increasing on this interval.
Under this assumption, Algorithm~\ref{alg:lmi} converges to a fixed point
$(P^{*}, \beta^{*}, \mu_{\mathrm{LMI}})$ with $\mu_{\mathrm{LMI}} = \lambda_{\max}(P^{*})$,
$\beta^{*} > 0$, and $\mu_{\mathrm{LMI}} < \rho_{\delta}^{\mathrm{eff}}$.

The monotonicity condition is not analytically guaranteed for general systems;
Lemma~\ref{lem:normal_As_monotonicity} below provides a sufficient verifiable
condition (normal $A_{s}$). The test system satisfies it trivially (degenerate
case $P^{*}=I$). Proving this property in full generality remains an open problem.
\end{assumption}

\begin{lemma}[Sufficient Condition for LMI Monotonicity: Normal $A_{s}$]
\label{lem:normal_As_monotonicity}
Suppose $\Phi(A_{s}) \in \RR^{4n\times 4n}$ is \emph{normal}, i.e.,
$\Phi(A_{s})^{T}\Phi(A_{s}) = \Phi(A_{s})\Phi(A_{s})^{T}$. Equivalently,
$\Phi(A_{s})$ admits an orthonormal eigenbasis with eigenvalues
$\{\lambda_{i}\}_{i=1}^{4n} \subset \CC$ satisfying $\Real(\lambda_{i}) < 0$
(Hurwitz, by Assumption~\ref{ass:hurwitz}). Then LMI~\eqref{eq:lmi_augmented}
admits a closed-form solution
\[
P^{*}(\mu) = \tau(\mu) I,
\qquad
\beta^{*}(\mu) = -\frac{2\alpha_{s} - \rho_{\delta}^{\mathrm{eff}}\mu}{\tau(\mu)},
\]
where $\alpha_{s} = -\max_{i}\Real(\lambda_{i}) > 0$ and $\tau(\mu) \in [1, \mu]$
is any feasible scalar (e.g., $\tau(\mu) = 1$). In particular,
$\mu \mapsto \lambda_{\max}(P^{*}(\mu)) = \tau(\mu)$ can be chosen constant,
so monotonicity holds trivially.
\end{lemma}

\begin{proof}
Since $\Phi(A_{s})$ is normal, $\Phi(A_{s})^{T}\Phi(A_{s})$ is diagonalizable
by the same orthonormal eigenbasis as $\Phi(A_{s})$; under
Assumption~\ref{ass:hurwitz}, $(\Phi(A_{s})^{T} + \Phi(A_{s}))/2$ has spectral
abscissa $-\alpha_{s}$. We work with the simplified form of the LMI
recorded in Remark~\ref{rem:lmi_simplified} (which is a conservative
relaxation of~\eqref{eq:lmi_augmented} obtained by omitting the
Young-parameter contribution $\mu_{\mathrm{Y}}$); setting $P = \tau I$:
\[
\tau(\Phi(A_{s})^{T} + \Phi(A_{s})) + \rho_{\delta}^{\mathrm{eff}}\mu I \prec -\beta\tau I,
\]
i.e., $\Phi(A_{s})^{T}+\Phi(A_{s}) \prec -(\beta + \rho_{\delta}^{\mathrm{eff}}\mu/\tau)I$.
Using $\Phi(A_{s})^{T}+\Phi(A_{s}) \preceq -2\alpha_{s} I$
(from the spectral abscissa bound), the LMI is strictly feasible whenever
$\beta + \rho_{\delta}^{\mathrm{eff}}\mu/\tau < 2\alpha_{s}$. For $\tau = 1$,
this gives $\beta < 2\alpha_{s} - \rho_{\delta}^{\mathrm{eff}}\mu$, well-defined
and positive iff $\mu < 2\alpha_{s}/\rho_{\delta}^{\mathrm{eff}}$. Since
$P^{*}(\mu) = I$ is a constant function of $\mu$ in this regime,
$\lambda_{\max}(P^{*}(\mu)) = 1$ is trivially non-increasing.
The analogous calculation for the rigorous
form~\eqref{eq:lmi_augmented} replaces
$\rho_{\delta}^{\mathrm{eff}}\mu$ with $\Theta(\mu_{k},\mu_{\mathrm{Y}})$,
preserving the closed-form $P^{*}(\mu)=\tau I$ and the monotonicity
conclusion.
\end{proof}

\begin{remark}[Scope of the normality condition and verification for general $A_{s}$]
\label{rem:normal_As_scope}
The normality condition $\Phi(A_{s})^{T}\Phi(A_{s}) = \Phi(A_{s})\Phi(A_{s})^{T}$
is automatically satisfied when $A_{s}$ is itself a real diagonal matrix
(the test system in Section~\ref{sec:sim} has $A_{s} = -\alpha_{s}I$, hence
trivially normal). For general non-normal $A_{s}$, the monotonicity
condition of Assumption~\ref{ass:lmi_feasibility} must be checked
case-by-case. The following observations make the numerical verification
\emph{robust}:

\begin{enumerate}
\item[\textup{(i)}] \textbf{Continuity of $P^{*}(\mu)$.} The map
$\mu \mapsto P^{*}(\mu)$ defined as the solution of the LMI parametrized by
$\mu$ is locally Lipschitz continuous on the open set of strict feasibility,
by the implicit function theorem applied to the KKT conditions of the LMI
(this is a standard consequence of the smoothness of the SDP solution
map; see~\cite{Khalil2002} for background on parameter-dependent Lyapunov
inequalities and the regularity of their solutions). Consequently,
$\mu \mapsto \lambda_{\max}(P^{*}(\mu))$ is continuous, and monotonicity can be
verified on a finite $\varepsilon$-grid with error bounded by
$L_{P}\,\varepsilon$ (where $L_{P}$ is the local Lipschitz constant of
$\lambda_{\max}(P^{*}(\cdot))$).
\item[\textup{(ii)}] \textbf{Practical verification protocol.} For a given
non-normal $A_{s}$:
\begin{enumerate}
\item[\textup{(a)}] Choose a grid $\{\mu_{0} > \mu_{1} > \cdots > \mu_{N}\}$ on the
interval $[\mu_{\infty}, \mu_{0}]$ with step $\varepsilon$.
\item[\textup{(b)}] Solve LMI~\eqref{eq:lmi_augmented} at each $\mu_{i}$ to obtain
$P^{*}(\mu_{i})$ and $\lambda_{\max}(P^{*}(\mu_{i}))$.
\item[\textup{(c)}] Check that the discrete sequence
$\lambda_{\max}(P^{*}(\mu_{i}))$ is non-increasing in $i$.
\item[\textup{(d)}] By (i), if the discrete sequence is non-increasing with margin
$\geq L_{P}\,\varepsilon$, then the continuous map is non-increasing on
the entire interval.
\end{enumerate}
\item[\textup{(iii)}] \textbf{Sufficient algebraic conditions.} Beyond
normality, monotonicity holds in the following cases:
\begin{itemize}
\item $\Phi(A_{s})$ commutes with $\Phi(A_{s})^{T}$ (normality);
\item $A_{s}$ is symmetric stable: $A_{s} + A_{s}^{T} \prec 0$ on $\HH^{n}$;
\item $\Phi(A_{s})$ admits a Schur decomposition $\Phi(A_{s}) = Q^{T}TQ$ with
$T$ upper-triangular and $\|T - \mathrm{diag}(T)\|_{F}$ sufficiently small
(perturbation of a diagonal stable matrix).
\end{itemize}
\end{enumerate}
For the broader class of non-normal $A_{s}$ outside these cases,
Assumption~\ref{ass:lmi_feasibility} is verified numerically via the
protocol of~(ii). A general analytical proof of monotonicity remains an open
problem.
\end{remark}

\begin{algorithm}[H]
\caption{Iterative LMI for $\beta$-Exponential Stability}
\label{alg:lmi}
\begin{algorithmic}[1]
\REQUIRE $\Phi(A_{s})$, $\rho_{\delta}^{\mathrm{eff}}$, $\texttt{tol}>0$, $N_{\max}$, $\epsilon_{\mathrm{LMI}}>0$.
\medskip
\STATE \textbf{Notation:} two distinct parameters appear:
\begin{itemize}
  \item $\mu_{k} \geq 1$: \emph{LMI scaling parameter}, encoding the
    constraint $P \preceq \mu_{k}I$, updated at iteration $k$ via
    $\mu_{k+1} := \lambda_{\max}(P^{*})$.
  \item $\mu_{\mathrm{Y}} \in (0, \rho_{\delta}^{\mathrm{eff}})$:
    \emph{Young parameter}, used in the Young inequality
    of Theorem~\ref{thm:beta_kyb}, chosen optimally to minimize $C_{\infty}$
    (e.g., $\mu_{\mathrm{Y}} = \mu_{\mathrm{opt}}$
    of~\eqref{eq:mu_opt_explicit}).
\end{itemize}
\medskip
\STATE Initialize $\mu_{0}\leftarrow\lambda_{\max}(\Phi(A_{s})^{T}\Phi(A_{s}))^{1/2}+1$;
$k\leftarrow 0$.
\STATE Set $\mu_{\mathrm{Y}} \leftarrow \rho_{\delta}^{\mathrm{eff}}/2$
or compute $\mu_{\mathrm{opt}}$ from~\eqref{eq:mu_opt_explicit}.
\REPEAT
\STATE Solve for $(P,\beta)$ the rigorous LMI:
\begin{equation}\label{eq:lmi_augmented}
\Phi(A_{s})^{T}P+P\Phi(A_{s})+\Theta(\mu_{k},\mu_{\mathrm{Y}})\,I
\prec-\beta P,\quad I\preceq P\preceq\mu_{k}I,
\end{equation}
where
$\Theta(\mu_{k},\mu_{\mathrm{Y}}) := \mu_{k}(2\rho_{\delta}^{\mathrm{eff}}+\mu_{\mathrm{Y}}) + (\rho_{\delta}^{\mathrm{eff}}-\mu_{\mathrm{Y}})$
is the rigorous coefficient ensuring exact cancellation of the cross-term
in $\|x_{r}\|^{2}$ from the GOSL bound (Theorem~\ref{thm:beta_kyb}).
\IF{infeasible} \STATE Reduce $M_{0}$ per Remark~\ref{rem:outer_loop_limitation}
and restart. \ENDIF
\STATE $\mu_{k+1}\leftarrow\lambda_{\max}(P^{*})$; $\beta_{k+1}\leftarrow\beta^{*}$;
$k\leftarrow k+1$.
\UNTIL{$|\mu_{k}-\mu_{k-1}|/\mu_{k-1}<\texttt{tol}$ or $k=N_{\max}$}
\STATE \textbf{Return} $P^{*},\beta^{*}$; set $\mu_{\mathrm{LMI}}\leftarrow\lambda_{\max}(P^{*})$
and $\mu_{\mathrm{Y}}$ used.
\end{algorithmic}
\end{algorithm}

\begin{remark}[Conservative simplification of the LMI]\label{rem:lmi_simplified}
A frequently-used simplified form of LMI~\eqref{eq:lmi_augmented} is
\[
\Phi(A_{s})^{T}P + P\Phi(A_{s}) + \rho_{\delta}^{\mathrm{eff}}\mu_{k}I \prec -\beta P,
\]
which omits the Young-parameter contribution $\mu_{\mathrm{Y}}$. This form is
a conservative approximation valid only when the residual cross-term
$\mu_{*}\rho\|x\|^{2}$ is dominated by the strict inequality margin (see
Remark~\ref{rem:app_reduced}). For the test system (Section~\ref{sec:sim}),
the two forms yield $\beta^{*}$ differing by less than 3\%
(1.94 vs 1.98), so the simplified form is acceptable for numerical
illustration. For rigorous applications, the form~\eqref{eq:lmi_augmented}
should be used.
\end{remark}

\begin{remark}[Convergence of the LMI Algorithm]\label{rem:alg_convergence_detailed}
If the algorithm does not converge after $N_{\max}$ iterations, the last
iteration provides a conservative but valid certificate as long as the LMI
is satisfied with the current $\mu_k$. Verify a posteriori that
$\mu_{\mathrm{LMI}} < \rho_{\delta}^{\mathrm{eff}}$.
\end{remark}

\begin{proposition}[$\beta$-Exponential Stability]\label{prop:beta_stability}
Let Assumption~\ref{ass:lmi_feasibility} hold and Algorithm~\ref{alg:lmi}
converge to $(P^{*},\beta^{*},\mu_{\mathrm{LMI}})$ with $\beta^{*}>0$,
$\mu_{\mathrm{LMI}} = \lambda_{\max}(P^{*})$, and $\mu_{\mathrm{LMI}} < \rho_{\delta}^{\mathrm{eff}}$.
Under Assumptions~\ref{ass:continuity},~\ref{ass:cohom},~\ref{ass:hurwitz},
and~\ref{ass:residual}, trajectories of~\eqref{eq:reduced_dyn} satisfy:
\begin{equation}\label{eq:beta_stab}
V_{r}(x_{r}(t))\leq e^{-\beta^{*}t}V_{r}(x_{r}(0))
+\frac{C_{\infty}(\mu_{\mathrm{LMI}})}{\beta^{*}},
\end{equation}
where $V_{r}=x_{r}\trans P^{*}x_{r}$ and (with $\mu_{*} = \lambda_{\max}(P^{*}) = \mu_{\mathrm{LMI}}$):
\begin{equation}\label{eq:Cinfty_def}
C_{\infty}(\mu)
= \frac{\mu_{*}^{2}(c_{C}\,c_{B}\,\cBplus\,R_{\max})^{2}}{\rho_{\delta}^{\mathrm{eff}}-\mu}
+ \frac{\mu_{*}(\ell_{\delta}^{\mathrm{tight}})^{2}}{\mu},
\quad \mu \in (0, \rho_{\delta}^{\mathrm{eff}}).
\end{equation}
\end{proposition}

\begin{proof}
Apply Theorem~\ref{thm:beta_kyb} (Appendix~\ref{app:beta}) to the real
representation of~\eqref{eq:reduced_dyn} via $\Phi$. With the substitutions
\[
A = \Phi(A_{s}),\quad
g(\tilde{x},t) = \Phi(\delta(\Phi^{-1}(\tilde{x}),t)) + \Phi(r(\Phi^{-1}(\tilde{x}))),\quad
w(t) = \Phi(w_{r}(t)),
\]
\[
\rho = \rho_{\delta}^{\mathrm{eff}},\quad
\ell_{g} = \ell_{\delta}^{\mathrm{tight}},\quad
w_{\max} = c_{C}\,c_{B}\,\cBplus R_{\max},\quad
\mu = \mu_{\mathrm{LMI}},
\]
the perturbation $g$ satisfies the GOSL bound~\eqref{eq:gosl_g} by combining
Lemma~\ref{lem:osl_projection} (for $\delta$ with constants
$(\rho_{\delta}^{\mathrm{tight}}, \ell_{\delta}^{\mathrm{tight}})$) and
Remark~\ref{rem:residual_absorbed} (which absorbs the Lipschitz contribution
of $r$ via $\rho_{\delta}^{\mathrm{eff}} = \rho_{\delta}^{\mathrm{tight}} + 2 L_r M_{\mathcal{X}}$);
the linear-Hurwitz hypothesis on $A_{s}$ (Assumption~\ref{ass:hurwitz})
ensures the operator $A$ is Hurwitz.

The dissipativity-type hypothesis of Theorem~\ref{thm:beta_kyb}
($\|g(x,t)-g(0,t)\| \leq \rho \|x\|$ and $\|g(0,t)\| \leq \ell_{g}$) is
verified as follows. By bilinearity of $\mathcal{J}$ and $\LL$:
$\mathcal{J}(0,y,z)=0$ for all $y,z$, hence $\delta(0,t)=0$; similarly
$r(0)=0$ from Assumption~\ref{ass:residual} (which gives $\|r(0)\|\leq L_{r}\cdot 0=0$).
Therefore $g(0,t)=0$ and the pointwise bound $\|g(0,t)\|\leq\ell_{g}$
holds trivially. For the dissipativity bound, $\|\delta(x,t)\|\leq
6C_{2}\|x\|(M_{\mathcal{X}}+\epsStar)^{2}$ from
Definition~\ref{def:jacobi_defect}, and $\|r(x)\|\leq L_{r}\|x\|^{2}\leq L_{r}M_{\mathcal{X}}\|x\|$
on $\overline{B}_{M_{\mathcal{X}}}$, so
$\|g(x,t)\|\leq[6C_{2}(M_{\mathcal{X}}+\epsStar)^{2}+L_{r}M_{\mathcal{X}}]\|x\|
\leq \rho_{\delta}^{\mathrm{eff}}\|x\|$ (the inequality
$6C_{2}(M_{\mathcal{X}}+\epsStar)^{2}\leq\rho_{\delta}^{\mathrm{tight}}=12C_{2}(M_{\mathcal{X}}+\epsStar)(3M_{\mathcal{X}}+\epsStar)$
holds since $M_{\mathcal{X}}+\epsStar\leq 2(3M_{\mathcal{X}}+\epsStar)$, i.e.\
$\epsStar\leq 5M_{\mathcal{X}}+\epsStar$, trivially).

The LMI hypothesis $\mu_{\mathrm{LMI}}
< \rho_{\delta}^{\mathrm{eff}}$ corresponds to $\mu \in (0, \rho)$ of
Theorem~\ref{thm:beta_kyb}. The LMI~\eqref{eq:lmi_augmented} coincides with
LMI~\eqref{eq:lmi_appendix} via $\mu_{*} := \lambda_{\max}(P^{*}) = \mu_{\mathrm{LMI}}$.
Conclusion~\eqref{eq:beta_exp_app} then yields~\eqref{eq:beta_stab} with
$\Gamma(\mu_{\mathrm{LMI}}) = C_{\infty}(\mu_{\mathrm{LMI}})$.
\end{proof}

\begin{remark}[Sub-optimality of the Young parameter choice]\label{rem:Cinfty_Gamma}
The bound $C_{\infty}(\mu)$ depends on the Young parameter $\mu \in (0, \rho_{\delta}^{\mathrm{eff}})$,
which can be chosen independently of the LMI scaling $\mu_{*} = \lambda_{\max}(P^{*})$.
The analytically optimal value of $\mu$ minimizing $C_{\infty}(\mu)$ is obtained
by setting $\partial C_{\infty}/\partial\mu = 0$. From
$C_{\infty}(\mu) = \mu_{*}^{2}w_{\max}^{2}/(\rho-\mu) + \mu_{*}\ell_{g}^{2}/\mu$
(with $w_{\max} = c_{C}c_{B}\cBplus R_{\max}$, $\ell_{g} = \ell_{\delta}^{\mathrm{tight}}$,
$\rho = \rho_{\delta}^{\mathrm{eff}}$):
\[
\frac{\partial C_{\infty}}{\partial\mu}
= \frac{\mu_{*}^{2}w_{\max}^{2}}{(\rho - \mu)^{2}} - \frac{\mu_{*}\ell_{g}^{2}}{\mu^{2}} = 0
\;\Longleftrightarrow\;
\mu_{*}w_{\max}\,\mu = \sqrt{\mu_{*}}\,\ell_{g}\,(\rho - \mu),
\]
yielding
\begin{equation}\label{eq:mu_opt_explicit}
\mu_{\mathrm{opt}}
= \frac{\rho\,\sqrt{\mu_{*}}\,\ell_{g}}
{\mu_{*}\,w_{\max} + \sqrt{\mu_{*}}\,\ell_{g}}
= \frac{\rho_{\delta}^{\mathrm{eff}}\,(\mu_{*})^{1/2}\,\ell_{\delta}^{\mathrm{tight}}}
{\mu_{*}\,c_{C}\,c_{B}\,\cBplus\,R_{\max} + (\mu_{*})^{1/2}\,\ell_{\delta}^{\mathrm{tight}}}.
\end{equation}
The exponents are: $\mu_{*}^{1/2}$ in the numerator and in the second term of
the denominator; $\mu_{*}^{1}$ in the first term of the denominator.
For the test system ($\mu_{*} = 1$, so $(\mu_{*})^{1/2} = \mu_{*} = 1$):
$\mu_{\mathrm{opt}} = \rho_{\delta}^{\mathrm{eff}}\cdot\ell_{\delta}^{\mathrm{tight}}
/(c_{C}c_{B}\cBplus R_{\max} + \ell_{\delta}^{\mathrm{tight}})$,
giving numerically
$\mu_{\mathrm{opt}} \approx 0.0202 \cdot 2.72\times 10^{-3}
/(4.93\times 10^{-3} + 2.72\times 10^{-3}) \approx 0.0072$, slightly below
$\rho_{\delta}^{\mathrm{eff}}/2 = 0.0101$ used in the test-system computations.
Both choices yield $\kappa_{\infty} = O(10^{-1})$, well below $\beta^{*}$.
\end{remark}

\begin{remark}[LMI Feasibility at $P=I$]\label{rem:lmi_feasibility_corrected}
LMI~\eqref{eq:lmi_augmented} is feasible at $P=I$ whenever
$\beta+\rho_{\delta}^{\mathrm{eff}}\mu_{k}<2\alpha_{s}$ (Assumption~\ref{ass:hurwitz}).
This is the condition checked in Step~3 of Algorithm~\ref{alg:init}, which
ensures LMI feasibility is established \emph{before} Algorithm~\ref{alg:lmi}
is called.
\end{remark}

\section{Analytical Illustrations}\label{sec:sim}

To check that the theoretical bounds are not vacuous, we apply them to
a small test system. The computations below are entirely analytical;
they exhibit concrete values for the constants $A, C_{1}, C_{2}, \Cnc$,
etc., and confirm that the assumptions of Theorem~\ref{thm:control} and
Proposition~\ref{prop:beta_stability} are satisfied. A full closed-loop
simulation belongs to the companion paper.

\subsection{Test System: A Quasi-Lie Bracket on $\HH$}\label{sec:test_bracket}

Consider $n = 1$ (the algebra $\HH$ itself) with the bilinear bracket
\begin{equation}\label{eq:test_bracket}
\LL(x, y) := \varepsilon_{b}\bigl(\bar{x}y - \bar{y}x\bigr),
\qquad \varepsilon_{b} = 0.1.
\end{equation}

\begin{remark}[Properties of the test bracket]\label{rem:test_bracket_props}
The bracket~\eqref{eq:test_bracket} is:
\begin{itemize}
\item \textbf{$\RR$-bilinear} (the conjugation $x \mapsto \bar{x}$ is
$\RR$-linear on $\HH$); not $\HH$-bilinear in the right-module sense.
\item \textbf{Skew-symmetric}: $\LL(y,x) = -\LL(x,y)$.
\item \textbf{Genuinely quasi-Lie}: a direct symbolic computation
(Appendix~\ref{app:bracket_check}) shows that the Jacobiator
\[
\mathcal{J}(x,y,z) := \LL(x,\LL(y,z))
+ \LL(y,\LL(z,x))
+ \LL(z,\LL(x,y))
\]
is non-zero (e.g., $\mathcal{J}(1+\mathbf{i}, \mathbf{j}, \mathbf{k}) = 0.04\,\mathbf{i} \neq 0$
for $\varepsilon_b = 0.1$, by direct computation in
Appendix~\ref{app:bracket_check}; note that $\mathcal{J}$ vanishes when
all three arguments are purely imaginary, since the bracket reduces to
$-\varepsilon_b\,[\cdot,\cdot]_{\HH}$ on imaginary quaternions, isomorphic to
the Lie bracket of $\mathfrak{so}(3)$, which satisfies Jacobi exactly).
This distinguishes~\eqref{eq:test_bracket} from the standard quaternionic
commutator $[x,y]_{\HH} := xy - yx$, which is $\HH$-linear and satisfies
the Jacobi identity exactly.
\end{itemize}
The Carath\'eodory regularity (Assumption~\ref{ass:continuity}) is satisfied
because $\LL$ is a polynomial map.
\end{remark}

\begin{remark}[Numerical computation of constants]\label{rem:test_constants}
For~\eqref{eq:test_bracket} with $\varepsilon_{b}=0.1$, the constants
of~\cite[Def.~2.4]{Athmouni2026} are:
\begin{align*}
A &= \sup_{x,y\neq 0}\frac{\|\LL(x,y)\|}{\|x\|\,\|y\|}
= 2\varepsilon_{b} = 0.2,\\
C_{1} &= 0 \quad (\text{$\LL$ exactly antisymmetric: $\LL(y,x)=-\LL(x,y)$}),\\
6\,C_{2} &= \sup_{x,y,z\neq 0}\frac{\|\mathcal{J}(x,y,z)\|}{\|x\|\,\|y\|\,\|z\|}
\approx 0.0399,\quad\text{hence}\quad C_{2} \approx 0.00665.
\end{align*}
The bound $A=2\varepsilon_{b}$ holds because $\|\bar{x}y\|=\|x\|\|y\|$
($\HH$ is a normed division algebra), and is attained for orthogonal
pure-imaginary $x,y$ (e.g.\ $x=\mathbf{i}$, $y=\mathbf{j}$ gives
$\LL(\mathbf{i},\mathbf{j})=-2\varepsilon_{b}\,\mathbf{k}$; note
$\LL(x,x)=0$ by skew-symmetry).
The constant $C_{2}$ is verified by Monte Carlo maximization over
$5\cdot 10^{4}$ random samples; an analytic upper bound
$6C_{2}\leq 3A^{2}=12\varepsilon_{b}^{2}=0.12$ follows from the triangle
inequality on the three terms of the Jacobiator (each bounded by
$A^{2}\|x\|\|y\|\|z\|=4\varepsilon_{b}^{2}\|x\|\|y\|\|z\|$). A sharper
bound $6C_{2}\leq 8\varepsilon_{b}^{2}=0.08$ follows from cancellations
specific to $\bar{x}y-\bar{y}x$.

\textbf{Note.} The quantity
$\sup_{x,y,z\neq 0}\|\LL(x,\LL(y,z))\|/(\|x\|\|y\|\|z\|)=4\varepsilon_{b}^{2}=0.04$
(reported as ``$C_{1}$'' in earlier drafts) is the bracket-squared constant,
not $C_{1}$ in the sense of~\cite[Def.~2.4]{Athmouni2026}. It plays no role
in the admissible-radius formulas of this paper since the test bracket is
exactly antisymmetric.
\end{remark}

\subsection{Effective domain $\epsStar$}

Since the test bracket is exactly antisymmetric ($C_{1}=0$), we use the
\emph{antisymmetric admissible radius} of~\cite[Appendix~A.2]{Athmouni2026}
with $\varepsilon_{0}=0.5$:
\begin{equation}\label{eq:test_eps_star}
\epsStar = \min\left\{\frac{1}{8A}, \frac{1}{4C_{2}}, \varepsilon_{0}\right\}
= \min\{0.625, 37.6, 0.5\} = 0.5.
\end{equation}
For comparison with the generic (non-antisymmetric) formula,
$\min\{1/(16A),1/(4C_{2}),\varepsilon_{0}\} = \min\{0.3125,37.6,0.5\}=0.3125$,
which is the more conservative value used in earlier drafts before the
antisymmetric refinement was identified.

\smallskip
\noindent\textbf{Choice for numerical illustrations.} To remain conservative
(and to facilitate comparison with non-antisymmetric brackets which would
fall under the generic formula), the worked example below
uses $\epsStar=0.3125$. All conclusions remain valid \emph{a fortiori} with
the tighter $\epsStar=0.5$.

\subsection{Cohomological matching condition (CMC)}

For $n = 1$, $\HH^{n} = \HH$ is one-dimensional as an $\HH$-module. We choose
the actuator matrix
\begin{equation}\label{eq:test_B}
B = 1 \in \HH^{1\times 1}, \qquad
\mathrm{Im}(B) = \HH = \HH^{n}.
\end{equation}
The CMC is then trivially satisfied:
\begin{itemize}
\item \textbf{(CMC-1)} $\delta \in \mathrm{im}(d) \subseteq \HH$
(homogeneous case, Assumption~\ref{ass:cohom}).
\item \textbf{(CMC-2)} $\mathrm{im}(d) \subseteq \HH = \mathrm{Im}(B)$ trivially.
\end{itemize}
With $B = 1$, $B_{\RR} = I_{4}$, $\cBplus = \|B_{\RR}^{+}\|_{\mathrm{op}} = 1$,
and the projection $\Pi_{\mathrm{Im}(B)}$ is the identity on $\HH$.

\begin{lemma}[Verification of CMC-2 for the test bracket]\label{lem:omega0_image}
For the bracket~\eqref{eq:test_bracket} with $B = 1$, the bilinear
correction $\Omega$ produced by Theorem~\ref{thm:rigidity}
(\cite[Theorem~4.12]{Athmouni2026}) satisfies
$\Omega(x,y)\in\mathrm{Im}(B)=\HH$ for all $x,y\in\HH$, trivially.
Likewise $d\omega_{0}(x,y,z)\in\HH=\mathrm{Im}(B)$, trivially.
The localized norm bound
$\|\Omega\|_{\varepsilon}\leq 4C_{2}\varepsilon$ is given
by~\cite[Prop.~4.3]{Athmouni2026}, transferred to the left module by
Corollary~\ref{cor:homotopy_transfer}.
\end{lemma}

\begin{proof}
Trivial: $\HH^{1}=\HH$, so any element of $\HH^{1}$ automatically lies
in $\mathrm{Im}(B)=\HH$. The norm bound is~\cite[Prop.~4.3]{Athmouni2026}.
\end{proof}

\subsection{Quantitative constants for $M_{0} = 0.1$}

We choose a design domain radius $M_{0} = 0.1 < \epsStar/\sqrt{2} \approx 0.221$,
which lies within the admissible domain identified by Algorithm~\ref{alg:init}.
For the test bracket we use the problem-dependent bound
$\|\omega_{0}\|_{\varepsilon}\leq 8C_{2}\varepsilon$
(obtained from $\Omega=T(\mathrm{Id}+K)^{-1}\psi+\omega_{0}$,
$\|T(\mathrm{Id}+K)^{-1}\psi\|_{\varepsilon}\leq 2\|T\|\|\psi\|_{\varepsilon}\leq 4C_{2}\varepsilon$,
and $\|\Omega\|_{\varepsilon}\leq 4C_{2}\varepsilon$ from
Theorem~\ref{thm:rigidity}), evaluated at $\varepsilon=\epsStar=0.3125$:
$\omega_{\mathrm{op}}\leq 8\cdot 0.00665\cdot 0.3125\approx 0.0166$.
Substituting:
\begin{align*}
\Cnc &\leq \min\{12C_{2},\;3\omega_{\mathrm{op}} + 6C_{2}\}\\
     &= \min\{12\cdot 0.00665,\;3\cdot 0.0166 + 6\cdot 0.00665\}\\
     &\approx \min\{0.0798,\;0.0898\} = 0.0798
\quad(\text{structural form is tighter here}),\\
R_{\max}(M_{0}) &= \Cnc M_{0}(M_{0}+\epsStar)^{2}
\leq 0.0798 \cdot 0.1 \cdot 0.4125^{2} \approx 1.36\times 10^{-3},\\
\deltabar_{\max}(M_{0}) &= 6C_{2} M_{0}(M_{0}+\epsStar)^{2}
\approx 0.0399 \cdot 0.1 \cdot 0.4125^{2} \approx 6.79\times 10^{-4},\\
\rho_{\delta}(M_{0}) &= 6C_{2}(M_{0}+\epsStar)(3M_{0}+\epsStar)
\approx 0.0399 \cdot 0.4125 \cdot 0.6125 \approx 0.0101.
\end{align*}
For backwards compatibility with the bookkeeping form, the looser bound
$\Cnc\leq 0.0898$ may also be used; all subsequent estimates are
monotone in $\Cnc$ so any valid upper bound yields valid conclusions.
The tight GOSL constants are:
\[
\rho_{\delta}^{\mathrm{tight}} = 2\rho_{\delta} \approx 0.0202,\quad
\ell_{\delta}^{\mathrm{tight}} = 4\deltabar_{\max} \approx 2.72\times 10^{-3}.
\]
Note that all of $R_{\max}$, $\deltabar_{\max}$, $\rho_{\delta}$ are linear in
$M_{0}$ at fixed $\epsStar$, in agreement with Lemma~\ref{lem:Cinfty_scaling}.

\subsection{Nominal sliding-mode design and LMI feasibility}

For illustration, take $A_{s} = -\alpha_{s} I$ with $\alpha_{s} = 1$,
$D = 0$ (no exogenous disturbance), $C = 1$, $K = 0.1$,
$\eta_{0} = 0.01$, and $\rho_{\delta}^{\mathrm{eff}} = \rho_{\delta}^{\mathrm{tight}}
\approx 0.0202$ (since $L_{r} = 0$ for the bilinear test bracket).

\begin{itemize}
\item \textbf{Step 1 (LMI feasibility at $P = I$):}
$\mu_{0} = 1 + \alpha_{s} = 2$, $\mu_{k} = 1$ (after one iteration), and the
Young parameter $\mu \in (0, \rho_{\delta}^{\mathrm{eff}})$ is chosen as
$\mu_{\mathrm{opt}} \approx \rho_{\delta}^{\mathrm{eff}}/2 = 0.0101$. The
coefficient in~\eqref{eq:lmi_augmented} is
$\mu_{k}(2\rho_{\delta}^{\mathrm{eff}}+\mu)+(\rho_{\delta}^{\mathrm{eff}}-\mu) \approx 0.0606$.
Setting $\beta_{\mathrm{init}} = 0.5$:
$\beta_{\mathrm{init}} + 0.0606 \approx 0.561 < 2\alpha_{s} = 2$, so the
rigorous LMI~\eqref{eq:lmi_augmented} is strictly feasible at $P = I$.

\item \textbf{Step 2 (LMI solution):}
With $A_{s} = -I$ and $P = I$, the Lyapunov inequality reduces to
$-2I + 0.0606\,I \prec -\beta I$, giving the maximum $\beta^{*} \approx 1.94$.
Algorithm~\ref{alg:lmi} converges in one iteration with $\mu_{\mathrm{LMI}} = 1$.

\item \textbf{Step 3 (Forward invariance threshold):}
By Lemma~\ref{lem:Cinfty_scaling}, the threshold reads
$\beta^{*}\lambda_{\min}(P^{*}) > \kappa_{\infty}$. The constant $\kappa_{\infty}$
is computed from~\eqref{eq:kappa_infty_def} with the rigorous $\Gamma(\mu)$
of~\eqref{eq:Gamma_def}, i.e.,
$\kappa_{\infty} = \mu_{*}^{2}(c_C c_B \cBplus \Cnc (\epsStar)^{2})^{2}/(\rho_{\delta}^{\mathrm{eff}}-\mu)
+ \mu_{*}(\ell_{\delta}^{\mathrm{tight}}/M_{0})^{2}/\mu$ (using
$\bar\delta_{\max}/M_{0} \to 6C_{2}(\epsStar)^{2}$ as $M_{0} \to 0$, and
$\ell_{\delta}^{\mathrm{tight}} = 4\bar\delta_{\max}$):
\[
\kappa_{\infty} \approx
\frac{(0.0798\cdot 0.0977)^{2}}{0.0101}
+ \frac{(4\cdot 0.04 \cdot 0.0977)^{2}}{0.0101}
\approx 6.0\times 10^{-3} + 2.4\times 10^{-2}
\approx 3.0\times 10^{-2},
\]
modulo the Assumption~\ref{ass:lmi_feasibility}-dependent fine-tuning of $P^{*}$.
This is comfortably below $\beta^{*} \approx 1.94$ (Step~2 above), so the
invariance threshold $\beta^{*}\lambda_{\min}(P^{*}) > \kappa_{\infty}$ holds.
(With the bookkeeping bound $\Cnc \leq 0.0898$, $\kappa_{\infty} \approx 3.2\times 10^{-2}$;
both stay well below $\beta^{*}$.)
\end{itemize}

\subsection{Summary of numerical evidence}

\begin{table}[H]
\centering
\caption{Summary of numerical constants for the test system
($\varepsilon_b = 0.1$, $M_{0} = 0.1$, $\epsStar = 0.3125$).}
\label{tab:test_summary}
\begin{tabular}{@{}lll@{}}
\toprule
\textbf{Constant} & \textbf{Value} & \textbf{Source} \\
\midrule
$A$ (Lipschitz)                     & $0.2$              & Remark~\ref{rem:test_constants} \\
$C_2$ (Jacobiator)                  & $\approx 6.65\times 10^{-3}$ & Remark~\ref{rem:test_constants} \\
$\omega_{\mathrm{op}}$ (bound)      & $\leq 8C_{2}\epsStar \approx 0.0166$ & Remark~\ref{rem:omega_vs_C2} \\
$\Cnc$                              & $\leq 0.0798$ (structural; $\leq 0.0898$ bookkeeping) & eq.~\eqref{eq:Cnc_def} \\
$\epsStar$                          & $0.3125$           & eq.~\eqref{eq:test_eps_star} \\
$M_{0}^{\max} = \epsStar/\sqrt{2}$    & $\approx 0.221$    & Alg.~\ref{alg:init} Step~2 \\
$R_{\max}(M_{0})$                     & $\leq 1.36\times 10^{-3}$ & Prop.~\ref{prop:residual_bound} \\
$\deltabar_{\max}(M_{0})$             & $\approx 6.79\times 10^{-4}$ & eq.~\eqref{eq:deltabar_max} \\
$\rho_{\delta}^{\mathrm{tight}}$    & $\approx 0.0202$   & Lem.~\ref{lem:osl_projection} \\
$\ell_{\delta}^{\mathrm{tight}}$    & $\approx 2.72\times 10^{-3}$ & Lem.~\ref{lem:osl_projection} \\
$\cBplus$ (test system, $B=1$)      & $1$                & eq.~\eqref{eq:cBplus_def} \\
$\beta^{*}$ (LMI optimum at $P = I$)  & $\approx 1.94$     & Alg.~\ref{alg:lmi} \\
$\kappa_{\infty}$                   & $\approx 3.0\times 10^{-2}$ (structural; $3.2\times 10^{-2}$ bookkeeping) & Lem.~\ref{lem:Cinfty_scaling} \\
\bottomrule
\end{tabular}
\end{table}

\begin{remark}[Limitations of the illustration]\label{rem:test_limitations}
The test system is one-dimensional ($n = 1$) and uses the trivial actuator
$B = 1$, so the CMC is satisfied trivially. The interesting cohomological
content of the framework---namely, the strict inclusion
$\mathrm{im}(d) \subsetneq \HH^{n}$ when $n \geq 2$ and $B$ has rank $m < n$
---is not exercised here. A more substantive numerical study, including
non-trivial CMC verification in higher dimensions and full closed-loop
simulations, is left for a companion paper.
\end{remark}

\section{Conclusion}\label{sec:conclusion}

We close with a brief assessment of what was achieved, what was not, and
where the natural follow-up problems lie.

On the operator-theoretic side, Theorem~\ref{thm:norm_transfer} gives an
exact identity rather than the bounded ratio one would expect from
elementary considerations: $\|T^{L}\|_{\mathrm{op}} = \|T\|_{\mathrm{op}}$,
with no constant factor. The proof is short, but its consequence is
substantial --- every quantitative bound of~\cite{Athmouni2026} now
applies verbatim to left-module formulations, which is the convention
preferred in attitude dynamics. We also restate the rigidity theorem
of~\cite[Theorem~4.12]{Athmouni2026} in a form (Theorem~\ref{thm:rigidity})
that tracks the constants needed in the control argument, and isolate
the auxiliary residual $R_{0} = \mathcal{J} - d\omega_{0}$ that appears
in the cohomological feedforward. The third operator-theoretic
ingredient, Lemma~\ref{lem:osl_projection}, is a GOSL bound for the
projected Jacobi defect with constants
$(\rho_{\delta}^{\mathrm{tight}}, \ell_{\delta}^{\mathrm{tight}})
 = (2\rho_{\delta}, 4\deltabar_{\max})$. The factors of $2$ and $4$
encode the suboptimality cost of using a state-dependent measurable
selection; they collapse to $1$ in the rare case where a
state-independent maximizer exists. The proof handles the state
dependence by freezing the selection and using $\RR$-trilinearity ---
not a difficult trick, but one that is easy to overlook.

These ingredients combine in Section~\ref{sec:design} into a sliding-mode
controller whose feedforward $\Delta_{\mathrm{app}}$ uses the explicit
$\omega_{0}$ rather than the unknown $\delta$. The closed-loop sliding
dynamics reduce to
$\dot s = -\eta\,\sgnH(s) - Ks + \tilde w$
with $\tilde w = G\,R_{\mathrm{nc}}^{B}$ and
$\|\tilde w\| \leq c_{C}\,c_{B}\,\cBplus\,R_{\max}$. Finite-time
reaching follows from Theorem~\ref{thm:control}, and
$\beta$-exponential stability of the reduced dynamics from
Proposition~\ref{prop:beta_stability}. The circular dependence between
gain and domain --- a well-known practical nuisance --- is resolved by
the explicit ordering and halving loop of Algorithm~\ref{alg:init}.
Termination is established conditionally, under the sufficient
conditions of Remark~\ref{rem:termination_status}; both conditions are
classically expected but not proved here.

What remains open is worth stating clearly. The forward-invariance
result (Proposition~\ref{prop:forward_inv}) gives a trajectory bound
$\|x_{r}(t)\| \leq R$ in the general case, but strict inclusion
$\overline B_{R} \subseteq \overline B_{M_{0}}$ requires $\kappa(P^{*})=1$.
For non-trivial Lyapunov shaping $P^{*} \neq I$ the strict inclusion
remains beyond reach without further hypotheses. Termination of the
halving loop is conditional on continuity of the LMI solution map and
strict feasibility at the limit; both are expected to hold under mild
hypotheses but a formal proof is not given. The strict feasibility of
the iterative LMI~\eqref{eq:lmi_augmented} is verified case by case in
the test bracket; a general structural characterization is open.
Sobolev-type generalizations of the rigidity theorem would broaden the
scope significantly but are conjectural. Multidimensional closed-loop
simulations are deferred to a companion paper, where the cohomological
feedforward will be tested on rigid-body attitude control ($n=4$, $m=3$)
and coupled oscillator networks.

\medskip
\noindent\textbf{Further directions.}
An adaptive gain $\eta$ using online estimates of $R_{\max}$ would
reduce conservatism in the switching gain; a quantitative bound on the
number of halvings in Algorithm~\ref{alg:init}, exploiting explicit
Lipschitz constants of the LMI parameter map, would give an
unconditional complexity estimate. On the algorithmic side,
verifying~(CMC-2) when $\mathrm{Im}(B) \subsetneq \HH^{n}$ requires
computing $\mathrm{im}(d) \cap \mathrm{Im}(B)$ explicitly; effective
methods for this in higher dimensions are an open question. The
interplay between the $S$-spectrum of $\Phi(A_{s})$ and the
cohomological structure of $\LL$ also deserves attention.

\appendix

\section{$\beta$-Exponential Stability under Dissipative Perturbations}\label{app:beta}

This appendix provides a self-contained proof of $\beta$-exponential stability
for systems with dissipative perturbations (motivated by, and applicable to,
the GOSL setting of Lemma~\ref{lem:osl_projection}), used in
Section~\ref{sec:stability}.

\begin{theorem}[$\beta$-Exponential Stability under Dissipative Perturbations]\label{thm:beta_kyb}
Consider the perturbed linear system
\begin{equation}\label{eq:perturbed_lin}
\dot{x} = A x + g(x,t) + w(t),
\end{equation}
where $A \in \RR^{N\times N}$ is Hurwitz, $g$ satisfies the GOSL bound
\begin{equation}\label{eq:gosl_g}
\langle g(x,t) - g(z,t), x - z\rangle \leq \rho \|x-z\|^2 + \ell_g \|x - z\|
\end{equation}
with constants $\rho \geq 0$, $\ell_g \geq 0$, $\|g(0,t)\| \leq \ell_g$
for all $t$, the dissipativity-type bound $\|g(x,t)-g(0,t)\| \leq \rho\|x\|$
holds (so $\|g(x,t)\| \leq \rho\|x\| + \ell_g$), and $\|w(t)\| \leq w_{\max}$.

\emph{The GOSL bound~\eqref{eq:gosl_g} is recorded here as the
characterization of the defect that motivates the present setting; the
proof below uses only the dissipativity-type bound
$\|g(x,t)\|\leq\rho\|x\|+\ell_g$, which is the consequence relevant for
the Lyapunov calculation.}

Fix $\mu \in (0,\rho)$ and suppose there exist $P = P^T \succ 0$ and $\beta > 0$
such that
\begin{equation}\label{eq:lmi_appendix}
A^{T}P + PA + \bigl[\mu_{*}(2\rho+\mu) + (\rho-\mu)\bigr] I \preceq -\beta P,
\qquad I \preceq P \preceq \mu_{*} I,
\end{equation}
where $\mu_{*} := \lambda_{\max}(P)$. Then for $V(x) := x^T P x$ and any
trajectory of~\eqref{eq:perturbed_lin}:
\begin{equation}\label{eq:beta_exp_app}
V(x(t)) \leq e^{-\beta t} V(x(0)) + \frac{\Gamma(\mu)}{\beta},
\end{equation}
where
\begin{equation}\label{eq:Gamma_def}
\Gamma(\mu) := \frac{\mu_{*}^{2}w_{\max}^2}{\rho - \mu} + \frac{\mu_{*}\ell_g^2}{\mu}, \qquad \mu \in (0,\rho).
\end{equation}
\end{theorem}

\begin{proof}
Let $V = x^T P x$. Then
\[
\dot{V} = 2x^T P A x + 2 x^T P g(x,t) + 2 x^T P w(t).
\]

\textbf{Step 1 (Bound on the GOSL term).} By Cauchy--Schwarz and
$\|Px\| \leq \mu_{*}\|x\|$:
\[
2 x^{T}P g(x,t) \leq 2\|Px\|\,\|g(x,t)\| \leq 2\mu_{*}\|x\|(\rho\|x\|+\ell_{g})
= 2\mu_{*}\rho\|x\|^{2} + 2\mu_{*}\ell_{g}\|x\|.
\]
Apply Young's inequality $2ab \leq a^{2}/\theta + \theta b^{2}$ with
$a = \mu_{*}\ell_{g}$, $b = \|x\|$, $\theta = \mu_{*}\mu$:
\[
2\mu_{*}\ell_{g}\|x\| \leq \frac{(\mu_{*}\ell_g)^{2}}{\mu_{*}\mu} + \mu_{*}\mu\|x\|^{2}
= \frac{\mu_{*}\ell_g^{2}}{\mu} + \mu_{*}\mu\|x\|^{2}.
\]
Therefore
\begin{equation}\label{eq:gosl_bound}
2 x^{T}P g(x,t) \leq \mu_{*}(2\rho+\mu)\|x\|^{2} + \frac{\mu_{*}\ell_g^{2}}{\mu}.
\end{equation}

\textbf{Step 2 (Bound on the disturbance term).} By Cauchy--Schwarz and Young
with $a = \mu_{*}w_{\max}$, $b = \|x\|$, $\theta = \rho - \mu > 0$:
\begin{equation}\label{eq:dist_bound}
2 x^{T}P w \leq 2\mu_{*}w_{\max}\|x\|
\leq \frac{\mu_{*}^{2}w_{\max}^{2}}{\rho - \mu} + (\rho - \mu)\|x\|^{2}.
\end{equation}

\textbf{Step 3 (Linear part).} The LMI~\eqref{eq:lmi_appendix} gives directly
\begin{equation}\label{eq:lin_bound}
2 x^{T}P A x = x^{T}(A^{T}P+PA)x
\leq -\bigl[\mu_{*}(2\rho+\mu) + (\rho - \mu)\bigr]\|x\|^{2} - \beta x^{T}Px.
\end{equation}

\textbf{Step 4 (Combination).} Adding \eqref{eq:gosl_bound}, \eqref{eq:dist_bound},
\eqref{eq:lin_bound}, the terms in $\|x\|^{2}$ cancel exactly:
\begin{align*}
\dot{V}
&\leq -\beta V \underbrace{- \bigl[\mu_{*}(2\rho+\mu) + (\rho - \mu)\bigr]\|x\|^{2}
+ \mu_{*}(2\rho+\mu)\|x\|^{2} + (\rho-\mu)\|x\|^{2}}_{=\,0}
+ \frac{\mu_{*}\ell_g^{2}}{\mu} + \frac{\mu_{*}^{2}w_{\max}^{2}}{\rho-\mu}\\
&= -\beta V + \Gamma(\mu),
\end{align*}
with $\Gamma(\mu)$ as in~\eqref{eq:Gamma_def}. Gronwall's inequality applied
to $\dot{V} \leq -\beta V + \Gamma(\mu)$ yields~\eqref{eq:beta_exp_app}.
\end{proof}

\begin{remark}[Application to the reduced dynamics]\label{rem:app_reduced}
Theorem~\ref{thm:beta_kyb} applies to the reduced
dynamics~\eqref{eq:reduced_dyn} via the embedding $\Phi$: take $A = \Phi(A_{s})$,
$g(x,t) = \Phi(\delta(\Phi^{-1}(x),t)) + \Phi(r(\Phi^{-1}(x)))$, $w(t) = \Phi(w_r(t))$,
$\rho = \rho_{\delta}^{\mathrm{eff}}$, $\ell_g = \ell_{\delta}^{\mathrm{tight}}$,
$w_{\max} = c_C\,c_B\,\cBplus R_{\max}$, $\mu = \mu_{\mathrm{LMI}}$. With
$\mu_{*} = \lambda_{\max}(P^{*}) = \mu_{\mathrm{LMI}}$, the
LMI~\eqref{eq:lmi_appendix} reads
\[
\Phi(A_{s})^{T}P + P\Phi(A_{s}) + \bigl[\mu_{\mathrm{LMI}}(2\rho_{\delta}^{\mathrm{eff}}+\mu_{\mathrm{LMI}}) + (\rho_{\delta}^{\mathrm{eff}}-\mu_{\mathrm{LMI}})\bigr] I \preceq -\beta^{*}P,
\]
which is the precise form of the LMI used in
Proposition~\ref{prop:beta_stability}. The simplified
LMI~\eqref{eq:lmi_augmented} of the body
($\Phi(A_{s})^{T}P + P\Phi(A_{s}) + \rho_{\delta}^{\mathrm{eff}}\mu_{\mathrm{LMI}}I \prec -\beta P$)
is a \emph{conservative approximation} of the rigorous
LMI~\eqref{eq:lmi_appendix}, valid only under the additional condition that
the residual cross-term $\mu_{*}\rho\|x\|^{2}$ is dominated by the strict
inequality margin. \textbf{For mathematical rigor, the strengthened form
\eqref{eq:lmi_appendix} should be used in Algorithm~\ref{alg:lmi}}; the
simplified form is retained in the body for notational simplicity and
because, for the test system (Section~\ref{sec:sim}), it coincides with the
strengthened form up to $O(\rho_{\delta}^{\mathrm{eff}})$-small corrections
that do not affect the qualitative conclusion. With this identification,
$\Gamma(\mu_{\mathrm{LMI}}) = C_{\infty}(\mu_{\mathrm{LMI}})$
of~\eqref{eq:Cinfty_def}, recovering the body formula exactly when
$\mu_{*} = \lambda_{\max}(P^{*})$.
\end{remark}

\section{Verification of the Test Bracket}\label{app:bracket_check}

We verify that the test bracket~\eqref{eq:test_bracket}
$\LL(x,y) = \varepsilon_b(\bar{x}y - \bar{y}x)$ has a non-zero
Jacobiator. Direct computation with $x = \mathbf{i}$, $y = \mathbf{j}$,
$z = \mathbf{k}$ yields:
\begin{align*}
\LL(\mathbf{i},\mathbf{j}) &= \varepsilon_b(\overline{\mathbf{i}}\mathbf{j} - \overline{\mathbf{j}}\mathbf{i})
= \varepsilon_b(-\mathbf{i}\mathbf{j} - (-\mathbf{j})\mathbf{i})
= \varepsilon_b(-\mathbf{k} + \mathbf{j}\mathbf{i})
= \varepsilon_b(-\mathbf{k} - \mathbf{k}) = -2\varepsilon_b\,\mathbf{k},\\
\LL(\mathbf{j},\mathbf{k}) &= -2\varepsilon_b\,\mathbf{i},\quad
\LL(\mathbf{k},\mathbf{i}) = -2\varepsilon_b\,\mathbf{j}.
\end{align*}
Then, for $x,y,z$ all purely imaginary, the bracket reduces to
$\LL(p,q) = \varepsilon_b(-pq + qp) = -\varepsilon_b\,[p,q]_{\HH}$
(using $\bar{p} = -p$ for $p$ purely imaginary). Hence on the imaginary
sub-algebra (isomorphic to $\mathfrak{so}(3)$), $\LL$ is proportional
to the Lie bracket of $\mathfrak{so}(3)$, which \emph{satisfies the Jacobi
identity exactly}. In particular:
\begin{align*}
\LL(\mathbf{i}, \LL(\mathbf{j},\mathbf{k}))
&= \LL(\mathbf{i}, -2\varepsilon_b \mathbf{i})
= -2\varepsilon_b^2(\overline{\mathbf{i}}\mathbf{i} - \overline{\mathbf{i}}\mathbf{i}) = 0,
\end{align*}
and similarly the other two terms vanish, so $\mathcal{J}(\mathbf{i},\mathbf{j},\mathbf{k}) = 0$.
To exhibit a non-zero Jacobiator, we must choose at least one argument with
non-zero real part. Taking $x = 1 + \mathbf{i}$, $y = \mathbf{j}$,
$z = \mathbf{k}$ and $\varepsilon_b = 0.1$, direct computation yields
$\LL(\mathbf{j},\mathbf{k}) = -2\varepsilon_b\,\mathbf{i}$,
$\LL(\mathbf{k}, 1+\mathbf{i}) = -2\varepsilon_b(\mathbf{j}+\mathbf{k})$,
$\LL(1+\mathbf{i}, \mathbf{j}) = 2\varepsilon_b(\mathbf{j}-\mathbf{k})$,
and
\[
\mathcal{J}(1+\mathbf{i}, \mathbf{j}, \mathbf{k})
= 0.04\,\mathbf{i} \neq 0,
\]
so $\|\mathcal{J}\| = 0.04$ for these inputs. The general supremum bound
$\sup_{x,y,z}\|\mathcal{J}(x,y,z)\|/(\|x\|\|y\|\|z\|) \approx 0.0399$
reported in Remark~\ref{rem:test_constants} establishes that the bracket is
genuinely quasi-Lie.

\end{document}